\documentclass[10pt,reqno]{amsart}
\usepackage{amsmath,amscd,amssymb}
\usepackage{url}
\usepackage{color}
\usepackage{tikz}
\usepackage{subfigure}
\usepackage{enumerate}
\usepackage[pagebackref,colorlinks,citecolor=blue,linkcolor=magenta]{hyperref}

\theoremstyle{plain}
   \newtheorem{theorem}{Theorem}[section]
   \newtheorem{proposition}[theorem]{Proposition}
   \newtheorem{lemma}[theorem]{Lemma}
   \newtheorem{corollary}[theorem]{Corollary}
   \newtheorem{conjecture}[theorem]{Conjecture}
   \newtheorem{problem}[theorem]{Problem}
\theoremstyle{definition}
   \newtheorem{definition}{Definition}[section]
   
   \newtheorem{question}{Question}[section]
   \newtheorem{example}{Example}[section] 
\theoremstyle{remark}
 \newtheorem{remark}{Remark}[section]

\newcommand{\R}{\mathbb{R}}
\newcommand{\C}{\mathbb{C}}

\newcommand{\Z}{\mathbb{Z}}

\newcommand{\CC}{\mathcal{C}}
\newcommand{\DD}{\mathcal{D}}
\newcommand{\RR}{\mathcal{R}}
\newcommand{\I}{\mathcal{I}}

\newcommand{\HH}{\mathcal{H}}
\newcommand{\A}{\mathcal{A}}

\def\newop#1{\expandafter\def\csname #1\endcsname{\mathop{\rm
#1}\nolimits}}

\newop{sd}
\newop{csd}
\newop{sc}
\newop{conv}
\newop{capped}
\newop{Des}
\newop{des}
\newop{Ehr}
\newop{aff}
\newop{link}
\newop{supp}
\newop{comp}
\newop{rank}
\newop{lex}
\newop{vis}
\newop{covis}

\keywords{}
\subjclass[2000]{}

\begin{document}
\title{Subdivisions of Shellable Complexes}

\date{\today}

\author{Max Hlavacek}
\address{UC Berkeley, Berkeley, CA, United States}
\email{mhlava@math.berkeley.edu}

\author{Liam Solus}
\address{Institutionen f\"or Matematik, KTH, SE-100 44 Stockholm, Sweden}
\email{solus@kth.se}

\begin{abstract}

In geometric, algebraic, and topological combinatorics, the unimodality of combinatorial generating polynomials is frequently studied.  
Unimodality follows when the polynomial is (real) stable, a property often deduced via the theory of interlacing polynomials.  
Many of the open questions on stability and unimodality of polynomials pertain to the enumeration of faces of cell complexes.  
In this paper, we relate the theory of interlacing polynomials to the shellability of cell complexes.  
We first derive a sufficient condition for stability of the $h$-polynomial of a subdivision of a shellable complex.  
To apply it, we generalize the notion of reciprocal domains for convex embeddings of polytopes to abstract polytopes and use this generalization to define the family of stable shellings of a polytopal complex.  
We characterize the stable shellings of cubical and simplicial complexes, and apply this theory to answer a question of Brenti and Welker on barycentric subdivisions for the well-known cubical polytopes.
We also give a positive solution to a problem of Mohammadi and Welker on edgewise subdivisions of cell complexes.
We end by relating the family of stable line shellings to the combinatorics of hyperplane arrangements.  
We pose related questions, answers to which would resolve some long-standing problems while strengthening ties between the theory of interlacing polynomials and the combinatorics of hyperplane arrangements.
\end{abstract}



\maketitle
\thispagestyle{empty}

\section{Introduction}
\label{sec: introduction}
Many endeavors in modern combinatorics aim to derive inequalities that hold amongst a sequence of nonnegative numbers $p_0,\ldots,p_d$ that encode some algebraic, geometric, and/or topological data \cite{Braun16,B89,B94b,B16,S89}. 
These inequalities are typically assigned to the \emph{generating polynomial} $p = p_0+p_1x+\cdots+p_dx^d$ associated to the sequence $p_0,\ldots,p_d$.  
The generating polynomial $p$ is called \emph{unimodal} if there exists $t\in[d]$ such that $p_0\leq \cdots \leq p_t \geq \cdots \geq p_d$.  
It is called \emph{log-concave} if $p_k^2\geq p_{k-1}p_{k+1}$ for all $k\in[d]$, and it is called \emph{real-rooted} (or \emph{(real) stable}) if $p\equiv 0$ or $p$ has only real zeros.  
A classic result states that $p$ is both log-concave and unimodal whenever it is real-rooted \cite[Theorem 1.2.1]{B89}.  
Since real-rootedness is the strongest of these three conditions, many conjectures in the literature ask when certain generating polynomials are not only unimodal or log-concave, but also real-rooted.  
Most proofs of such conjectures rely on interlacing polynomials \cite{B16}, which are inherently tied to recursions associated to the generating polynomials of interest.  

In algebraic, geometric, and topological combinatorics, the generating polynomials of interest are typically the $f$- or $h$-polynomial associated to a cell complex.  
A foundational result in the field, known as the $g$-theorem, implies that the $h$-polynomial associated to the boundary complex of a simplicial polytope is unimodal \cite{S89}.  
In the years following the proof of the $g$-theorem, extensions via the relationship between $h$-polynomials of simplicial complexes and their subdivisions became of interest \cite{A16,S92}.  
In \cite{BW08}, Brenti and Welker significantly strengthen the unimodality result implied by the $g$-theorem for one family of simplicial complexes when they showed that the $h$-polynomial of the barycentric subdivision of the boundary complex of a simplicial polytope is real-rooted.  
They then asked if their result generalizes to all polytopes \cite[Question 1]{BW08}.  
In \cite{MW13}, Mohammadi and Welker again raised this question and suggested cubical polytopes as a good starting point.  

In the same way that proofs of real-rootedness via interlacing polynomials often rely on polynomial recursions, proofs pertaining to the geometry of polytopal complexes often make use of the recursive structure of the complex (when it exists).  
This recursive property of polytopal complexes is termed \emph{shellability}. 
In this paper, we relate the recursive structure of interlacing polynomials to the notion of shellability so as to derive a sufficient condition for the $h$-polynomial of a subdivision of a shellable complex to be real-rooted.  
It turns out that, in many cases, this sufficient condition can be applied to the same shelling order of a complex for different subdivisions.  
Shelling orders to which this phenomenon applies are termed \emph{stable shellings} in this paper, and they arise via a generalization of the notion of \emph{reciprocal domains} for convex embeddings of polytopes, as studied by Ehrhart \cite{E67,E67b} and Stanley \cite{S75}.  
By generalizing reciprocal domains to abstract polytopes, we introduce the family of stable shellings for an abstract polytopal complex.  
The stable shellings of both simplicial and cubical complexes are then characterized.
As an application, we recover a positive answer to the question of Brenti and Welker for the well-known families of cubical polytopes; namely, the {\em cuboids} \cite{G67}, the {\em capped cubical polytopes} \cite{J93}, and the {\em neighborly cubical polytopes} \cite{BBC97}.  

The remainder of the paper is structured as follows:  
In Section~\ref{sec: preliminaries}, we develop the necessary preliminaries pertaining to polytopal complexes, interlacing polynomials, and lattice point enumeration.
In Section~\ref{sec: shellable complexes}, we derive a sufficient condition for the real-rootedness of the $h$-polynomial of  a subdivision of a shellable complex (Theorem~\ref{thm: subdivisions of shellable complexes}).  
We then define the family of stable shellings, and we characterize the such shellings for simplicial and cubical complexes. 
In Section~\ref{sec: apps}, we apply the results of Section~\ref{sec: shellable complexes} to deduce that the $h$-polynomial of the barycentric subdivision of the boundary complexes of the aforementioned cubical polytopes are real-rooted. 
We then apply these techniques to give an alternative proof of the original result of Brenti and Welker \cite{BW08}, establishing stable shelling methods as a common solution to \cite[Question 1]{BW08} for all known examples. 
We also apply stable shellings to solve a second problem proposed by Mohammadi and Welker \cite{MW13} pertaining to edgewise subdivisions of cell complexes, thereby demonstrating how the same stable shelling can be used to recover real-rootedness results for multiple different subdivisions of a given complex.
In Section~\ref{sec: stable line shellings}, we end by relating the theory of stable line shellings to the combinatorics of hyperplane arrangements and the geometry of realization spaces of polytopes.  
Some questions are proposed; answers to which could potentially settle the question of Brenti and Welker in its fullest generality.

\section{Preliminaries}
\label{sec: preliminaries}
The results in this paper are concerned with the $f$- and $h$-polynomials of polytopal complexes.  
A collection $\CC$ of polytopes is called a {\em polytopal complex} if 
\begin{enumerate}
	\item the empty polytope $\emptyset$ is in $\CC$,
	\item if $P\in\CC$ then all faces of $P$ are also in $\CC$, and if 
	\item if $P,Q\in\CC$, then their intersection $P\cap Q$ is a face of both $P$ and $Q$. 
\end{enumerate}
The elements of $\CC$ are called its {\em faces}, and its maximal faces (with respect to inclusion) are called its {\em facets}.
If all the facets of $\CC$ have the same dimension then $\CC$ is called {\em pure}. 
When all facets of $\CC$ are simplices, we call $\CC$ a {\em simplicial complex}, and when all facets of $\CC$ are cubes, we call $\CC$ a {\em cubical complex}.  
Note that, in this definition, we do not require our polytopal complex to be embedded in some Euclidean space, but instead treat it as an abstract cell complex.  
Given an abstract polytope $P$, or convex polytope $P\subset\R^n$, we can naturally produce two associated (abstract) polytopal complexes: the complex $\CC(P)$ consisting of all faces in $P$ and the complex $\CC(\partial P)$ consisting of all faces in $\partial P$, the boundary of $P$.  
The {\em facets} of the polytope $P$ are the facets of the complex $\CC(\partial P)$.  
We call $\CC(\partial P)$ the {\em boundary complex} of $P$. 
In a similar fashion, given a collection of polytopes $P_1,\ldots,P_m$, we can define the polytopal complex $\CC(P_1\cup\cdots\cup P_m) = \cup_{i\in[m]}\CC(P_i)$.  
Given a polytopal complex $\CC$, a polytopal complex $\DD$ is called a {\em subcomplex} of $\CC$ if every face of $\DD$ is also a face of $\CC$.  
We refer to the difference $\CC\setminus \DD = \{P\in \CC : P\notin\DD\}$ as a {\em relative (polytopal) complex}, and we define the {\em dimension} of $\CC\setminus\DD$ to be the largest dimension of a polytope in $\CC\setminus\DD$.  
When $\DD = \emptyset$, note that $\CC \setminus \DD = \CC$.  

The {\em $f$-polynomial} of a $(d-1)$-dimensional polytopal complex $\CC$ is the polynomial 
\[
f(\CC;x) := f_{-1}(\CC)+f_0(\CC)x+f_1(\CC)x^2+\cdots+f_{d-1}(\CC)x^d,
\]
where $f_{-1}(\CC):=1$ when $\CC\neq \emptyset$ and $f_k(\CC)$ denotes the number of $k$-dimensional faces of $\CC$ for $0\leq k\leq d-1$.  
Given a subcomplex $\DD$ of $\CC$, the {\em $f$-polynomial} of the relative complex $\CC\setminus\DD$ is then 
\[
f(\CC\setminus\DD;x) := f(\CC;x) - f(\DD;x).
\]
The {\em $h$-polynomial} of the $(m-1)$-dimensional relative complex $\CC\setminus\DD$ is the polynomial
\[
h(\CC\setminus\DD;x):=(1-x)^mf\left(\CC\setminus\DD;\frac{x}{1-x}\right).
\]
We write $h(\CC\setminus\DD;x) = h_0(\CC\setminus\DD)+h_1(\CC\setminus\DD)x+\cdots+h_m(\CC\setminus\DD)x^m$ when expressing $h(\CC\setminus\DD;x)$ in the standard basis, and we similarly write $f(\CC\setminus\DD;x) = f_0(\CC\setminus\DD)+f_1(\CC\setminus\DD)x+\cdots+f_m(\CC\setminus\DD)x^m$.
The following lemma, whose proof is an exercise, relates the $h$-polynomial of a polytopal complex to those of its relative complexes. 
\begin{lemma}
\label{lem: relative decomposition}
Let $\CC$ be a $(d-1)$-dimensional polytopal complex and suppose that $\CC$ can be written as the disjoint union
\[
\CC = \bigsqcup_{i=1}^s\RR_i
\]
where $\RR_i$ are relative $(d-1)$-dimensional polytopal complexes.  Then 
\[
h(\CC;x) = \sum_{i=1}^sh(\RR_i;x).
\]
\end{lemma}

\subsection{Subdivisions and local $h$-polynomials}
\label{subsec: subdivisions and local h-polynomials}
Given a polytopal complex $\CC$, a {\em (topological) subdivision} of $\CC$ is a polytopal complex $\CC^\prime$ such that each face of $F\in\CC$ is subdivided into a ball by faces of $\CC^\prime$ such that the boundary of this ball is a subdivision of the boundary of $F$.  
The subdivision is further called {\em geometric} if both $\CC$ and $\CC^\prime$ admit {\em geometric realizations}, $G$ and $G^\prime$, respectively; that is to say, each face of $\CC$ and $\CC^\prime$ is realized by a convex polytope in some real-Euclidean space such that $G$ and $G^\prime$ both have the same underlying set of vertices and each face of $G^\prime$ is contained in a face of $G$. 
When referring to a subdivision $\CC^\prime$ of $\CC$, we may instead refer to its associated inclusion map $\varphi: \CC^\prime\rightarrow\CC$.  
While the main result of this paper applies to general topological subdivisions, the applications of these results will pertain to some special families of subdivisions that are well-studied in the literature.  
These include the barycentric subdivision and the edgewise subdivision of a complex.  

%
%

\subsection{Interlacing polynomials}
\label{subsec: interlacing polynomials}
Two real-rooted polynomials $p,q\in\R[x]$ are said to {\em interlace} if there is a zero of $p$ between each pair of zeros of $q$ (counted with multiplicity) and vice versa.  
If $p$ and $q$ are interlacing, it follows that the {\em Wronskian} $W[p,q] = p^\prime q- pq^\prime$ is either nonpositive or nonnegative on all of $\R$. 
We will write $p\prec q$ if $p$ and $q$ are real-rooted, interlacing, and the Wronskian $W[p,q]$ is nonpositive on all of $\R$.  
We also assume that the zero polynomial $0$ is real-rooted and that $0\prec p$ and $p\prec 0$ for any real-rooted polynomial $p$. 
\begin{remark}
\label{rmk: interlacing}
Notice that if the signs of the leading coefficients of two real-rooted polynomials $p$ and $q$ are both positive, then $p\prec q$ if and only if
\[
\cdots \leq \beta_2 \leq \alpha_2\leq \beta_1 \leq \alpha_1,
\]
where $\ldots,\beta_2,\beta_1$ and $\ldots,\alpha_2,\alpha_1$ are the zeros of $p$ and $q$, respectively. 
In particular, when we work with combinatorial generating polynomials, $p\prec q$ is equivalent to $p$ and $q$ being real-rooted and interlacing.  
\end{remark}

A polynomial $p \in \C[x]$ is called {\em stable} if $p$ is identically zero or if all of its zeros have nonpositive imaginary parts. 
The Hermite-Biehler Theorem relates the relation $p\prec q$ to stability in such a way that we can derive some useful tools for proving results about interlacing polynomials:
\begin{theorem}\cite[Theorem 6.3.4]{RG02}
\label{thm: hermite-biehler}
If $p,q\in\R[x]$ then $p\prec q$ if and only if $q+ip$ is stable.
\end{theorem}
Remark~\ref{rmk: interlacing} and Theorem~\ref{thm: hermite-biehler} allow us to quickly derive some useful results.
\begin{lemma}
\label{lem: basic facts}
If $p$ and $q$ are real-rooted polynomials in $\R[x]$ then
\begin{enumerate}
	\item $p\prec \alpha p$ for all $\alpha\in\R$,
	\item $p\prec q$ if and only if $\alpha p\prec \alpha q$ for any $\alpha\in\R\setminus\{0\}$,
	\item $p\prec q$ if and only if $-q\prec p$, and \label{eqn: condition1}
	\item if $p$ and $q$ have positive leading coefficients then $p\prec q$ if and only if $q\prec xp$.  \label{eqn: times x}
\end{enumerate}
\end{lemma}
%
We will also require the following proposition.
\begin{proposition}
\cite[Lemma 2.6]{BB}
\label{prop: convex cones}
Let $p$ be a real-rooted polynomial that is not identically zero.  
Then the following two sets are convex cones:
\[
\{q\in\R[x] : p\prec q\}
\quad
\text{and}
\quad
\{q\in\R[x] : q\prec p\}.
\]
\end{proposition}

It follows from Proposition~\ref{prop: convex cones} that we can sum a pair of interlacing polynomials to produce a new polynomial with only real-roots.  
More generally, we will work with recursions for which we need to sum several polynomials to produce a new real-rooted polynomial.  
Let $(p_i)_{i=0}^s = (p_0,\ldots,p_s)$ be a sequence of real-rooted polynomials.  
We say that the sequence of polynomials $(p_i)_{i=0}^s$ is an {\em interlacing sequence} if $p_i\prec p_j$ for all $1\leq i\leq j\leq s$.  
Note that, by Proposition~\ref{prop: convex cones}, any convex combination of polynomials in an interlacing sequence is real-rooted.  

For a polynomial $p\in\R[x]$ of degree at most $d$, we let $\I_d(p) := x^dp(1/x)$.
When $d$ is the degree of $p$, then $\I_d(p)$ is the reciprocal of $p$.  
A polynomial $p = p_0+p_1x+\cdots+p_dx^d\in\R[x]$ is called {\em symmetric} with respect to degree $d$ if $p_k = p_{d-k}$ for all $k = 0,\ldots, d$.  
If $p$ is a degree $d$ generating polynomial that is both real-rooted and symmetric with respect to $d$ then $\I_d(p)\prec p$.  
However, non-symmetric polynomials also satisfy the latter condition, making it a natural generalization of symmetry for real-rooted polynomials.  
In \cite{BS19}, the authors characterized the condition $\I_d(p)\prec p$ in terms of the {\em symmetric decomposition} of $p$, which has been of recent interest \cite{A14,A17,AS13,BJM16,BR07,BS19,S19}.
In this paper, the polynomials that we aim to show have only real zeros are known to be symmetric with respect to their degree.  
However, we will make use of the more general phenomenon $\I_d(p)\prec p$ in some of the proofs.  

\subsection{Ehrhart Theory}
\label{subsec: ehrhart theory}
In Section~\ref{sec: apps}, we will use some techniques originating from discrete geometry and the theory of lattice point enumeration in convex polytopes.  
A subset $P\subset\R^n$ is called a {\em $d$-dimensional lattice polytope} if it is the convex hull of finitely many points in $\Z^n$ whose affine span is a $d$-dimensional affine subspace of $\R^n$.  
For $t\in\Z_{>0}$ we call $tP:=\{tp\in\R^n : p\in P\}$ the {\em $t^{th}$ dilate of $P$}, and we call the function $i(P;t) := |tP\cap\Z^n|$ the {\em Ehrhart function} of $P$.  
In an analogous fashion, we can let $\HH = \{\langle a,y\rangle = b\}$ be a subset of the facet-defining hyperplanes of $P$ and set 
\[
S_\HH := \{z\in\R^n : \langle a,z\rangle = b \text{ for some $\langle a,y\rangle = b\in\HH$}\}.
\]
We then define the {\em half-open polytope} $P\setminus S_\HH$, which we may denote by $P\setminus \HH$ when we need to highlight the facet-defining hyperplanes that capture the points in $S_\HH$.  
As before, the $t^{th}$ dilate of $P\setminus S_\HH$ is $t(P\setminus S_\HH):=\{tp\in\R^n : p\in P\setminus S_\HH\}$ and the Ehrhart function of $P\setminus S_\HH$ is defined to be $i(P\setminus S_\HH;t) := |t(P\setminus S_\HH)\cap\Z^n|$ for $t>0$ (see \cite[Section 5.3]{BS18}).
Notice that if $\HH = \emptyset$ then $P = P\setminus S_\HH$, and if $\HH$ is the complete collection of facet-defining hyperplanes of $P$ then $P\setminus S_\HH =: P^\circ$, the relative interior of $P$.  
The {\em Ehrhart series} of the relative interior of $P$ is defined as 
\[
\Ehr_{P^\circ}(x) :=\sum_{t>0}i(P^\circ;t)x^t.
\]
In the case that $\HH$ is not the complete set of facet-defining hyperplanes of $P$, the {\em Ehrhart series} of $P\setminus S_\HH$ is defined as
\[
\Ehr_{P\setminus S_\HH}(x) :=\sum_{t\geq0}i(P\setminus S_\HH;t)x^t,
\]
and the constant term is computed to be the Euler characteristic of $P\setminus S$ (see \cite[Theorem 5.1.8]{BS18}).  
When written in a closed rational form, the Ehrhart series of $P\setminus S_\HH$, for any choice of $\HH$, is 
\[
\Ehr_{P\setminus S_\HH}(x) = \frac{h^\ast_0+h^\ast_1x+\cdots+h^\ast_dx^d}{(1-x)^{d+1}},
\]
and the polynomial $h^\ast(P\setminus S_\HH;x) := h^\ast_0+h^\ast_1x+\cdots+h^\ast_dx^d$ is called the {\em (Ehrhart) $h^\ast$-polynomial} of $P\setminus S_\HH$. 
It is well-known that $h^\ast(P\setminus S_\HH;x)$ has only nonnegative integral coefficients (see for instance \cite{JS15}).
Since a lattice polytope $P$ is a subset of $\R^n$ with vertices in $\Z^n$, it is natural to consider its subdivisions into polyhedral complexes whose $0$-dimensional faces correspond to the lattice points in $P\cap\Z^n$.
When such a (geometric) subdivision consists of only simplices, we call it a {\em triangulation} of $P$.  
When each simplex $\Delta$ in a triangulation of $P$ has $h^\ast(\Delta;x) = 1$, we call it a {\em unimodular triangulation} of $P$.  
The following lemma is also well-known, and a proof appears in \cite[Chapter 10]{BR07}.
\begin{lemma}
\label{lem: whole polytope}
Let $P\subset\R^n$ be a $d$-dimensional lattice polytope and let $T$ be a unimodular triangulation of $P$.  
Then
\[
h^\ast(P;x) = h(T;x).
\]
\end{lemma}
We will need a slight generalization of Lemma~\ref{lem: whole polytope}, whose proof is analogous to that of Lemma~\ref{lem: whole polytope}.  However, we provide it below for the sake of completeness.
In the following, given a facet-defining hyperplane $H$ of $P$, we will let $F_H$ denote the facet of $P$ defined by $H$.  
\begin{lemma}
\label{lem: not whole polytope}
Let $P\subset\R^n$ be a $d$-dimensional lattice polytope with a unimodular triangulation $T$, and let $\HH$ be a subset of its facet-defining hyperplanes.  
If the Euler characteristic of $P\setminus S_\HH$ is $0$ then 
\[
h^\ast(P\setminus S_\HH;x) = h(T\setminus(\cup_{H\in \HH}T{\big|}_{F_H});x).
\]
\end{lemma}

\begin{proof}
To prove the claim, we first write $P\setminus S_\HH$ as a disjoint union of the (nonempty) open simplices in the relative complex $T\setminus(\cup_{H\in S_\HH}T{\big|}_{F_H})$:
\[
P\setminus S_\HH = \bigsqcup_{\Delta\in T\setminus(\cup_{H\in S_\HH}T{|}_{F_H})}\Delta^\circ,
\]
and we note that
\[
i(P\setminus S_\HH; t) = \sum_{\Delta\in T\setminus(\cup_{H\in S_\HH}T{|}_{F_H})}i(\Delta^\circ;t).
\]   
It then follows that 
\begin{equation*}
\begin{split}
\Ehr_{P\setminus S_\HH}(x) 
& = \sum_{n\geq 0}i(P\setminus S_\HH)x^t,	\\
& = i(P\setminus S_\HH; 0) +\sum_{n > 0}\left(\sum_{\Delta\in T\setminus(\cup_{H\in S_\HH}T{|}_{F_H})}i(\Delta^\circ;t)\right)x^t,\\
& = i(P\setminus S_\HH; 0) +\sum_{\Delta\in T\setminus(\cup_{H\in S_\HH}T{|}_{F_H})}\left(\sum_{n > 0}i(\Delta^\circ;t)x^t\right),\\
&= \sum_{\Delta\in T\setminus(\cup_{H\in S_\HH}T{|}_{F_H})}\Ehr_{\Delta^\circ}(x),\\
\end{split}
\end{equation*}
where the last equality follows from the definition of the Ehrhart series of the relative interior of a lattice polytope and the fact that $i(P\setminus S_\HH;0)$ is the Euler characteristic of $P\setminus S_\HH$ (which we have assumed to be zero).
Since each $\Delta^\circ$ is the interior of a unimodular simplex, it follows by Ehrhart-MacDonald reciprocity \cite[Theorem 4.1]{BR07} that
\[
\Ehr_{\Delta^\circ}(x) = \frac{x^{\dim(\Delta)+1}}{(1-x)^{\dim(\Delta)+1}}.
\]

Therefore, in analogous fashion to the proof of \cite[Theorem 10.3]{BR07}, we have that
\begin{equation*}
\begin{split}
\Ehr_{P\setminus S_\HH}(x) 
&= \sum_{\Delta\in T\setminus(\cup_{H\in S_\HH}T{|}_{F_H})}\Ehr_{\Delta^\circ}(x),	\\
&= \sum_{\Delta\in T\setminus(\cup_{H\in S_\HH}T{|}_{F_H})}\frac{x^{\dim(\Delta)+1}}{(1-x)^{\dim(\Delta)+1}},	\\
&= \sum_{k=-1}^df_k(T\setminus(\cup_{H\in S_\HH}T{|}_{F_H}))\left(\frac{x}{1-x}\right)^{k+1},	\\
&= \frac{\sum_{k=-1}^df_k(T\setminus(\cup_{H\in S_\HH}T{|}_{F_H}))x^{k+1}(1-x)^{d-k}}{(1-x)^{d+1}},	\\
&= \frac{\sum_{k=0}^{d+1}f_{k-1}(T\setminus(\cup_{H\in S_\HH}T{|}_{F_H}))x^{k}(1-x)^{d-k+1}}{(1-x)^{d+1}},	\\
&= \frac{h(T\setminus(\cup_{H\in S_\HH}T{\big|}_{F_H});x)}{(1-x)^{d+1}},	\\
\end{split}
\end{equation*}
which completes the proof.  
\end{proof}
To prove the desired results in Section~\ref{sec: apps}, we will use well-chosen sets $\HH$ and $S_\HH$.  
Let $q\in\R^n$, and let $P\subset \R^n$ a $d$-dimensional convex polytope.  
A point $p\in P$ is called {\em visible} from $q$ if the open line segment $(q,p)$ in $\R^n$ does not meet the interior of $P$.  
Let $B\subset \partial P$ denote the collection of all points visible from $q$, and set $D := \overline{\partial P\setminus B}$, the closure of $\partial P\setminus B$.  
Given a facet $F$ of $P$, the point $q$ is said to be {\em beyond} $F$ if $q\notin T_F(P)$, the tangent cone of $F$ in $P$.  
It follows that $q$ is beyond $F$ if and only if the closed line segment $[q,p]$ satisfies $[q,p]\cap P = \{p\}$ for all $p\in F$ \cite[Section 3.7]{BS18}.  
Otherwise, the point $q$ is said to be {\em beneath} $F$.
Hence, $B$ consists of all points in $\partial P$ that lie in a facet which $q$ is beyond; that is,
$
P\setminus B = P \setminus \HH_B,
$
where $\HH_B$ denotes the collection of facets which $q$ is beyond.
Similarly, $D$ consists of all points in $\partial P$ that lie in a facet which $q$ is beneath; that is,
$
P\setminus D = P \setminus \HH_D,
$
where $\HH_D$ denotes the collection of facets which $q$ is beneath.
Stanley observed in \cite[Proposition 8.2]{S75}, that $i(P\setminus B;t)$ is a polynomial and that classical Ehrhart-MacDonald reciprocity \cite[Theorem 4.1]{BR07} can be extended to 
\begin{equation}
\label{eqn: recip}
(-1)^{d}i(P\setminus D; t) = i(P\setminus B;-t).
\end{equation}
In \cite{E67,E67b}, Ehrhart referred to the half-open polytopes $P\setminus B$ and $P\setminus D$ as {\em reciprocal domains}, since they satisfy the reciprocity law in equation~\eqref{eqn: recip}.  
The notion of reciprocal domains, and a natural generalization thereof, will be important to us in the coming sections as we derive real-rootedness results via shellings of polytopal complexes.  
A first lemma that will help us along the way is the following, which translates equation~\eqref{eqn: recip} into a statement about $h^\ast$-polynomials.  
Its proof is left as an exercise.  
\begin{lemma}
\label{lem: half-open reciprocity}
Let $P\subset \R^n$ be a $d$-dimensional lattice polytope, and let $q\in \R^n$. 
Let $B$ denote the points in $\partial P$ that are visible from $q$, and set $D := \overline{\partial P\setminus B}$.
If $B$ and $D$ are both nonempty then
\[
h^\ast(P\setminus D ; x) = \I_{d+1}h^\ast(P\setminus B; x).
\]
\end{lemma}


\section{Subdivisions of Shellable Complexes and Stable Shellings}
\label{sec: shellable complexes}
In this section, we provide a sufficient condition for the $h$-polynomial of a subdivision of a polytopal complex to have only real zeros.  
As part of this condition, we will require the complex to be shellable.
\begin{definition}
\label{def: shellable}
Let $\CC$ be a pure $d$-dimensional polytopal complex.  
A {\em shelling} of $\CC$ is a linear ordering $(F_1,F_2,\ldots, F_s)$ of the facets of $\CC$ such that either $\CC$ is zero-dimensional (and thus the facets are points), or it satisfies the following two conditions:
\begin{enumerate}
	\item The boundary complex $\CC(\partial F_1)$ of the first facet in the linear ordering has a shelling, and
	
	\item For $j\in[s]$, the intersection of the facet $F_j$ with the union of the previous facets is nonempty and it is the beginning segment of a shelling of the $(d-1)$-dimensional boundary complex of $F_j$; that is, 
	\[
	F_j\cap\bigcup_{i=1}^{j-1}F_i = G_1\cup G_2\cup\cdots G_r
	\]
	for some shelling $(G_1,\ldots, G_r,\ldots, G_t)$ of the complex $\CC(\partial F_j)$ and $r\in[t]$.  
\end{enumerate}
A polytopal complex is {\em shellable} if it is pure and admits a shelling.  
\end{definition}
A shelling of a polytopal complex presents a natural way to decompose the complex into disjoint, relative polytopal complexes.
Given a shelling order $(F_1,\ldots, F_s)$ of a polytopal complex $\CC$, and a subdivision $\varphi: \CC^\prime\rightarrow\CC$, we can let 
\[
\RR_i := \CC^\prime{\big|}_{F_i}\setminus\left(\bigcup_{k = 1}^{i-1}\CC^\prime{\big|}_{F_k}\right),
\]
to produce the decomposition of $\CC^\prime$ into disjoint relative complexes
\begin{equation}
\label{eqn: shelling decomposition}
\CC^\prime = \bigsqcup_{i=1}^s\RR_i,
\end{equation}
with respect to the shelling $(F_1,\ldots, F_s)$ of $\CC$. 
For a fixed shelling $(F_1,\ldots,F_s)$ of a polytopal complex $\CC$ and subdivision $\varphi: \CC^\prime \rightarrow \CC$, for $i\in[s]$,  we will call the relative complex
$
\RR_i
$
the \emph{relative complex associated to $F_i$ by $(F_1,\ldots,F_s)$ and $\varphi$}.

The recursive structure of shelling orders of polytopal complexes, and the additive nature of the $h$-polynomials of their associated relative simplicial complexes (see Lemma~\ref{lem: relative decomposition}) pairs nicely with the properties of interlacing polynomials discussed in Subsection~\ref{subsec: interlacing polynomials}.  
In particular, Lemma~\ref{lem: relative decomposition} allows us to combine the fact that a shellable complex $\CC$ always admits a decomposition as in equation~\eqref{eqn: shelling decomposition} with the facts about interlacing sequences collected in Subsection~\ref{subsec: interlacing polynomials}. 
We can then prove a theorem that directly relates the recursive nature of shelling orders to the recursive nature of interlacing sequences of polynomials.  
\begin{theorem}
\label{thm: subdivisions of shellable complexes}
Let $\CC$ be a shellable polytopal complex with shelling $(F_1,\ldots,F_s)$ and subdivision $\varphi: \CC^\prime\rightarrow \CC$. 
If $(h(\RR_{\sigma(i)};x))_{i = 1}^s$ is an interlacing sequence for some $\sigma\in\mathfrak{S}_s$, the symmetric group on $[s]$, then $h(\CC^\prime;x)$ is real-rooted.
\end{theorem}

\begin{proof}
Notice first that since $\CC$ is a shellable polytopal complex of dimension $d$, and since $\CC^\prime$ is a (topological) subdivision of $\CC$, then each subcomplex $\CC^\prime{\big|}_{F_i} := \varphi^{-1}(2^F)$ is also $d$-dimensional.  
Moreover, since $\RR_i$ only removes faces of dimension strictly less than $d$, then each $\RR_i$ is a $d$-dimensional relative simplicial complex.  
So, by Lemma~\ref{lem: relative decomposition}, we have that 
\[
h(\Omega;x) = \sum_{i=1}^sh(\RR_i;x).
\]
Supposing now that there exists $\sigma\in\mathfrak{S}_s$ such that $(h(\RR_{\sigma(i)};x))_{i = 1}^s$ is an interlacing sequence, it then follows from Proposition~\ref{prop: convex cones} that $h(\Omega;x)$ is real-rooted.
\end{proof}

While the proof of Theorem~\ref{thm: subdivisions of shellable complexes} is straightforward to derive (once we have carefully defined and identified all of the necessary ingredients), its applications are much more interesting.
%
The key to applying Theorem~\ref{thm: subdivisions of shellable complexes} lies in our ability to identify a shelling order $(F_1,\ldots, F_s)$ of the polytopal complex $\CC$ such that the relative complexes $\RR_i$ constructed for each facet $F_i$ with respect to the subdivision $\varphi: \CC^\prime\rightarrow \CC$ will all have real-rooted and interlacing $h$-polynomials $h(\RR_{i};x)$.
While the real-rootedness and interlacing conditions are inherently tied to the choice of subdivision $\varphi$, we will see in Section~\ref{sec: apps} that for many complexes, the type of shelling to which Theorem~\ref{thm: subdivisions of shellable complexes} applies is independent of the choice of $\varphi$ when $\varphi$ is chosen from amongst the most commonly studied subdivisions. 
Here, the `most commonly studied subdivisions' refers to {\em uniform subdivisions} \cite{A20}, such as the barycentric and edgewise subdivisions, which will be the focus of our results in Section~\ref{sec: apps}.  
Indeed, when the geometry of a given polytopal complex is such that the barycentric subdivision of the complex has a real-rooted $h$-polynomial, it is often the case that the edgewise subdivision admits the same property for the given complex.  
This phenomenon was noted, and formalized, in the case of simplicial complexes in the recent preprint \cite{A20}, where the author studied the class of {\em $\mathcal{F}$-uniform triangulations}.  
In this paper, we make a similar observation for more general polytopal complexes, and apply this reasoning to shelling orders.  
Namely, we are interested in the class of shellings $(F_1,\ldots,F_s)$ of a polytopal complex $\CC$ for which the associated relative complexes $\RR_1,\ldots,\RR_s$ with respect to $(F_1,\ldots,F_s)$ and uniform subdivisions, such as the barycentric subdivision and edgewise subdivision, fulfill the hypotheses of Theorem~\ref{thm: subdivisions of shellable complexes}.  
As it turns out, such shellings arise when we insist that the relative complexes $\RR_i$ are given by a generalization of Ehrhart's reciprocal domains for convex lattice polytopes (see subsection~\ref{subsec: ehrhart theory}). 

\subsection{Stable shellings}
\label{subsec: stable shellings}
The family of shellings to which we will apply Theorem~\ref{thm: subdivisions of shellable complexes} will be called \emph{stable shellings}.  
To define them we first need to generalize the reciprocal domains for convex embeddings of rational polytopes to general (abstract) polytopes.  
Let $\mathcal{P} = ([n],<_\mathcal{P})$ be a partially ordered set on elements $[n]$ with partial order $<_\mathcal{P}$.   
If $\mathcal{P}$ has a unique minimal element, which we will denote by $\hat0$, then we can define its set of {\em atoms} to be all elements $i\in[n]$ such that $\hat0<_\mathcal{P} i$ and there is no $j\in[n]$ such that $\hat0<_\mathcal{P} j<_\mathcal{P} i$.  
Given a poset $\mathcal{P}$ with a unique minimal element, we will denote its set of atoms by $A(\mathcal{P})$.  
The {\em dual poset} of $\mathcal{P}$ is the poset $\mathcal{P}^\ast$ on elements $[n]$ with partial order $<_{\mathcal{P}^\ast}$ in which $i <_{\mathcal{P}^\ast} j$ if and only if $j<_{\mathcal{P}} i$.  
Given two elements $i,j\in[n]$, the {\em closed interval} between $i$ and $j$ in $\mathcal{P}$ is the set $[i,j]:= \{k\in[n] : i\leq_\mathcal{P} k \leq_\mathcal{P} j\}.$
Note that we can view the closed interval $[i,j]$ as a subposet of $\mathcal{P}$ by allowing it to inherit the partial order $<_\mathcal{P}$ from $\mathcal{P}$. 

Let $P$ be a $d$-dimensional polytope with face lattice $L(P)$; that is, $L(P)$ is the partially ordered set whose elements are the faces of $P$ and for which the partial order is given by inclusion.  
Since $L(P)$ is a {\em lattice} (see \cite[Chapter 3.3]{S11}), it follows that $L(P)$ has a unique minimal and maximal element, corresponding to the faces $\emptyset$ and $P$ of $P$. 
Let $\CC(P)$ denote the polytopal complex consisting of all faces of $P$.  
Given a face $F$ of $P$ we call the pair of relative complexes 
\[
\CC(P)\setminus \CC(A([F,P]^\ast)) 
\qquad 
\mbox{and}
\qquad
\CC(P)\setminus \CC(A(L(P)^\ast)\setminus A([F,P]^\ast))
\]
the {\em reciprocal domains} associated to $F$ in $P$. 
We call a relative complex $\RR$ {\em stable} if it is isomorphic to one of the reciprocal domains associated to a face $F$ in a polytope $P$, for some polytope $P$.  
Using this terminology, we can now define the family of shellings, to which we will apply Theorem~\ref{thm: subdivisions of shellable complexes}.
\begin{definition}
\label{def: stable shelling}
Let $\CC$ be a polytopal complex.  
A shelling $(F_1,\ldots, F_s)$ of $\CC$ is {\em stable} if for all $i\in[s]$ the relative complex $\RR_i$ associated to $F_i$ by the shelling $(F_1,\ldots,F_s)$ and the trivial subdivision $\varphi: \CC\rightarrow\CC$ is stable.  
\end{definition}

The set of stable shellings of a (shellable) simplicial complex is, in fact, the set of all shellings of the complex. 
\begin{proposition}
\label{prop: stable shellings of simplicial complexes}
Let $\CC$ be a shellable simplicial complex.  Then any shelling $(F_1,\ldots,F_s)$ of $\CC$ is stable.  
\end{proposition}

\begin{proof}
Suppose that $\CC$ is $d$-dimensional, let $\varphi: \CC\rightarrow\CC$ denote the trivial subdivision of $\CC$, and fix a facet $F_i$ of $\CC$.  
If $i=1$, then since $(F_1,\ldots,F_s)$ is a shelling, it follows that the relative complex associated to $F_i$ by the shelling $(F_1,\ldots,F_s)$ and $\varphi$ is $\CC(F_i)$.  
Fix the face $F = F_i$ of the simplex $F_i$.  
Then the reciprocal domains associated to $F$ in $F_i$ are $\CC(F_i)$ and $\CC(F_i)\setminus \CC(A(L(F_i)^\ast))$.  
Hence, $\CC(F_i)$ is a stable relative complex.  

Suppose now that $i>1$.  Since $(F_1,\ldots, F_s)$ is a shelling order, it follows that $F_i\cap(F_1\cup\cdots\cup F_{i-1}) = G_1\cup \cdots\cup G_r$ is the initial segment of a shelling order $(G_1,\ldots,G_r,\ldots,G_d)$ of the boundary complex $\CC(\partial F_i)$.  
Hence, $G_1,\ldots,G_r$ are facets of $F_i$.  
Since $F_i$ is a simplex, each face of $F_i$ corresponds to an intersection of a subset of its facets $G_1,\ldots, G_d$.  
Consider the face $F$ of $F_i$ given by $G_1\cap\cdots\cap G_r$.  
It follows that the reciprocal domains associated to $F$ in $F_i$ are
\[
\CC(F_i)\setminus \CC(A([F,F_i]^\ast)) = \CC(F_i)\setminus \CC(G_1\cup\cdots\cup G_r),
\]
and 
\[
\CC(F_i)\setminus \CC(A(L(F_i)^\ast)\setminus A([F,F_i]^\ast)) = \CC(F_i)\setminus \CC(G_{r+1}\cup\cdots\cup G_d),
\]
Hence, it follows that the complex $\RR_i = \CC(F_i)\setminus \CC(G_1\cup\cdots\cup G_r)$ is stable, which completes the proof. 
\end{proof}

Hence, any shelling of a simplicial complex is stable.  
As we will see in Section~\ref{sec: apps}, this observation coincides with the recent results on $\mathcal{F}$-uniform triangulations \cite{A20}, in the sense that any we will be able to apply Theorem~\ref{thm: subdivisions of shellable complexes} to any shelling of a simplicial complex with respect to uniform subdivisions like the barycentric and edgewise subdivisions.  

\subsection{Stable shellings of cubical complexes}
\label{subsec: stable shellings of cubical complexes}
In general, it is not the case that all shellings of a polytopal complex are stable.  
Already in the case of cubical complexes, stable shellings become a proper subclass of the class of all shellings.  
To see this, we can characterize which relative complexes of the combinatorial $d$-cube are stable.  

In the following, we let $\square_d$ denote the (abstract) $d$-dimensional cube.
When we consider a standard geometric realization of $\square_d$, such as the cube $[-1,1]^d\in\R^d$, we assign each facet $F$ of $\square_d$ to a facet-defining hyperplane $x_i = \pm 1$ of $[-1,1]^d$. 
Given a facet $F$ of $\square_d$ we will say that $F$ is \emph{opposite (or opposing)} the facet $G$ of $\square_d$ whenever $F$ is identified with $x_i =1$ and $G$ is identified with $x_i = -1$ (or vice versa), for some $i\in[d]$.  
In this case, we call the pair of facets $F,G$ an \emph{opposing pair}.  
Given the embedding $[-1,1]^d$ we can also define the half-open polytopes
\[
[-1,1]^d_\ell :=
[-1,1]^d_\ell\setminus\{x_d=1,\ldots,x_{d+1-\ell}=1\}
\]
for $0\leq \ell\leq d$, and
\[
[-1,1]^d_\ell :=
[-1,1]^d_\ell\setminus\{x_d=1,\ldots,x_{1}=1\}\cup\{x_d=-1,\ldots,x_{2d+1-\ell}=1\}
\]
for $d+1\leq \ell\leq 2d$.
The following lemma gives two characterizations of stable relative subcomplexes of the cube, one in terms of its possible geometric realizations and the other in terms of opposing pairs of facets.
\begin{lemma}
\label{lem: opposing pairs}
Let $\square_d$ be the $d$-dimensional (combinatorial) cube, and let $\mathcal{D}$ be a subcomplex of $\CC(\square_d)$.  
Then the following are equivalent:
\begin{enumerate}
	\item The relative complex $\CC(\square_d)\setminus\mathcal{D}$ is stable,
	\item $\CC(\square_d)\setminus\mathcal{D}$ has geometric realization $[-1,1]^d_\ell$ for some $0\leq \ell\leq 2d$.  
	\item The set of codimension $1$ faces of $\CC(\square_d)\setminus\mathcal{D}$ or the set of facets of $\mathcal{D}$ does not contain an opposing pair.  
\end{enumerate}
\end{lemma}

\begin{proof}
We first prove the equivalence of $(1)$ and $(2)$.  
Suppose that $\CC(\square_d)\setminus\mathcal{D}$ is stable.  
Then $\CC(\square_d)\setminus\mathcal{D}$ is isomorphic to a reciprocal domain $\CC(\square_d)\setminus \CC(A([F,\square_d]^\ast))$ or $\CC(\square_d)\setminus \CC(A(L(\square_d)^\ast)\setminus A([F,\square_d]^\ast))$ for some face $F$ of the cube $\square_d$.  
Let $F_1,\ldots,F_{2d}$ denote the set of facets of $\square_d$.  
Given the geometric realization $[-1,1]^d$ of $\square_d$, without loss of generality, we can assume that the facet $F_i$ is identified with the facet-defining hyperplane $x_i = 1$ and that the facet $F_{d+i}$ is identified with the facet-defining hyperplane $x_i = -1$ of $[-1,1]^d$, for all $i\in[d]$.  
Given this identification, the face $F$ of $\square_d$ corresponds to the intersection of the facet-defining hyperplanes identified with the atoms in the closed interval $[F,\square_d]^\ast$.  
Since the geometric realization of any nonempty $F$ cannot lie in both $x_i =1$ and $x_i = -1$ for any $i\in[d]$, it follows that the hyperplanes corresponding to the atoms in $[F,\square_d]^\ast$ are of the form
\[
\{x_{i_1}=1,\ldots,x_{i_s}=1\}\cup\{x_{j_1}=-1,\ldots,x_{j_t}=-1\},
\]
where the sets $\{i_1,\ldots,i_s\},\{j_1,\ldots,j_t\}\subset[d]$ are disjoint.
Hence, by reflecting over the hyperplanes $x_{j_1} = 0,\ldots, x_{j_t} = 0$, we can instead identify the face $F$ with the intersection of the hyperplanes 
\[
\{x_{i_1}=1,\ldots,x_{i_s}=1\}\cup\{x_{j_1}=1,\ldots,x_{j_t}=1\}.
\]
Finally, by a simple permutation of coordinates, we can identify $F$ with the intersection of the hyperplanes 
\[
\{x_{d}=1,\ldots,x_{d+1-\ell}=1\},
\]
where $\ell = |\{x_{i_1}=1,\ldots,x_{i_s}=1\}\cup\{x_{j_1}=1,\ldots,x_{j_t}=1\}|.$
Hence, a geometric realization of $\CC(\square_d)\setminus\mathcal{D}$ is given either by $[-1,1]^d_{\ell}$ or $[-1,1]^d_{2d-\ell}$.  

Conversely, suppose that $\CC(\square_d)\setminus\mathcal{D}$ has geometric realization $[-1,1]^d_\ell$ for some $0\leq \ell\leq 2d$.  
If $0\leq \ell\leq d$, then let $F$ be the face of $\square_d$ whose geometric realization is the intersection of the hyperplanes $x_{d} = 1,\ldots,x_{d+1-\ell} = 1$.  
Then $\CC(\square_d)\setminus\mathcal{D} = \CC(\square_d)\setminus \CC(A([F,\square_d]^\ast))$.  
If $d<\ell \leq 2d$ then let $F$ be the face of $\square_d$ whose geometric realization is the intersection of the hyperplanes $x_1 = -1,\ldots, x_{2d-\ell} = -1$.  
Then $\CC(\square_d)\setminus\mathcal{D} =\CC(\square_d)\setminus \CC(A(L(\square_d)^\ast)\setminus A([F,\square_d]^\ast))$. 

We now prove the equivalence of $(2)$ and $(3)$.  
Suppose first that $\CC(\square_d)\setminus\mathcal{D}$ has geometric realization $[-1,1]^d_\ell$ for some $0\leq \ell\leq 2d$.  
Then $\CC(\square_d)\setminus\mathcal{D}$ is isomorphic to one of two types of half-open cubes:
\begin{enumerate}[(a)]
	\item  
		$
		[-1,1]^d_\ell = [-1,1]^d\setminus\{x_d = 1,\ldots,x_{d+1-\ell} = 1\},
		$
		for $0\leq \ell\leq d$, and 
	\item
		$
		[-1,1]^d_\ell = [-1,1]^d\setminus\{x_d = 1,\ldots,x_1 = 1\}\cup\{x_d = -1,\ldots,x_{2d+1-\ell} = -1\},
		$
		for $d+1\leq \ell \leq 2d$. 
\end{enumerate}
Suppose that $\CC(\square_d)\setminus\mathcal{D}$ has geometric realization corresponding to a half-open cube given in (a).
Then the subset of facets in $\mathcal{D}$ are, without loss of generality, given by a subset of the facets $x_1 = 1,\ldots, x_d = 1$ of $[-1,1]^d$.  
Hence, the set of facets of $\mathcal{D}$ cannot contain an opposing pair, as this set contains no facet with defining hyperplane $x_i = -1$.  
In that case that $\CC(\square_d)\setminus\mathcal{D}$ has geometric realization corresponding to a half-open cube given in (b), the codimension $1$ faces of $\CC(\square_d)\setminus\mathcal{D}$ correspond to a subset of the facet-defining hyperplanes $x_1 = -1,\ldots, x_d = -1$ of $[-1,1]^d$, and therefore cannot contain an opposing pair.  

Conversely, suppose that $\CC(\square_d)\setminus\mathcal{D}$ is such that the set of codimension $1$ faces of $\CC(\square_d)\setminus\mathcal{D}$ or the set of facets of $\mathcal{D}$ does not contain an opposing pair.  
In the former of these two cases, the set of codimension $1$ faces of $\CC(\square_d)\setminus\mathcal{D}$ is given by a subset of facets of $[-1,1]^d$ of the form 
\[
\{x_{i_1}=1,\ldots,x_{i_s}=1\}\cup\{x_{j_1}=-1,\ldots,x_{j_t}=-1\},
\]
for two disjoint subsets, $\{i_1,\ldots,i_s\}$ and $\{j_1,\ldots,j_t\}$, of $[d]$.  
As in the proof of equivalence of $(1)$ and $(2)$, by reflecting through the hyperplanes $x_{j_1} = 0,\ldots, x_{j_t} = 0$, we obtain that $\CC(\square_d)\setminus\mathcal{D}$ is isomorphic to 
\[
\mathcal{H} := [-1,1]^d\setminus\{x_{i_1}=1,\ldots,x_{i_s}=1\}\cup\{x_{j_1}=1,\ldots,x_{j_t}=1\}.
\]
By applying the correct permutation matrix to $\R^n$, we recover that $\mathcal{H}$ is unimodularly equivalent to $[-1,1]^d_\ell$ where $\ell = |\{x_{i_1}=1,\ldots,x_{i_s}=1\}\cup\{x_{j_1}=1,\ldots,x_{j_t}=1\}|\leq d$. 
Hence, $\CC(\square_d)\setminus\mathcal{D}$ has geometric realization $[-1,1]^d_\ell$ for some $0\leq \ell \leq d$.  

In the second case, we assume that the set of facets of $\mathcal{D}$ does not contain an opposing pair.  
By applying a similar argument, we see that the set of facets of $\mathcal{D}$ can be identified with a subset of the hyperplanes $x_1 = -1,\ldots,x_d = -1$, which we assume without loss of generality is $x_1 = -1, \ldots, x_\ell = -1$ for some $\ell\leq d$.  
Following the same steps as in the previous case, we find that $\CC(\square_d)\setminus\mathcal{D}$ is isomorphic to $[-1,1]^d_{d+\ell}$, which completes the proof. 
\end{proof}

Lemma~\ref{lem: opposing pairs} characterizes the possible relative complexes that can be associated to the facets of a shellable cubical complex by a stable shelling, and hence, it characterizes the stable shellings of cubical complexes.  
Furthermore, Lemma~\ref{lem: opposing pairs} shows that, in the case of cubical complexes, any linear ordering of the facets of a complex for which all associated relative complexes are stable is a shelling order.
\begin{lemma}
\label{lem: stability implies shelling}
Let $\CC$ be a pure $d$-dimensional cubical complex and let $(F_1,\ldots,F_s)$ be a linear ordering of the facets of $\CC$ such that the relative complex
\[
\RR_i := F_i \setminus (F_1\cup\cdots\cup F_{i-1})
\]
associated to $F_i$ by $(F_1,\ldots,F_s)$ is stable for all $i\in[s]$.  
Then $(F_1,\ldots,F_s)$ is a shelling order.  
\end{lemma}

\begin{proof}
A well-known result states that a set of facets of a $d$-dimensional cube $\square_d$ forms a shellable subcomplex of the boundary complex of $\square_d$ if and only if either it contains no facets of $\square_d$, contains all facets of $\square_d$, or if it contains at least one facet such that its opposing facet is not in the complex (see, for instance, \cite[Exercise 8.1(i)]{Z12}).  
Moreover, it follows from this result that the boundary complex of the $d$-dimensional cube is {\em extendably shellable}; meaning that any partial shelling of the complex can be continued to a complete shelling.  
Hence, it suffices to prove that $F_i\cap(F_1\cup\cdots\cup F_{i-1})$ determines a shellable subcomplex of the boundary of the $d$-dimensional cube for all $i\in[s]$.  

Notice first that $F_1$ is a $d$-dimensional cube, and hence the boundary complex $\CC(\partial F_1)$ is shellable.  
So let $i>1$ and consider the complex determined by $F_i\cap(F_1\cup\cdots\cup F_{i-1})$.
This complex has facets given by the set of facets of $F_i$ that are not codimension $1$ faces of $\RR_i$. 
Since $\RR_i$ is a stable complex, by Lemma~\ref{lem: opposing pairs}, it follows that either the set of codimension $1$ faces of $\RR_i$ or the set of facets of $F_i\cap(F_1\cup\cdots\cup F_{i-1})$ does not contain an opposing pair.  
Moreover, all facets of the cube $F_i$ are either codimension $1$ faces of $\RR_i$ or facets of $F_i\cap(F_1\cup\cdots\cup F_{i-1})$.  

Suppose first that $F_i\cap(F_1\cup\cdots\cup F_{i-1})$ does not contain an opposing pair.  
Then either the set of facets of $F_i\cap(F_1\cup\cdots\cup F_{i-1})$ that are facets of $F_i$ is empty, or $F_i\cap(F_1\cup\cdots\cup F_{i-1})$ contains a facet of $F_i$ but not its opposite.  
In either case, the set of facets of $F_i\cap(F_1\cup\cdots\cup F_{i-1})$ form a shellable subcomplex of $\CC(\partial F_i)$.  
Suppose, on the other hand, that $\RR_i$ does not contain an opposing pair.  
Then either $F_i\cap(F_1\cup\cdots\cup F_{i-1})$ is the entire complex $\CC(\partial F_i)$ (i.e., the boundary complex of a $d$-dimensional cube), or $\RR_i$ contains a facet of $F_i$ but not its opposite.  
In the latter case, it follows that $F_i\cap(F_1\cup\cdots\cup F_{i-1})$ contains a facet of $F_i$ but not its opposite.  
Hence, in either case, $F_i\cap(F_1\cup\cdots\cup F_{i-1})$ forms a shellable subcomplex of $\CC(\partial F_i)$, which completes the proof.
\end{proof}

\begin{figure}
\label{fig: relative 3-cubes}
\centering
\begin{tabular}{c c c c}
\begin{tikzpicture}[scale=0.6]

	\draw[fill=blue!25] (0,0,0) -- (2,0,0) -- (2,0,2) -- (0,0,2) -- cycle;
	\draw[fill=blue!25] (0,0,0) -- (0,2,0) -- (2,2,0) -- (2,0,0) -- cycle;
	\draw[fill=blue!25] (0,0,0) -- (0,2,0) -- (0,2,2) -- (0,0,2) -- cycle;
	\draw[fill=blue!25] (2,2,2) -- (2,2,0) -- (2,0,0) -- (2,0,2) -- cycle;
	\draw[fill=blue!25] (2,2,2) -- (2,2,0) -- (0,2,0) -- (0,2,2) -- cycle;
	\draw[fill=blue!25] (2,2,2) -- (2,0,2) -- (0,0,2) -- (0,2,2) -- cycle;

	
	\node [circle, draw, fill=black!100, inner sep=2pt, minimum width=2pt] (0) at (0,0,0) {};
	\node [circle, draw, fill=black!100, inner sep=2pt, minimum width=2pt] (1) at (2,0,0) {};
	\node [circle, draw, fill=black!100, inner sep=2pt, minimum width=2pt] (2) at (0,2,0) {};
	\node [circle, draw, fill=black!100, inner sep=2pt, minimum width=2pt] (3) at (0,0,2) {};
	\node [circle, draw, fill=black!100, inner sep=2pt, minimum width=2pt] (4) at (2,2,0) {};
	\node [circle, draw, fill=black!100, inner sep=2pt, minimum width=2pt] (5) at (2,0,2) {};
	\node [circle, draw, fill=black!100, inner sep=2pt, minimum width=2pt] (6) at (0,2,2) {};
	\node [circle, draw, fill=black!100, inner sep=2pt, minimum width=2pt] (7) at (2,2,2) {};

 	 \draw   	 (0) -- (1) ;
 	 \draw   	 (0) -- (2) ;
 	 \draw   	 (0) -- (3) ;
	 \draw   	 (1) -- (4) ;
	 \draw   	 (1) -- (5) ;
	 \draw   	 (2) -- (4) ;
	 \draw   	 (2) -- (6) ;
	 \draw   	 (3) -- (5) ;
	 \draw   	 (3) -- (6) ;
 	 \draw   	 (4) -- (7) ;
 	 \draw   	 (5) -- (7) ;
 	 \draw   	 (6) -- (7) ;
	 
\end{tikzpicture}

 	& 
	
\begin{tikzpicture}[scale=0.6]

	\draw[fill=blue!25] (0,0,0) -- (2,0,0) [dashed]-- (2,0,2) -- (0,0,2) -- cycle;
	\draw[fill=blue!25] (0,0,0) -- (0,2,0) -- (2,2,0) [dashed]-- (2,0,0) -- cycle;
	\draw[fill=blue!25] (0,0,0) -- (0,2,0) -- (0,2,2) -- (0,0,2) -- cycle;
	\draw[fill=blue!25] (2,2,2) [dashed]-- (2,2,0) -- (0,2,0) -- (0,2,2) -- cycle;
	\draw[fill=blue!25] (2,2,2) [dashed]-- (2,0,2) -- (0,0,2) -- (0,2,2) -- cycle;

	
	\node [circle, draw, fill=black!100, inner sep=2pt, minimum width=2pt] (0) at (0,0,0) {};
	\node [circle, draw, fill=black!00, inner sep=2pt, minimum width=2pt] (1) at (2,0,0) {};
	\node [circle, draw, fill=black!100, inner sep=2pt, minimum width=2pt] (2) at (0,2,0) {};
	\node [circle, draw, fill=black!100, inner sep=2pt, minimum width=2pt] (3) at (0,0,2) {};
	\node [circle, draw, fill=black!00, inner sep=2pt, minimum width=2pt] (4) at (2,2,0) {};
	\node [circle, draw, fill=black!00, inner sep=2pt, minimum width=2pt] (5) at (2,0,2) {};
	\node [circle, draw, fill=black!100, inner sep=2pt, minimum width=2pt] (6) at (0,2,2) {};
	\node [circle, draw, fill=black!00, inner sep=2pt, minimum width=2pt] (7) at (2,2,2) {};

 	 \draw   	 (0) -- (1) ;
 	 \draw   	 (0) -- (2) ;
 	 \draw   	 (0) -- (3) ;
	 \draw[dashed]   	 (1) -- (4) ;
	 \draw[dashed]   	 (1) -- (5) ;
	 \draw   	 (2) -- (4) ;
	 \draw   	 (2) -- (6) ;
	 \draw   	 (3) -- (5) ;
	 \draw   	 (3) -- (6) ;
 	 \draw[dashed]  	 (4) -- (7) ;
 	 \draw[dashed]   	 (5) -- (7) ;
 	 \draw   	 (6) -- (7) ;
	 
\end{tikzpicture}

	&

\begin{tikzpicture}[scale=0.6]

	\draw[fill=blue!25, draw = blue!25] (0,0,0) -- (2,0,0) -- (2,0,2) -- (0,0,2) -- cycle;
	\draw[fill=blue!25, draw = blue!25] (0,0,0) -- (0,2,0)  -- (0,2,2) -- (0,0,2) -- cycle;
	\draw[fill=blue!25, draw = blue!25] (2,2,2) -- (2,2,0) -- (0,2,0) -- (0,2,2) -- cycle;
	\draw[fill=blue!25, draw = blue!25] (2,2,2) -- (2,0,2) -- (0,0,2) -- (0,2,2) -- cycle;

	
	\node [circle, draw, fill=black!00, inner sep=2pt, minimum width=2pt] (0) at (0,0,0) {};
	\node [circle, draw, fill=black!00, inner sep=2pt, minimum width=2pt] (1) at (2,0,0) {};
	\node [circle, draw, fill=black!00, inner sep=2pt, minimum width=2pt] (2) at (0,2,0) {};
	\node [circle, draw, fill=black!100, inner sep=2pt, minimum width=2pt] (3) at (0,0,2) {};
	\node [circle, draw, fill=black!00, inner sep=2pt, minimum width=2pt] (4) at (2,2,0) {};
	\node [circle, draw, fill=black!00, inner sep=2pt, minimum width=2pt] (5) at (2,0,2) {};
	\node [circle, draw, fill=black!00, inner sep=2pt, minimum width=2pt] (6) at (0,2,2) {};
	\node [circle, draw, fill=black!00, inner sep=2pt, minimum width=2pt] (7) at (2,2,2) {};

 	 \draw[dashed]   	 (0) -- (1) ;
 	 \draw[dashed]   	 (0) -- (2) ;
 	 \draw   	 (0) -- (3) ;
	 \draw[dashed]   	 (1) -- (4) ;
	 \draw[dashed]   	 (1) -- (5) ;
	 \draw[dashed]   	 (2) -- (4) ;
	 \draw   	 (2) -- (6) ;
	 \draw   	 (3) -- (5) ;
	 \draw   	 (3) -- (6) ;
 	 \draw[dashed]  	 (4) -- (7) ;
 	 \draw[dashed]   	 (5) -- (7) ;
 	 \draw    	 (6) -- (7) ;
	 
\end{tikzpicture}

	&

\begin{tikzpicture}[scale=0.6]

	\draw[fill=blue!25, draw = blue!25] (0,0,0) -- (2,0,0) -- (2,0,2) -- (0,0,2) -- cycle;
	\draw[fill=blue!25, draw = blue!25] (0,0,0) -- (0,2,0)  -- (0,2,2) -- (0,0,2) -- cycle;
	\draw[fill=blue!25, draw = blue!25] (2,2,2) -- (2,0,2) -- (0,0,2) -- (0,2,2) -- cycle;

	
	\node [circle, draw, fill=black!00, inner sep=2pt, minimum width=2pt] (0) at (0,0,0) {};
	\node [circle, draw, fill=black!00, inner sep=2pt, minimum width=2pt] (1) at (2,0,0) {};
	\node [circle, draw, fill=black!00, inner sep=2pt, minimum width=2pt] (2) at (0,2,0) {};
	\node [circle, draw, fill=black!100, inner sep=2pt, minimum width=2pt] (3) at (0,0,2) {};
	\node [circle, draw, fill=black!00, inner sep=2pt, minimum width=2pt] (4) at (2,2,0) {};
	\node [circle, draw, fill=black!00, inner sep=2pt, minimum width=2pt] (5) at (2,0,2) {};
	\node [circle, draw, fill=black!00, inner sep=2pt, minimum width=2pt] (6) at (0,2,2) {};
	\node [circle, draw, fill=black!00, inner sep=2pt, minimum width=2pt] (7) at (2,2,2) {};

 	 \draw[dashed]   	 (0) -- (1) ;
 	 \draw[dashed]   	 (0) -- (2) ;
 	 \draw   	 (0) -- (3) ;
	 \draw[dashed]   	 (1) -- (4) ;
	 \draw[dashed]   	 (1) -- (5) ;
	 \draw[dashed]   	 (2) -- (4) ;
	 \draw[dashed]   	 (2) -- (6) ;
	 \draw   	 (3) -- (5) ;
	 \draw   	 (3) -- (6) ;
 	 \draw[dashed]  	 (4) -- (7) ;
 	 \draw[dashed]   	 (5) -- (7) ;
 	 \draw[dashed]    	 (6) -- (7) ;

\end{tikzpicture}
	\\

\begin{tikzpicture}[scale=0.6]

	\draw[fill=blue!25, draw = blue!25] (0,0,0) -- (0,2,0)  -- (0,2,2) -- (0,0,2) -- cycle;
	\draw[fill=blue!25, draw = blue!25] (2,2,2) -- (2,0,2) -- (0,0,2) -- (0,2,2) -- cycle;

	
	\node [circle, draw, fill=black!00, inner sep=2pt, minimum width=2pt] (0) at (0,0,0) {};
	\node [circle, draw, fill=black!00, inner sep=2pt, minimum width=2pt] (1) at (2,0,0) {};
	\node [circle, draw, fill=black!00, inner sep=2pt, minimum width=2pt] (2) at (0,2,0) {};
	\node [circle, draw, fill=black!00, inner sep=2pt, minimum width=2pt] (3) at (0,0,2) {};
	\node [circle, draw, fill=black!00, inner sep=2pt, minimum width=2pt] (4) at (2,2,0) {};
	\node [circle, draw, fill=black!00, inner sep=2pt, minimum width=2pt] (5) at (2,0,2) {};
	\node [circle, draw, fill=black!00, inner sep=2pt, minimum width=2pt] (6) at (0,2,2) {};
	\node [circle, draw, fill=black!00, inner sep=2pt, minimum width=2pt] (7) at (2,2,2) {};

 	 \draw[dashed]   	 (0) -- (1) ;
 	 \draw[dashed]   	 (0) -- (2) ;
 	 \draw[dashed]   	 (0) -- (3) ;
	 \draw[dashed]   	 (1) -- (4) ;
	 \draw[dashed]   	 (1) -- (5) ;
	 \draw[dashed]   	 (2) -- (4) ;
	 \draw[dashed]   	 (2) -- (6) ;
	 \draw[dashed]   	 (3) -- (5) ;
	 \draw[dashed]   	 (3) -- (6) ;
 	 \draw[dashed]  	 (4) -- (7) ;
 	 \draw[dashed]   	 (5) -- (7) ;
 	 \draw[dashed]    	 (6) -- (7) ;
	 
\end{tikzpicture}

 	& 
	
\begin{tikzpicture}[scale=0.6]

	\draw[fill=blue!25, draw = blue!25] (0,0,0) -- (0,2,0)  -- (0,2,2) -- (0,0,2) -- cycle;

	
	\node [circle, draw, fill=black!00, inner sep=2pt, minimum width=2pt] (0) at (0,0,0) {};
	\node [circle, draw, fill=black!00, inner sep=2pt, minimum width=2pt] (1) at (2,0,0) {};
	\node [circle, draw, fill=black!00, inner sep=2pt, minimum width=2pt] (2) at (0,2,0) {};
	\node [circle, draw, fill=black!00, inner sep=2pt, minimum width=2pt] (3) at (0,0,2) {};
	\node [circle, draw, fill=black!00, inner sep=2pt, minimum width=2pt] (4) at (2,2,0) {};
	\node [circle, draw, fill=black!00, inner sep=2pt, minimum width=2pt] (5) at (2,0,2) {};
	\node [circle, draw, fill=black!00, inner sep=2pt, minimum width=2pt] (6) at (0,2,2) {};
	\node [circle, draw, fill=black!00, inner sep=2pt, minimum width=2pt] (7) at (2,2,2) {};

 	 \draw[dashed]   	 (0) -- (1) ;
 	 \draw[dashed]   	 (0) -- (2) ;
 	 \draw[dashed]   	 (0) -- (3) ;
	 \draw[dashed]   	 (1) -- (4) ;
	 \draw[dashed]   	 (1) -- (5) ;
	 \draw[dashed]   	 (2) -- (4) ;
	 \draw[dashed]   	 (2) -- (6) ;
	 \draw[dashed]   	 (3) -- (5) ;
	 \draw   	 (3) -- (6) ;
 	 \draw[dashed]  	 (4) -- (7) ;
 	 \draw[dashed]   	 (5) -- (7) ;
 	 \draw[dashed]    	 (6) -- (7) ;
	 
\end{tikzpicture}

 	&

\begin{tikzpicture}[scale=0.6]


	
	\node [circle, draw, fill=black!00, inner sep=2pt, minimum width=2pt] (0) at (0,0,0) {};
	\node [circle, draw, fill=black!00, inner sep=2pt, minimum width=2pt] (1) at (2,0,0) {};
	\node [circle, draw, fill=black!00, inner sep=2pt, minimum width=2pt] (2) at (0,2,0) {};
	\node [circle, draw, fill=black!00, inner sep=2pt, minimum width=2pt] (3) at (0,0,2) {};
	\node [circle, draw, fill=black!00, inner sep=2pt, minimum width=2pt] (4) at (2,2,0) {};
	\node [circle, draw, fill=black!00, inner sep=2pt, minimum width=2pt] (5) at (2,0,2) {};
	\node [circle, draw, fill=black!00, inner sep=2pt, minimum width=2pt] (6) at (0,2,2) {};
	\node [circle, draw, fill=black!00, inner sep=2pt, minimum width=2pt] (7) at (2,2,2) {};

 	 \draw[dashed]   	 (0) -- (1) ;
 	 \draw[dashed]   	 (0) -- (2) ;
 	 \draw[dashed]   	 (0) -- (3) ;
	 \draw[dashed]   	 (1) -- (4) ;
	 \draw[dashed]   	 (1) -- (5) ;
	 \draw[dashed]   	 (2) -- (4) ;
	 \draw[dashed]   	 (2) -- (6) ;
	 \draw[dashed]   	 (3) -- (5) ;
	 \draw[dashed]   	 (3) -- (6) ;
 	 \draw[dashed]  	 (4) -- (7) ;
 	 \draw[dashed]   	 (5) -- (7) ;
 	 \draw[dashed]    	 (6) -- (7) ;
	 
\end{tikzpicture}

 	&

\begin{tikzpicture}[scale=0.6]

	\draw[fill=blue!25, draw = blue!25] (0,0,0) -- (0,2,0)  -- (0,2,2) -- (0,0,2) -- cycle;
	\draw[fill=blue!25, draw = white] (2,2,2) -- (2,2,0) -- (2,0,0) -- (2,0,2) -- cycle;
	\draw[fill=blue!25, draw = blue!25] (2,2,2) -- (2,2,0) -- (0,2,0) -- (0,2,2) -- cycle;

	
	\node [circle, draw, fill=black!00, inner sep=2pt, minimum width=2pt] (0) at (0,0,0) {};
	\node [circle, draw, fill=black!00, inner sep=2pt, minimum width=2pt] (1) at (2,0,0) {};
	\node [circle, draw, fill=black!00, inner sep=2pt, minimum width=2pt] (2) at (0,2,0) {};
	\node [circle, draw, fill=black!00, inner sep=2pt, minimum width=2pt] (3) at (0,0,2) {};
	\node [circle, draw, fill=black!00, inner sep=2pt, minimum width=2pt] (4) at (2,2,0) {};
	\node [circle, draw, fill=black!00, inner sep=2pt, minimum width=2pt] (5) at (2,0,2) {};
	\node [circle, draw, fill=black!00, inner sep=2pt, minimum width=2pt] (6) at (0,2,2) {};
	\node [circle, draw, fill=black!00, inner sep=2pt, minimum width=2pt] (7) at (2,2,2) {};

 	 \draw[dashed]   	 (0) -- (1) ;
 	 \draw[dashed]   	 (0) -- (2) ;
 	 \draw[dashed]   	 (0) -- (3) ;
	 \draw[dashed]   	 (1) -- (4) ;
	 \draw[dashed]   	 (1) -- (5) ;
	 \draw[dashed]   	 (2) -- (4) ;
	 \draw   	 (2) -- (6) ;
	 \draw[dashed]   	 (3) -- (5) ;
	 \draw[dashed]   	 (3) -- (6) ;
 	 \draw  	 (4) -- (7) ;
 	 \draw[dashed]   	 (5) -- (7) ;
 	 \draw[dashed]    	 (6) -- (7) ;
	 
\end{tikzpicture}
	\\
\end{tabular}
\vspace{-0.2cm}
\caption{The eight possible relative complexes $\RR_i$ for a facet $F_i$  in a shelling order $(F_1,\ldots,F_s)$ if $F_i$ is a $3$-dimensional cube.  All of the complexes are stable, excluding the bottom-right complex.}
\end{figure}
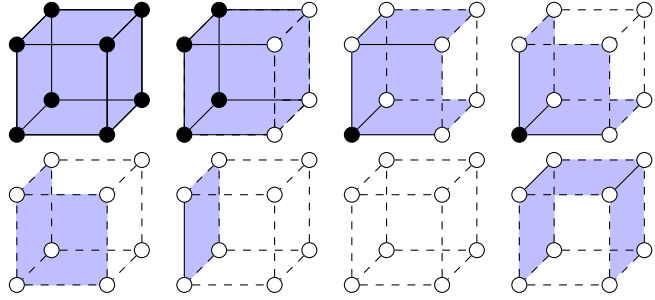

In two dimensions, any relative subcomplex of the cube is stable.  
However, already in three dimensions there exist relative complexes that are not stable, and hence are forbidden from being associated to a facet of a cubical complex in any of its stable shellings.  
For example, Figure~\ref{fig: relative 3-cubes} depicts the eight possible relative complexes of a three cube that may arise as the relative complex associated to a $3$-dimensional facet of a cubical complex with respect to an arbitrary shelling.  
By Lemma~\ref{lem: opposing pairs}, we see that only the first seven are stable.  
The eighth, with its table-top shape, is such that both the codimension $1$  faces of $\CC(\square_3)\setminus\mathcal{D}$ (depicted in blue) and the facets of $\mathcal{D}$ (depicted by their absence) contain an opposing pair.  
Hence, this relative complex cannot be included in any stable shelling of a $3$-dimensional cubical complex.  

\begin{example}[The stable subcomplexes of the $3$-cube]
\label{ex: table top}
Recall that our motivation for excluding certain relative complexes is that, upon subdivision, we need that all relative complexes associated to our shelling order have real-rooted and interlacing $h$-polynomials (see Theorem~\ref{thm: subdivisions of shellable complexes}).  
Suppose we fix a $3$-dimensional cubical complex $\CC$ and consider its barycentric subdivision $\sd(\CC)$ (as defined in subsection~\ref{subsec: barycentric cubical}).  
We will see in Section~\ref{sec: apps} that any of the stable relative complexes depicted in Figure~\ref{fig: relative 3-cubes} have real-rooted and interlacing $h$-polynomials with respect to this subdivision.  
On the other hand, the barycentric subdivision of the table-top shaped relative complex in Figure~\ref{fig: relative 3-cubes} will have $h$-polynomial
\[
22x+4x^2+22x^3.
\]
which is not even unimodal.  Hence, it cannot be real-rooted nor can it interlace the $h$-polynomials of the other relative complexes.  
In this way, the stable relative subcomplexes of the $3$-cube are precisely the relative complexes to which we can apply Theorem~\ref{thm: subdivisions of shellable complexes} with respect to the barycentric subdivision.   
As we will see in Section~\ref{sec: apps}, this will also be true for the other well-studied uniform subdivisions. 
\end{example}
 
Given that not all relative subcomplexes of the cube are stable, it is then natural to ask which cubical complexes admit stable shellings.  
In the remainder of this section, we give some first examples of cubical complexes admitting stable shellings, and we provide an example of a shelling of a cubical complex that is not stable.  
These results will be used in Section~\ref{sec: apps}, when we apply this theory to answer some open questions on the real-rootedness of $h$-polynomials of barycentric subdivisions of cubical complexes. 

\begin{example}[The boundary of the $d$-cube]
\label{ex: d-cube}
Let $\partial\square_d$ denote the boundary of the $d$-dimensional cube $\square_d$, and consider its geometric realization as the boundary of $[0,1]^d\subset\R^d$.   
Let $F_i$ denote the facet of $\square_d$ corresponding to the facet of $[0,1]^d$ defined by the hyperplane $x_i = 0$ for $i\in[d]$, and let $F_{d+i}$ denote the facet corresponding to that of $[0,1]^d$ defined by $x_i = 1$ for $i\in[d]$.  
Hence $F_i, F_{d+i}$ is an opposing pair for all $i\in[d]$.  
We claim that the linear ordering $(F_1,\ldots,F_d,\ldots,F_{2d})$ is a stable shelling of $\CC(\partial\square_d)$.  
By Lemma~\ref{lem: stability implies shelling}, it suffices to show that, for all $i\in\{1,\ldots,2d\}$, the relative complex $\RR_i$ associated to $F_i$ by $(F_1,\ldots,F_d,\ldots,F_{2d})$ is stable.  

Fix $i\in[d]$ and consider $F_i\cap(F_1\cup\cdots\cup F_{i-1})$.
This intersection consists of the facets of $[0,1]^{d-1}\simeq F_i$ defined by $x_1 = 0, x_2 = 0,\ldots, x_{i-1} = 0$.  
Hence, $F_i\cap(F_1\cup\cdots\cup F_{i-1})$ determines a subcomplex of $\CC(\partial\square_{d-1})$ whose set of facets does not contain an opposing pair.  
It follows from Lemma~\ref{lem: opposing pairs} that $\RR_i = \CC(F_i)\setminus \CC(F_i\cap(F_1\cup\cdots\cup F_{i-1}))$ is stable. 
Now consider a facet $F_{d+i}$ for $i\in[d]$, and the subcomplex 
\[
G = F_{d+i}\cap(F_1\cup\cdots\cup F_d\cup \cdots \cup F_{d+i-1})
\]
of its boundary complex $\CC(\partial F_{d+i})$.  
It follows that the codimension $1$ faces of $\RR_{d+i} = \CC(F_{d+i})\setminus\CC(G)$ are determined by the hyperplanes $x_{i+1} = 1,\ldots, x_{d} = 1$.  
Hence, the set of codimension $1$ faces of $\RR_{d+i}$ does not contain an opposing pair of facets of $F_{d+i}$.  
By Lemma~\ref{lem: opposing pairs}, $\RR_{d+i}$ is stable, and we conclude that $(F_1,\ldots,F_d,\ldots,F_{2d})$ is a stable shelling of $\CC(\partial\square_d)$.  

\end{example}

\begin{example}[Piles of cubes]
\label{ex: piles of cubes}
For integers $a_1,\ldots, a_d\in\Z_{\geq0}$, the \emph{pile of cubes} $\mathcal{P}_d(a_1,\ldots,a_d)$ is the polytopal complex formed by all unit cubes with integer vertices in the $d$-dimensional box
\[
B(a_1,\ldots,a_d) := \{x\in\R^d : 0\leq x_i\leq a_i, i\in[d]\}.
\]
Each facet of $B(a_1,\ldots,a_d)$ is uniquely associated to an integer point in the half-open box
\[
B^\circ(a_1,\ldots,a_d) := \{x\in\R^d : 0\leq x_i< a_i, i\in[d]\}.
\]
In particular, the integer point $(z_1,\ldots,z_d)\in\Z^d\cap B^\circ(a_1,\ldots,a_d)$ indexes the unit cube whose lexicographically smallest vertex is $(z_1,\ldots, z_d)$.  
(For two points $a,b\in\Z_{\geq0}^d$, $a>_{\lex} b$ in the lexicographic ordering $>_{\lex}$ whenever the left-most entry in $a-b$ is positive.)
The lexicographic ordering induces a total (linear) ordering of the points in $B^\circ(a_1,\ldots,a_d)$.  
Consider the linear ordering of the facets of $\mathcal{P}_d(a_1,\ldots,a_d)$ induced by the lexicographic order on the integer points in $B^\circ(a_1,\ldots,a_d)$ indexing the facets (from smallest-to-largest).  
By \cite[Example 8.2]{Z12}, this is a shelling order for $\mathcal{P}_d(a_1,\ldots,a_d)$.  
To see that this shelling order is stable, consider a facet $F_{(z_1,\ldots,z_d)}$ of $\mathcal{P}_d(a_1,\ldots,a_d)$, and suppose its associated relative complex $\RR_{(z_1,\ldots,z_d)}$ does not contain its facet lying in the hyperplane $x_i = z_i+1$ for some $i\in[d]$.  
It follows that $F_{(z_1,\ldots,z_i+1,\ldots,z_d)}$ was before $F_{(z_1,\ldots,z_d)}$ in the shelling order.  
However, $(z_1,\ldots,z_i+1,\ldots,z_d) >_{\lex} (z_1,\ldots,z_i,\ldots,z_d)$, so this cannot not be the case.  
Hence, the set of facets of the complex 
\[
\CC\left(F_{(z_1,\ldots,z_d)}\cap\bigcup_{(z_1,\ldots,z_d)>_{\lex}(y_1,\ldots,y_d)}F_{(y_1,\ldots,y_d)}\right)
\]
does not contain an opposing pair because it does not contain any of the facets of $F_{(z_1,\ldots,z_d)}$ lying in a hyperplane $x_i = z_i + 1$ for any $i\in[d]$. 
It follows from Lemma~\ref{lem: opposing pairs} that $\RR_{(z_1,\ldots,z_d)}$ is stable, and thus the lexicographic shelling of $\mathcal{P}_d(a_1,\ldots,a_d)$ is stable. 
\end{example}

\begin{figure}
\label{fig: nonstable shelling}
\centering
\begin{tabular}{c c c}
\begin{tikzpicture}[scale=0.6]

	\draw[fill=blue!25] (0,0,0) -- (2,0,0) -- (2,0,2) -- (0,0,2) -- cycle;
	\draw[fill=blue!25] (0,0,0) -- (0,2,0) -- (2,2,0) -- (2,0,0) -- cycle;
	\draw[fill=blue!25] (0,0,0) -- (0,2,0) -- (0,2,2) -- (0,0,2) -- cycle;
	\draw[fill=blue!25] (2,2,2) -- (2,2,0) -- (2,0,0) -- (2,0,2) -- cycle;
	\draw[fill=blue!25] (2,2,2) -- (2,2,0) -- (0,2,0) -- (0,2,2) -- cycle;
	\draw[fill=blue!25] (2,2,2) -- (2,0,2) -- (0,0,2) -- (0,2,2) -- cycle;

	\draw[fill=blue!25] (2,0,0) -- (4,0,0) -- (4,0,2) -- (2,0,2) -- cycle;
	\draw[fill=blue!25] (2,0,0) -- (2,2,0) -- (4,2,0) -- (4,0,0) -- cycle;
	\draw[fill=blue!25] (2,0,0) -- (2,2,0) -- (2,2,2) -- (2,0,2) -- cycle;
	\draw[fill=blue!25] (4,2,2) -- (4,2,0) -- (4,0,0) -- (4,0,2) -- cycle;
	\draw[fill=blue!25] (4,2,2) -- (4,2,0) -- (2,2,0) -- (2,2,2) -- cycle;
	\draw[fill=blue!25] (4,2,2) -- (4,0,2) -- (2,0,2) -- (2,2,2) -- cycle;

	\draw[fill=blue!25] (4,0,0) -- (6,0,0) -- (6,0,2) -- (4,0,2) -- cycle;
	\draw[fill=blue!25] (4,0,0) -- (4,2,0) -- (6,2,0) -- (6,0,0) -- cycle;
	\draw[fill=blue!25] (4,0,0) -- (4,2,0) -- (4,2,2) -- (4,0,2) -- cycle;
	\draw[fill=blue!25] (6,2,2) -- (6,2,0) -- (6,0,0) -- (6,0,2) -- cycle;
	\draw[fill=blue!25] (6,2,2) -- (6,2,0) -- (4,2,0) -- (4,2,2) -- cycle;
	\draw[fill=blue!25] (6,2,2) -- (6,0,2) -- (4,0,2) -- (4,2,2) -- cycle;
	
	\draw[fill=blue!25] (0,2,0) -- (2,2,0) -- (2,2,2) -- (0,2,2) -- cycle;
	\draw[fill=blue!25] (0,2,0) -- (0,4,0) -- (2,4,0) -- (2,2,0) -- cycle;
	\draw[fill=blue!25] (0,2,0) -- (0,4,0) -- (0,4,2) -- (0,2,2) -- cycle;
	\draw[fill=blue!25] (2,4,2) -- (2,4,0) -- (2,2,0) -- (2,2,2) -- cycle;
	\draw[fill=blue!25] (2,4,2) -- (2,4,0) -- (0,4,0) -- (0,4,2) -- cycle;
	\draw[fill=blue!25] (2,4,2) -- (2,2,2) -- (0,2,2) -- (0,4,2) -- cycle;

	\draw[fill=blue!25] (2,2,0) -- (4,2,0) -- (4,2,2) -- (2,2,2) -- cycle;
	\draw[fill=blue!25] (2,2,0) -- (2,4,0) -- (4,4,0) -- (4,2,0) -- cycle;
	\draw[fill=blue!25] (2,2,0) -- (2,4,0) -- (2,4,2) -- (2,2,2) -- cycle;
	\draw[fill=blue!25] (4,4,2) -- (4,4,0) -- (4,2,0) -- (4,2,2) -- cycle;
	\draw[fill=blue!25] (4,4,2) -- (4,4,0) -- (2,4,0) -- (2,4,2) -- cycle;
	\draw[fill=blue!25] (4,4,2) -- (4,2,2) -- (2,2,2) -- (2,4,2) -- cycle;

	\draw[fill=blue!25] (4,2,0) -- (6,2,0) -- (6,2,2) -- (4,2,2) -- cycle;
	\draw[fill=blue!25] (4,2,0) -- (4,4,0) -- (6,4,0) -- (6,2,0) -- cycle;
	\draw[fill=blue!25] (4,2,0) -- (4,4,0) -- (4,4,2) -- (4,2,2) -- cycle;
	\draw[fill=blue!25] (6,4,2) -- (6,4,0) -- (6,2,0) -- (6,2,2) -- cycle;
	\draw[fill=blue!25] (6,4,2) -- (6,4,0) -- (4,4,0) -- (4,4,2) -- cycle;
	\draw[fill=blue!25] (6,4,2) -- (6,2,2) -- (4,2,2) -- (4,4,2) -- cycle;

	
	\node [circle, draw, fill=black!100, inner sep=2pt, minimum width=2pt] (0) at (0,0,0) {};
	\node [circle, draw, fill=black!100, inner sep=2pt, minimum width=2pt] (1) at (2,0,0) {};
	\node [circle, draw, fill=black!100, inner sep=2pt, minimum width=2pt] (2) at (0,2,0) {};
	\node [circle, draw, fill=black!100, inner sep=2pt, minimum width=2pt] (3) at (0,0,2) {};
	\node [circle, draw, fill=black!100, inner sep=2pt, minimum width=2pt] (4) at (2,2,0) {};
	\node [circle, draw, fill=black!100, inner sep=2pt, minimum width=2pt] (5) at (2,0,2) {};
	\node [circle, draw, fill=black!100, inner sep=2pt, minimum width=2pt] (6) at (0,2,2) {};
	\node [circle, draw, fill=black!100, inner sep=2pt, minimum width=2pt] (7) at (2,2,2) {};

	\node [circle, draw, fill=black!100, inner sep=2pt, minimum width=2pt] (a0) at (2,0,0) {};
	\node [circle, draw, fill=black!100, inner sep=2pt, minimum width=2pt] (a1) at (4,0,0) {};
	\node [circle, draw, fill=black!100, inner sep=2pt, minimum width=2pt] (a2) at (2,2,0) {};
	\node [circle, draw, fill=black!100, inner sep=2pt, minimum width=2pt] (a3) at (2,0,2) {};
	\node [circle, draw, fill=black!100, inner sep=2pt, minimum width=2pt] (a4) at (4,2,0) {};
	\node [circle, draw, fill=black!100, inner sep=2pt, minimum width=2pt] (a5) at (4,0,2) {};
	\node [circle, draw, fill=black!100, inner sep=2pt, minimum width=2pt] (a6) at (2,2,2) {};
	\node [circle, draw, fill=black!100, inner sep=2pt, minimum width=2pt] (a7) at (4,2,2) {};

	\node [circle, draw, fill=black!100, inner sep=2pt, minimum width=2pt] (b0) at (4,0,0) {};
	\node [circle, draw, fill=black!100, inner sep=2pt, minimum width=2pt] (b1) at (6,0,0) {};
	\node [circle, draw, fill=black!100, inner sep=2pt, minimum width=2pt] (b2) at (4,2,0) {};
	\node [circle, draw, fill=black!100, inner sep=2pt, minimum width=2pt] (b3) at (4,0,2) {};
	\node [circle, draw, fill=black!100, inner sep=2pt, minimum width=2pt] (b4) at (6,2,0) {};
	\node [circle, draw, fill=black!100, inner sep=2pt, minimum width=2pt] (b5) at (6,0,2) {};
	\node [circle, draw, fill=black!100, inner sep=2pt, minimum width=2pt] (b6) at (4,2,2) {};
	\node [circle, draw, fill=black!100, inner sep=2pt, minimum width=2pt] (b7) at (6,2,2) {};

	\node [circle, draw, fill=black!100, inner sep=2pt, minimum width=2pt] (c0) at (0,2,0) {};
	\node [circle, draw, fill=black!100, inner sep=2pt, minimum width=2pt] (c1) at (2,2,0) {};
	\node [circle, draw, fill=black!100, inner sep=2pt, minimum width=2pt] (c2) at (0,4,0) {};
	\node [circle, draw, fill=black!100, inner sep=2pt, minimum width=2pt] (c3) at (0,2,2) {};
	\node [circle, draw, fill=black!100, inner sep=2pt, minimum width=2pt] (c4) at (2,4,0) {};
	\node [circle, draw, fill=black!100, inner sep=2pt, minimum width=2pt] (c5) at (2,2,2) {};
	\node [circle, draw, fill=black!100, inner sep=2pt, minimum width=2pt] (c6) at (0,4,2) {};
	\node [circle, draw, fill=black!100, inner sep=2pt, minimum width=2pt] (c7) at (2,4,2) {};

	\node [circle, draw, fill=black!100, inner sep=2pt, minimum width=2pt] (d0) at (2,2,0) {};
	\node [circle, draw, fill=black!100, inner sep=2pt, minimum width=2pt] (d1) at (4,2,0) {};
	\node [circle, draw, fill=black!100, inner sep=2pt, minimum width=2pt] (d2) at (2,4,0) {};
	\node [circle, draw, fill=black!100, inner sep=2pt, minimum width=2pt] (d3) at (2,2,2) {};
	\node [circle, draw, fill=black!100, inner sep=2pt, minimum width=2pt] (d4) at (4,4,0) {};
	\node [circle, draw, fill=black!100, inner sep=2pt, minimum width=2pt] (d5) at (4,2,2) {};
	\node [circle, draw, fill=black!100, inner sep=2pt, minimum width=2pt] (d6) at (2,4,2) {};
	\node [circle, draw, fill=black!100, inner sep=2pt, minimum width=2pt] (d7) at (4,4,2) {};

	\node [circle, draw, fill=black!100, inner sep=2pt, minimum width=2pt] (e0) at (4,2,0) {};
	\node [circle, draw, fill=black!100, inner sep=2pt, minimum width=2pt] (e1) at (6,2,0) {};
	\node [circle, draw, fill=black!100, inner sep=2pt, minimum width=2pt] (e2) at (4,4,0) {};
	\node [circle, draw, fill=black!100, inner sep=2pt, minimum width=2pt] (e3) at (4,2,2) {};
	\node [circle, draw, fill=black!100, inner sep=2pt, minimum width=2pt] (e4) at (6,4,0) {};
	\node [circle, draw, fill=black!100, inner sep=2pt, minimum width=2pt] (e5) at (6,2,2) {};
	\node [circle, draw, fill=black!100, inner sep=2pt, minimum width=2pt] (e6) at (4,4,2) {};
	\node [circle, draw, fill=black!100, inner sep=2pt, minimum width=2pt] (e7) at (6,4,2) {};

 	 \draw   	 (0) -- (1) ;
 	 \draw   	 (0) -- (2) ;
 	 \draw   	 (0) -- (3) ;
	 \draw   	 (1) -- (4) ;
	 \draw   	 (1) -- (5) ;
	 \draw   	 (2) -- (4) ;
	 \draw   	 (2) -- (6) ;
	 \draw   	 (3) -- (5) ;
	 \draw   	 (3) -- (6) ;
 	 \draw   	 (4) -- (7) ;
 	 \draw   	 (5) -- (7) ;
 	 \draw   	 (6) -- (7) ;

 	 \draw   	 (a0) -- (a1) ;
 	 \draw   	 (a0) -- (a2) ;
 	 \draw   	 (a0) -- (a3) ;
	 \draw   	 (a1) -- (a4) ;
	 \draw   	 (a1) -- (a5) ;
	 \draw   	 (a2) -- (a4) ;
	 \draw   	 (a2) -- (a6) ;
	 \draw   	 (a3) -- (a5) ;
	 \draw   	 (a3) -- (a6) ;
 	 \draw   	 (a4) -- (a7) ;
 	 \draw   	 (a5) -- (a7) ;
 	 \draw   	 (a6) -- (a7) ;

 	 \draw   	 (b0) -- (b1) ;
 	 \draw   	 (b0) -- (b2) ;
 	 \draw   	 (b0) -- (b3) ;
	 \draw   	 (b1) -- (b4) ;
	 \draw   	 (b1) -- (b5) ;
	 \draw   	 (b2) -- (b4) ;
	 \draw   	 (b2) -- (b6) ;
	 \draw   	 (b3) -- (b5) ;
	 \draw   	 (b3) -- (b6) ;
 	 \draw   	 (b4) -- (b7) ;
 	 \draw   	 (b5) -- (b7) ;
 	 \draw   	 (b6) -- (b7) ;

 	 \draw   	 (c0) -- (c1) ;
 	 \draw   	 (c0) -- (c2) ;
 	 \draw   	 (c0) -- (c3) ;
	 \draw   	 (c1) -- (c4) ;
	 \draw   	 (c1) -- (c5) ;
	 \draw   	 (c2) -- (c4) ;
	 \draw   	 (c2) -- (c6) ;
	 \draw   	 (c3) -- (c5) ;
	 \draw   	 (c3) -- (c6) ;
 	 \draw   	 (c4) -- (c7) ;
 	 \draw   	 (c5) -- (c7) ;
 	 \draw   	 (c6) -- (c7) ;

 	 \draw   	 (d0) -- (d1) ;
 	 \draw   	 (d0) -- (d2) ;
 	 \draw   	 (d0) -- (d3) ;
	 \draw   	 (d1) -- (d4) ;
	 \draw   	 (d1) -- (d5) ;
	 \draw   	 (d2) -- (d4) ;
	 \draw   	 (d2) -- (d6) ;
	 \draw   	 (d3) -- (d5) ;
	 \draw   	 (d3) -- (d6) ;
 	 \draw   	 (d4) -- (d7) ;
 	 \draw   	 (d5) -- (d7) ;
 	 \draw   	 (d6) -- (d7) ;

 	 \draw   	 (e0) -- (e1) ;
 	 \draw   	 (e0) -- (e2) ;
 	 \draw   	 (e0) -- (e3) ;
	 \draw   	 (e1) -- (e4) ;
	 \draw   	 (e1) -- (e5) ;
	 \draw   	 (e2) -- (e4) ;
	 \draw   	 (e2) -- (e6) ;
	 \draw   	 (e3) -- (e5) ;
	 \draw   	 (e3) -- (e6) ;
 	 \draw   	 (e4) -- (e7) ;
 	 \draw   	 (e5) -- (e7) ;
 	 \draw   	 (e6) -- (e7) ;
	 
\end{tikzpicture}

 	& &
	
\begin{tikzpicture}[scale=0.6]

	\draw[fill=blue!25] (0,0,0) -- (2,0,0) -- (2,0,2) -- (0,0,2) -- cycle;
	\draw[fill=blue!25] (0,0,0) -- (0,2,0) -- (2,2,0) -- (2,0,0) -- cycle;
	\draw[fill=blue!25] (0,0,0) -- (0,2,0) -- (0,2,2) -- (0,0,2) -- cycle;
	\draw[fill=blue!25] (2,2,2) -- (2,2,0) -- (2,0,0) -- (2,0,2) -- cycle;
	\draw[fill=blue!25] (2,2,2) -- (2,2,0) -- (0,2,0) -- (0,2,2) -- cycle;
	\draw[fill=blue!25] (2,2,2) -- (2,0,2) -- (0,0,2) -- (0,2,2) -- cycle;

	\draw[fill=blue!25] (2,0,0) -- (4,0,0) -- (4,0,2) -- (2,0,2) -- cycle;
	\draw[fill=blue!25] (2,0,0) -- (2,2,0) -- (4,2,0) -- (4,0,0) -- cycle;
	\draw[fill=blue!25] (2,0,0) -- (2,2,0) -- (2,2,2) -- (2,0,2) -- cycle;
	\draw[fill=blue!25] (4,2,2) -- (4,2,0) -- (4,0,0) -- (4,0,2) -- cycle;
	\draw[fill=blue!25] (4,2,2) -- (4,2,0) -- (2,2,0) -- (2,2,2) -- cycle;
	\draw[fill=blue!25] (4,2,2) -- (4,0,2) -- (2,0,2) -- (2,2,2) -- cycle;

	\draw[fill=blue!25] (4,0,0) -- (6,0,0) -- (6,0,2) -- (4,0,2) -- cycle;
	\draw[fill=blue!25] (4,0,0) -- (4,2,0) -- (6,2,0) -- (6,0,0) -- cycle;
	\draw[fill=blue!25] (4,0,0) -- (4,2,0) -- (4,2,2) -- (4,0,2) -- cycle;
	\draw[fill=blue!25] (6,2,2) -- (6,2,0) -- (6,0,0) -- (6,0,2) -- cycle;
	\draw[fill=blue!25] (6,2,2) -- (6,2,0) -- (4,2,0) -- (4,2,2) -- cycle;
	\draw[fill=blue!25] (6,2,2) -- (6,0,2) -- (4,0,2) -- (4,2,2) -- cycle;
	
	\draw[fill=blue!25] (0,2,0) -- (2,2,0) -- (2,2,2) -- (0,2,2) -- cycle;
	\draw[fill=blue!25] (0,2,0) -- (0,4,0) -- (2,4,0) -- (2,2,0) -- cycle;
	\draw[fill=blue!25] (0,2,0) -- (0,4,0) -- (0,4,2) -- (0,2,2) -- cycle;
	\draw[fill=blue!25] (2,4,2) -- (2,4,0) -- (2,2,0) -- (2,2,2) -- cycle;
	\draw[fill=blue!25] (2,4,2) -- (2,4,0) -- (0,4,0) -- (0,4,2) -- cycle;
	\draw[fill=blue!25] (2,4,2) -- (2,2,2) -- (0,2,2) -- (0,4,2) -- cycle;

	\draw[fill=blue!25, draw = blue!25] (2,3.65,0) -- (2,5.65,0) -- (4,5.65,0) -- (4,3.65,0) -- cycle;
	\draw[fill=blue!25,draw = blue!25] (4,5.65,2) -- (4,5.65,0) -- (2,5.65,0) -- (2,5.65,2) -- cycle;
	\draw[fill=blue!25,draw = blue!25] (4,5.65,2) -- (4,3.65,2) -- (2,3.65,2) -- (2,5.65,2) -- cycle;

	\draw[fill=blue!25] (4,2,0) -- (6,2,0) -- (6,2,2) -- (4,2,2) -- cycle;
	\draw[fill=blue!25] (4,2,0) -- (4,4,0) -- (6,4,0) -- (6,2,0) -- cycle;
	\draw[fill=blue!25] (4,2,0) -- (4,4,0) -- (4,4,2) -- (4,2,2) -- cycle;
	\draw[fill=blue!25] (6,4,2) -- (6,4,0) -- (6,2,0) -- (6,2,2) -- cycle;
	\draw[fill=blue!25] (6,4,2) -- (6,4,0) -- (4,4,0) -- (4,4,2) -- cycle;
	\draw[fill=blue!25] (6,4,2) -- (6,2,2) -- (4,2,2) -- (4,4,2) -- cycle;

	
	\node [circle, draw, fill=black!100, inner sep=2pt, minimum width=2pt] (0) at (0,0,0) {};
	\node [circle, draw, fill=black!100, inner sep=2pt, minimum width=2pt] (1) at (2,0,0) {};
	\node [circle, draw, fill=black!100, inner sep=2pt, minimum width=2pt] (2) at (0,2,0) {};
	\node [circle, draw, fill=black!100, inner sep=2pt, minimum width=2pt] (3) at (0,0,2) {};
	\node [circle, draw, fill=black!100, inner sep=2pt, minimum width=2pt] (4) at (2,2,0) {};
	\node [circle, draw, fill=black!100, inner sep=2pt, minimum width=2pt] (5) at (2,0,2) {};
	\node [circle, draw, fill=black!100, inner sep=2pt, minimum width=2pt] (6) at (0,2,2) {};
	\node [circle, draw, fill=black!100, inner sep=2pt, minimum width=2pt] (7) at (2,2,2) {};

	\node [circle, draw, fill=black!100, inner sep=2pt, minimum width=2pt] (a0) at (2,0,0) {};
	\node [circle, draw, fill=black!100, inner sep=2pt, minimum width=2pt] (a1) at (4,0,0) {};
	\node [circle, draw, fill=black!100, inner sep=2pt, minimum width=2pt] (a2) at (2,2,0) {};
	\node [circle, draw, fill=black!100, inner sep=2pt, minimum width=2pt] (a3) at (2,0,2) {};
	\node [circle, draw, fill=black!100, inner sep=2pt, minimum width=2pt] (a4) at (4,2,0) {};
	\node [circle, draw, fill=black!100, inner sep=2pt, minimum width=2pt] (a5) at (4,0,2) {};
	\node [circle, draw, fill=black!100, inner sep=2pt, minimum width=2pt] (a6) at (2,2,2) {};
	\node [circle, draw, fill=black!100, inner sep=2pt, minimum width=2pt] (a7) at (4,2,2) {};

	\node [circle, draw, fill=black!100, inner sep=2pt, minimum width=2pt] (b0) at (4,0,0) {};
	\node [circle, draw, fill=black!100, inner sep=2pt, minimum width=2pt] (b1) at (6,0,0) {};
	\node [circle, draw, fill=black!100, inner sep=2pt, minimum width=2pt] (b2) at (4,2,0) {};
	\node [circle, draw, fill=black!100, inner sep=2pt, minimum width=2pt] (b3) at (4,0,2) {};
	\node [circle, draw, fill=black!100, inner sep=2pt, minimum width=2pt] (b4) at (6,2,0) {};
	\node [circle, draw, fill=black!100, inner sep=2pt, minimum width=2pt] (b5) at (6,0,2) {};
	\node [circle, draw, fill=black!100, inner sep=2pt, minimum width=2pt] (b6) at (4,2,2) {};
	\node [circle, draw, fill=black!100, inner sep=2pt, minimum width=2pt] (b7) at (6,2,2) {};

	\node [circle, draw, fill=black!100, inner sep=2pt, minimum width=2pt] (c0) at (0,2,0) {};
	\node [circle, draw, fill=black!100, inner sep=2pt, minimum width=2pt] (c1) at (2,2,0) {};
	\node [circle, draw, fill=black!100, inner sep=2pt, minimum width=2pt] (c2) at (0,4,0) {};
	\node [circle, draw, fill=black!100, inner sep=2pt, minimum width=2pt] (c3) at (0,2,2) {};
	\node [circle, draw, fill=black!100, inner sep=2pt, minimum width=2pt] (c4) at (2,4,0) {};
	\node [circle, draw, fill=black!100, inner sep=2pt, minimum width=2pt] (c5) at (2,2,2) {};
	\node [circle, draw, fill=black!100, inner sep=2pt, minimum width=2pt] (c6) at (0,4,2) {};
	\node [circle, draw, fill=black!100, inner sep=2pt, minimum width=2pt] (c7) at (2,4,2) {};

	\node [circle, draw, fill=black!100, inner sep=2pt, minimum width=2pt] (e0) at (4,2,0) {};
	\node [circle, draw, fill=black!100, inner sep=2pt, minimum width=2pt] (e1) at (6,2,0) {};
	\node [circle, draw, fill=black!100, inner sep=2pt, minimum width=2pt] (e2) at (4,4,0) {};
	\node [circle, draw, fill=black!100, inner sep=2pt, minimum width=2pt] (e3) at (4,2,2) {};
	\node [circle, draw, fill=black!100, inner sep=2pt, minimum width=2pt] (e4) at (6,4,0) {};
	\node [circle, draw, fill=black!100, inner sep=2pt, minimum width=2pt] (e5) at (6,2,2) {};
	\node [circle, draw, fill=black!100, inner sep=2pt, minimum width=2pt] (e6) at (4,4,2) {};
	\node [circle, draw, fill=black!100, inner sep=2pt, minimum width=2pt] (e7) at (6,4,2) {};

 	 \draw   	 (0) -- (1) ;
 	 \draw   	 (0) -- (2) ;
 	 \draw   	 (0) -- (3) ;
	 \draw   	 (1) -- (4) ;
	 \draw   	 (1) -- (5) ;
	 \draw   	 (2) -- (4) ;
	 \draw   	 (2) -- (6) ;
	 \draw   	 (3) -- (5) ;
	 \draw   	 (3) -- (6) ;
 	 \draw   	 (4) -- (7) ;
 	 \draw   	 (5) -- (7) ;
 	 \draw   	 (6) -- (7) ;

 	 \draw   	 (a0) -- (a1) ;
 	 \draw   	 (a0) -- (a2) ;
 	 \draw   	 (a0) -- (a3) ;
	 \draw   	 (a1) -- (a4) ;
	 \draw   	 (a1) -- (a5) ;
	 \draw   	 (a2) -- (a4) ;
	 \draw   	 (a2) -- (a6) ;
	 \draw   	 (a3) -- (a5) ;
	 \draw   	 (a3) -- (a6) ;
 	 \draw   	 (a4) -- (a7) ;
 	 \draw   	 (a5) -- (a7) ;
 	 \draw   	 (a6) -- (a7) ;

 	 \draw   	 (b0) -- (b1) ;
 	 \draw   	 (b0) -- (b2) ;
 	 \draw   	 (b0) -- (b3) ;
	 \draw   	 (b1) -- (b4) ;
	 \draw   	 (b1) -- (b5) ;
	 \draw   	 (b2) -- (b4) ;
	 \draw   	 (b2) -- (b6) ;
	 \draw   	 (b3) -- (b5) ;
	 \draw   	 (b3) -- (b6) ;
 	 \draw   	 (b4) -- (b7) ;
 	 \draw   	 (b5) -- (b7) ;
 	 \draw   	 (b6) -- (b7) ;

 	 \draw   	 (c0) -- (c1) ;
 	 \draw   	 (c0) -- (c2) ;
 	 \draw   	 (c0) -- (c3) ;
	 \draw   	 (c1) -- (c4) ;
	 \draw   	 (c1) -- (c5) ;
	 \draw   	 (c2) -- (c4) ;
	 \draw   	 (c2) -- (c6) ;
	 \draw   	 (c3) -- (c5) ;
	 \draw   	 (c3) -- (c6) ;
 	 \draw   	 (c4) -- (c7) ;
 	 \draw   	 (c5) -- (c7) ;
 	 \draw   	 (c6) -- (c7) ;

 	 \draw   	 (e0) -- (e1) ;
 	 \draw   	 (e0) -- (e2) ;
 	 \draw   	 (e0) -- (e3) ;
	 \draw   	 (e1) -- (e4) ;
	 \draw   	 (e1) -- (e5) ;
	 \draw   	 (e2) -- (e4) ;
	 \draw   	 (e2) -- (e6) ;
	 \draw   	 (e3) -- (e5) ;
	 \draw   	 (e3) -- (e6) ;
 	 \draw   	 (e4) -- (e7) ;
 	 \draw   	 (e5) -- (e7) ;
 	 \draw   	 (e6) -- (e7) ;

	\node [circle, draw, fill=black!00, inner sep=2pt, minimum width=2pt] (d0) at (2,3.65,0) {};
	\node [circle, draw, fill=black!00, inner sep=2pt, minimum width=2pt] (d1) at (4,3.65,0) {};
	\node [circle, draw, fill=black!00, inner sep=2pt, minimum width=2pt] (d2) at (2,5.65,0) {};
	\node [circle, draw, fill=black!00, inner sep=2pt, minimum width=2pt] (d3) at (2,3.65,2) {};
	\node [circle, draw, fill=black!00, inner sep=2pt, minimum width=2pt] (d4) at (4,5.65,0) {};
	\node [circle, draw, fill=black!00, inner sep=2pt, minimum width=2pt] (d5) at (4,3.65,2) {};
	\node [circle, draw, fill=black!00, inner sep=2pt, minimum width=2pt] (d6) at (2,5.65,2) {};
	\node [circle, draw, fill=black!00, inner sep=2pt, minimum width=2pt] (d7) at (4,5.65,2) {};

 	 \draw[dashed]   	 (d0) -- (d1) ;
 	 \draw[dashed]   	 (d0) -- (d2) ;
 	 \draw[dashed]   	 (d0) -- (d3) ;
	 \draw[dashed]   	 (d1) -- (d4) ;
	 \draw[dashed]   	 (d1) -- (d5) ;
	 \draw   	 (d2) -- (d4) ;
	 \draw[dashed]   	 (d2) -- (d6) ;
	 \draw[dashed]   	 (d3) -- (d5) ;
	 \draw[dashed]   	 (d3) -- (d6) ;
 	 \draw[dashed]   	 (d4) -- (d7) ;
 	 \draw[dashed] 	 (d5) -- (d7) ;
 	 \draw  	 (d6) -- (d7) ;
	 
	 \node at (1.1,1,1) {\tiny$F_{(0,0,0)}$} ;
	 \node at (3.1,1,1) {\tiny$F_{(0,1,0)}$} ;
	 \node at (5.1,1,1) {\tiny$F_{(0,2,0)}$} ;
	 \node at (1.1,3,1) {\tiny$F_{(0,0,1)}$} ;
	 \node at (3.1,4.8,1) {\tiny$F_{(0,1,1)}$} ;
	 \node at (5.1,3,1) {\tiny$F_{(0,2,1)}$} ;

\end{tikzpicture}

	\\
\end{tabular}
\vspace{-0.2cm}
\caption{The pile of cubes $P_3(1,3,2)$ (on the left) and the final step in the shelling from Example~\ref{ex: a non-stable shelling}(on the right).  Since the last relative complex is not stable, then neither is this shelling.}
\end{figure}
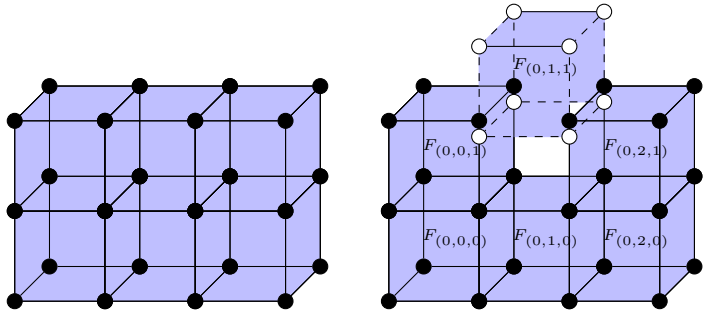

\begin{example}[A non-stable shelling]
\label{ex: a non-stable shelling}
While the pile of cubes $\mathcal{P}_d(a_1,\ldots,a_d)$ always admits a stable shelling (as seen in Example~\ref{ex: piles of cubes}), there exist piles of cubes with shellings that are not stable.  
For instance, the pile of cubes $P_3(1,3,2)$, depicted in Figure~\ref{fig: nonstable shelling}, has six facets
$
F_{(0,0,0)},F_{(0,1,0)},F_{(0,2,0)},F_{(0,0,1)},F_{(0,1,1)}, \mbox{ and } F_{(0,2,1)}.
$
The linear ordering of these facets $(F_{(0,0,0)},F_{(0,1,0)},F_{(0,2,0)},F_{(0,0,1)},F_{(0,2,1)},F_{(0,1,1)})$ is a shelling order of $P_3(1,3,2)$.  
However, the relative complex associated to $F_{(0,1,1)}$ by this order is the table-top complex depicted in the bottom-right of Figure~\ref{fig: relative 3-cubes}.  
Since this complex is not stable, then neither is this shelling of $P_3(1,3,2)$.  
\end{example}

The previous observations demonstrate that some of the classic examples of cubical complexes admit stable shellings but also that not all shellings of these complexes are stable.  
In the coming sections, we will show that more complicated examples of cubical complexes admit stable shellings, and we will use this fact, together with Theorem~\ref{thm: subdivisions of shellable complexes} to provide some answers to open questions in the literature.  
However, while we can prove the existence of stable shellings in the desired cases, it is not clear if they exist in all cases.  
So we end this section with the following question. 
\begin{question}
\label{quest: existence}
Does there exist a shellable cubical complex $\CC$ for which no shelling is stable?
\end{question}

\section{Applications}
\label{sec: apps}
In this section, we apply Theorem~\ref{thm: subdivisions of shellable complexes} and the notion of stable shellings to some classical subdivisions of the boundary complexes of polytopes that are of interest in algebraic, geometric, and topological combinatorics. 
In Subsection~\ref{subsec: barycentric cubical}, we show that the barycentric subdivision of a cubical complex admitting a stable line shelling has a real-rooted $h$-polynomial.  
Applying this result, we positively answer a question of Brenti and Welker \cite[Question 1]{BW08} for the well-known constructions of {\em cubical polytopes}; i.e., polytopes whose facets are all cubes.
In its most general form, the question is as follows:
\begin{problem}
\cite[Question 1]{BW08}
\label{prob: subdivisions of polytopes}
Let $\CC$ be the boundary complex of an arbitrary polytope. 
Is the $h$-polynomial of the barycentric subdivision of $\CC$ real-rooted?  
\end{problem}
In \cite{MW13}, cubical polytopes are proposed as the first case of interest, as the results of \cite{BW08} already answered the question in the case of simplicial (and simple) polytopes.  
Within the literature on cubical polytopes there are surprisingly few explicitly constructed cubical polytopes. 
The most well-known constructions are the {\em cuboids}, which were first introduced by Gr\"unbaum in \cite{G67}, the {\em capped cubical polytopes} \cite{J93}, which are a cubical generalization of stacked simplicial polytopes \cite{K04}, and the {\em neighborly cubical polytopes} \cite{BBC97}.  
Constructing cubical polytopes is, in general, a nontrivial task, as noted in \cite{JZ00} and the thesis \cite{S04}.  
However, by applying Theorem~\ref{thm: subdivisions of shellable complexes} and the notion of stable shellings, in Subsection~\ref{subsec: barycentric cubical} we will be able to positively answer Problem~\ref{prob: subdivisions of polytopes} for all three of these constructions.  

At the same time, Theorem~\ref{thm: subdivisions of shellable complexes} and the associated stable shellings can also be applied to subdivisions other than the barycentric subdivision.  
In Subsections~\ref{subsec: edgewise cubical} and~\ref{subsec: edgewise simplicial}, we apply these techniques to the edgewise subdivision of simplicial and cubical complexes so as to solve a second problem of Mohammadi and Welker \cite[Problem 27]{MW13} for shellable complexes.
We also observe that a (non-geometric) solution to this problem follows from a recent result of Jochemko \cite{J16}. 

In the following, we will make use of some well-studied real-rooted polynomials, which can be defined as follows:
For $d,r\geq 1$ and $0\leq \ell\leq d$, let $A_{d,\ell}^{(r)}$ be the polynomial defined by the relation
\begin{equation}
\label{eqn: colored}
\sum_{t\geq0}(rt)^\ell(rt+1)^{d-\ell}x^t = \frac{A_{d,\ell}^{(r)}}{(1-x)^{d+1}}.
\end{equation}
We call $A_{d,\ell}^{(r)}$ the {\em $d^{th}$ $r$-colored $\ell$-Eulerian polynomial}.  
When $r = 1$ and $\ell = 0$, $A_{d,\ell}^{(r)}$ is the classical Eulerian polynomial, which enumerates the elements of $\mathfrak{S}_d$ by the excedance statistic.  
When $r = 2$ and $\ell = 0$, $A_{d,\ell}^{(r)}$ is the {\em Type B Eulerian polynomial}, which enumerates signed permutations.
For $d\geq 1$, the polynomials $A_{d,0}^{(1)}$ and $A_{d,0}^{(2)}$ are symmetric with respect to degree $d-1$ and $d$, respectively.  
When $\ell = 0$ and $r\geq1$, $A_{d,\ell}^{(r)}$ is the {\em $d^{th}$ colored Eulerian polynomial}, which enumerates the the elements of the wreath product $\Z_r\wr\mathfrak{S}_d$ with respect to their excedance statistic (see \cite[Section 3]{BS19}, for example).  
It is an immediate consequence of \cite[Theorem 4.6]{B06} that $A_{d,\ell}^{(r)}$ has only real, simple zeros for all $d,r\geq1$ and $0\leq \ell \leq d$.  
\begin{lemma}
\label{lem: colored eulerian}
For $d,r\geq 1$ and $0\leq \ell\leq d$, the polynomial $A_{d,\ell}^{(r)}$ has only simple, real zeros.  
Moreover, for a fixed $d,r\geq1$, $\left(A_{d,\ell}^{(r)}\right)_{\ell=0}^d$ is an interlacing sequence. 
\end{lemma}
Lemma~\ref{lem: colored eulerian} will play a key role in the coming subsections.

\subsection{The barycentric subdivision of cubical complexes}
\label{subsec: barycentric cubical}
Given a polytopal complex $\CC$, let $L(\CC)$ denote its face lattice with partial order $<_\CC$ given by inclusion.  
The {\em barycentric subdivision} of $\CC$ is the simplicial complex $\sd(\CC)$ whose $k$-dimensional faces are the subsets $\{F_0,F_1,\ldots,F_k\}$ of faces of $\CC$ for which 
\[
\emptyset<_\CC F_0<_\CC F_1 <_\CC \cdots <_\CC F_k
\]
is a strictly increasing chain in $L(\CC)$.  
Our goal in this subsection is to show that $h(\sd(\CC);x)$ is real-rooted when $\CC$ is a cubical complex admitting a stable shelling.  
To do so, we need to first consider the $h$-polynomials of relative complexes of barycentrically subdivided cubes.
To determine these polynomials, we will make use of the lemmata developed in Subsection~\ref{subsec: ehrhart theory}.
In the following, we will let $\square_d$ denote the (abstract) $d$-dimensional cube.
The following result is well-known.
\begin{lemma}
\label{lem: h-polynomial of boundary}
For $d\geq1$ we have that
\[
h(\sd(\partial\square_d);x) = A_{d,0}^{(2)} = h(\sd(\square_d);x).
\]
\end{lemma}

\begin{proof}
The first equality is noted to be well-known in \cite{MW13}. 
The second equality follows from the general fact that if $\CC$ is a simplicial complex and $\DD$ is the simplicial complex produced by coning over $\CC$ then $h(\CC;x) = h(\DD;x)$.
\end{proof}

We will also let $[-1,1]^d\subset\R^d$ denote the geometric realization of $\square_d$ in $d$-dimensional real-Euclidean space as the convex hull of all $(-1,1)$-vectors in $\R^d$.  
The following lemma is likely well-known to experts in the field, but we include a proof for the sake of completeness.
\begin{lemma}
\label{lem: geometric realization}
Let $T_d$ denote the triangulation of the $d$-cube $[-1,1]^d$ that is induced by the hyperplanes $x_i = \pm x_j$ for $0\leq i<j\leq d$ and $x_i = 0$ for $i\in[d]$.  
Then $T_d$ is abstractly isomorphic to the barycentric subdivision of the $d$-cube.
\end{lemma}

\begin{proof}
We will induct on $d$.  
Observe that the result holds for the base case of $d=1$, and suppose it holds up to some $d-1$.  
We first consider the triangulation $T_d$ of $[-1,1]^d$ restricted to each facet $F$ of the cube.  
Note that $F$ is a $(d-1)$-cube and can be described by intersecting $[-1,1]^d$ with either the hyperplane $x_i = 1$ or $x_i = -1$ for some fixed $i \in [d]$.  
We now focus on how the hyperplanes defining our triangulation intersect $F$.  
First observe that the induced triangulation of $F$ by $T_d$ is defined by hyperplanes of the form $x_i = \pm x_j$, for $j \neq i$. 
Notice that the hyperplanes that intersect the interior of $F$ are exactly those of the form $x_j = \pm x_k$ or $x_j = 0$, for $j,k \neq i$. 
Up to a possible change of coordinates, we see that this subdivision when restricted to $F$ is exactly the subdivision $T_{d-1}$ of the $(d-1)$-cube $[-1,1]^{d-1}$.
By our inductive hypothesis, this is abstractly isomorphic to the barycentric subdivision of $F$.  
Thus, our given triangulation induces a barycentric subdivision of the boundary complex of $[-1,1]^d$.  
We now investigate how the given hyperplanes subdivide $[-1,1]^d$ as a whole.  
Since these hyperplanes meet at a unique point in the interior of $[-1,1]^d$ (the origin), the subdivision of $[-1,1]^d$ can be described by taking the induced subdivision of the boundary complex and coning over an interior point.  
Since the barycentric subdivision of a polytope is given by coning over the barycentric subdivision of its boundary, the result follows.
\end{proof}

The triangulation $T_d$ from Lemma~\ref{lem: geometric realization} has an $h$-polynomial with a well-known combinatorial interpretation:
A {\em signed permutation} on $[d]$ is a pair $(\pi,\varepsilon)\in\mathfrak{S}_d\times \{-1,1\}^d$, which we sometimes denote as $\pi_1^{\varepsilon_1}\cdots\pi_d^{\varepsilon_d}$, where $\pi = \pi_1\cdots\pi_d$ and $\varepsilon = (\varepsilon_1,\ldots,\varepsilon_d)$. 
Set $\pi_0:=0$ and $\varepsilon_0:=1$ for all $(\pi,\varepsilon)\in \mathfrak{S}_d\times \{-1,1\}^d$ and all $d\geq 1$.  
Then $i\in[d-1]_0 :=\{0,1,\ldots,d-1\}$ is a {\em descent} of $(\pi,\varepsilon)$ if $\varepsilon_i\pi_i > \varepsilon_{i+1}\pi_{i+1}$.
We also let
\begin{equation*}
\begin{split}
\Des(\pi,\varepsilon) &:= \{i\in[d-1]_0 : \varepsilon_i\pi_i > \varepsilon_{i+1}\pi_{i+1}\}, \text{ and } \\
\des(\pi,\varepsilon) &:= |\Des(\pi,\varepsilon)|.	\\
\end{split}
\end{equation*}
Let $0\leq \ell \leq d$.
Going one step further, we define the {\em $\ell$-descent set} of $(\pi,\varepsilon)$ to be
\[
\Des_\ell(\pi,\varepsilon) :=
\begin{cases}
\Des(\pi,\varepsilon)\cup\{d\} 		&	\text{ if $d+1-\ell \leq \varepsilon_d\pi_d\leq d$, }	\\
\Des(\pi,\varepsilon)				& 	\text{ otherwise. }						\\
\end{cases}
\]
We then let $\des_\ell(\pi,\varepsilon) :=|\Des_\ell(\pi,\varepsilon)|$.
The {\em Type B $\ell$-Eulerian polynomial} is defined as
\[
B_{d,\ell} := \sum_{(\pi,\varepsilon)\in\mathfrak{S}_d\times\{-1,1\}^d : \varepsilon_d\pi_d = d+1-\ell}x^{\des_\ell(\pi,\varepsilon)}.
\]
For $0\leq \ell \leq d$, we can then make use of the following theorem from \cite{BJM16}:
\begin{theorem}
\cite[Theorem 5.1]{BJM16}
\label{thm: BJM}
For $d\geq 1$ and $0\leq \ell\leq d$, 
\[
\Ehr_{[-1,1]^d_\ell}(x) = \frac{B_{d+1,\ell+1}}{(1-x)^{d+1}}.
\]
\end{theorem}
In \cite{BJM16}, it is further noted that for $0\leq \ell\leq d$, 
\begin{equation}
\label{eqn: in fact a polynomial}
i([-1,1]^d_\ell;t) = \sum_{[\ell]\subseteq S\subseteq[d]}(2t)^{|S|} = (2t)^\ell(2t+1)^{d-\ell}.
\end{equation}
From this, it follows that $B_{d+1,\ell+1} = A_{d,\ell}^{(2)}$, for $0\leq \ell\leq d$.
The polynomials $B_{d+1,\ell+1}$ for $d+1\leq \ell\leq 2d$ can also be computed using the polynomials $A_{d,\ell}^{(2)}$.  
However, it requires a small geometric trick.
\begin{lemma}
\label{lem: ehrhart reciprocity}
For $d\geq1$ and $0\leq \ell < d$,
\[
h^\ast([-1,1]^d_{2d-\ell};x) = \I_{d+1}h^\ast([-1,1]^d_{\ell};x).
\]
In particular, $h^\ast([-1,1]^d_{2d-\ell};x) = x\I_dA_{d,\ell}^{(2)}$.
\end{lemma}

\begin{proof}
Notice first that for $0\leq \ell\leq d$, the half-open polytope $[-1,1]^d_{\ell}$ corresponds to $[-1,1]^d\setminus \HH$ where we have removed all facets visible from a point $q\in\R^n$ for a fixed choice of $q$.  
Let $B_\ell$ denote the set of all such points visible from $q$ on $\partial[-1,1]^d$, and let $D_\ell :=\overline{\partial[-1,1]^d\setminus B_\ell}$.
Notice also, if $[-1,1]^d_\ell = [-1,1]^d\setminus B_\ell$ then $[-1,1]^d_{2d-\ell}$ is unimodularly equivalent to $[-1,1]^d\setminus D_\ell$ (namely, up to rotation).  

Next, consider the case when $\ell = 0$.  
Then the desired statement follows directly from classic Ehrhart-MacDonald reciprocity \cite[Theorem 4.1]{BR07}.
Thus, we need only prove the statement when $0<\ell<d$. 
Assuming this is the case, we then know that $B_\ell$ and $D_\ell$ are both nonempty, and $i(P\setminus B;t)$ is a polynomial in $t$ (see equation~\eqref{eqn: in fact a polynomial}).
Thus, by Lemma~\ref{lem: half-open reciprocity}, we know that
\[
h^\ast([-1,1]^d_{2d-\ell};x) = h^\ast(P\setminus D_\ell ; x) = \I_{d+1}h^\ast(P\setminus B_\ell; x) = \I_{d+1}h^\ast([-1,1]^d_\ell;x).
\]
The fact that $h^\ast([-1,1]^d_{2d-\ell};x) = x\I_dA_{d,\ell}^{(2)}$ then follows from equations~\eqref{eqn: colored} and~\eqref{eqn: in fact a polynomial}, and the fact that $A_{d,\ell}^{(2)}$ has $d$ simple, real zeros.  
The fact that the zeros of $A_{d,\ell}^{(2)}$ are simple and real is noted in Lemma~\ref{lem: colored eulerian}.  
The fact that there are $d$ such zeros can be derived from the combinatorial interpretation of $A_{d,\ell}^{(2)}$ as $B_{d+1,\ell+1}$.  
\end{proof}

Given our interpretation of the polynomials $h^\ast([-,1,1]^d_\ell;x)$ for $0\leq \ell\leq 2d$ in terms of the polynomials $A_{d,\ell}^{(2)}$, we can prove the following theorem for shellable cubical complexes.  
\begin{theorem}
\label{thm: shellable cubical complexes}
Let $\CC$ be a cubical complex with a stable shelling.  Then $h(\sd(\CC);x)$ is real-rooted. 
\end{theorem}

\begin{proof}
Let $(F_1,\ldots,F_s)$ be a stable shelling of the $d$-dimensional cubical complex $\CC$.  
Then for all $i\in[s]$, the relative complex
$
\RR_i
$
associated to $F_i$ by $(F_1,\ldots,F_s)$ is stable.  
Since $\CC$ is a $d$-dimensional cubical complex, each $\RR_i$ is a relative subcomplex of a $d$-dimensional cube.  
By Lemma~\ref{lem: opposing pairs}, it follows that each $\RR_i$ is isomorphic to $\sd([-1,1]^d)_\ell$ for some $0\leq \ell\leq 2d$, where
\[
\sd([-1,1]^d)_\ell := \sd([-1,1]^d)\setminus\{x_d = 1,\ldots, x_{d+1-\ell} = 1\}
\]
for $0\leq \ell\leq d$, and
\[
\sd([-1,1]^d)_\ell := \sd([-1,1]^d)\setminus\{x_d = 1,\ldots, x_{1} = 1\}\cup\{x_d = -1,\ldots,x_{2d+1-\ell} = -1\}
\]
for $d+1\leq \ell \leq 2d$.  
So by Theorem~\ref{thm: subdivisions of shellable complexes}, it suffices to show that the sequence of $h$-polynomials $(h(\sd([-1,1]^d)_\ell;x))_{\ell = 0}^{2d}$ forms an interlacing sequence. 

By Lemma~\ref{lem: geometric realization}, each $\sd([-1,1]^d)_\ell$ is a relative subcomplex of the unimodular triangulation $T_d$ of the $d$-cube $[-1,1]^d$.   
In the case that $\ell = 0$, this complex is a unimodular triangulation of the entire $d$-cube $[-1,1]^d$.  
So by Lemma~\ref{lem: whole polytope}, Theorem~\ref{thm: BJM}, and equation~\eqref{eqn: in fact a polynomial}, we find that
\[
h(\sd([-1,1]^d)_0;x) = h^\ast([-1,1]^d;x) = A_{d,0}^{(2)}.  
\]
When $\ell = 2d$, $\sd([-1,1]^d)_{2d}$ is the relative complex produced by taking the barycentric subdivision of the $d$-cube and then removing its subdivided boundary.  
Hence, the $f$-polynomial satisfies
$
f(\sd([-1,1]^d)_{2d};x) = xf(\sd(\partial\square_d);x).
$
So by Lemma~\ref{lem: h-polynomial of boundary}, we know that $h(\sd([-1,1]^d)_{2d};x) = xA_{d,0}^{(2)}$.
Moreover, since $A_{d,0}^{(2)}$ is known to be symmetric with respect to degree $d$, it follows that 
\[
h(\sd([-1,1]^d)_{2d};x) = x\I_dA_{d,0}^{(2)}.
\]

In the case that $0<\ell< 2d$, the complex $\sd([-1,1]^d)_\ell$ has Euler characteristic $0$. 
So it follows by Lemma~\ref{lem: not whole polytope}, Theorem~\ref{thm: BJM}, and equation~\eqref{eqn: in fact a polynomial} that
\[
h(\sd([-1,1]^d)_\ell;x) = h^\ast([-1,1]^d_\ell;x) = A_{d,\ell}^{(2)}
\]
when $0<\ell\leq d$.
Hence, by Lemma~\ref{lem: ehrhart reciprocity}, for all $0\leq \ell\leq 2d$ it follows that $h(\sd([-1,1]^d)_\ell;x)$ is of the form $A_{d,\ell^\prime}^{(2)}$ or $x\I_dA_{d,\ell^\prime}^{(2)}$ for some $0\leq \ell^\prime \leq d$.  
Thus, it suffices to show that the sequence
\[
\left(A_{d,0}^{(2)}, A_{d,1}^{(2)}, \ldots, A_{d,d-1}^{(2)}, A_{d,d}^{(2)}, x\I_dA_{d,d}^{(2)}, \ldots, x\I_dA_{d,1}^{(2)}, x\I_dA_{d,0}^{(2)}\right)
\]
is interlacing.  
By \cite[Lemma 2.3]{B06}, we need only check that each of the following relations are satisfied:
\begin{enumerate}
	\item $A_{d,0}^{(2)}\prec x\I_dA_{d,0}^{(2)}$,
	\item $A_{d,\ell}^{(2)}\prec A_{d,k}^{(2)}$ for all $0\leq \ell < k\leq d$,
	\item $A_{d,d}^{(2)} \prec x\I_dA_{d,d}^{(2)}$, and
	\item $x\I_dA_{d,k}^{(2)}\prec x\I_dA_{d,\ell}^{(2)}$ for all $0\leq \ell < k\leq d$.
\end{enumerate}
Case $(1)$ is immediate from the fact that $A_{d,0}^{(2)} = \I_dA_{d,0}^{(2)}$ and Lemma~\ref{lem: basic facts}~\eqref{eqn: times x}.
Case $(2)$ follows from Lemma~\ref{lem: colored eulerian}. 
Case $(3)$ follows from \cite[Theorem 3.1]{BS19}, which shows that $\I_dA_{d,d}^{(2)}\prec A_{d,d}^{(2)}$, and Lemma~\ref{lem: basic facts}~\eqref{eqn: times x}, and case $(4)$ follows directly from case $(2)$.
Thus, since this sequence is interlacing, it follows that $h(\sd(\CC);x)$ is real-rooted, which completes the proof.
\end{proof}

We now apply these results to give a positive answer to Problem~\ref{prob: subdivisions of polytopes} for the well-known constructions of cubical polytopes; namely, the cuboids, capped cubical polytopes, and the neighborly cubical polytopes. 
\subsubsection{Barycentric subdivisions of cuboids}
\label{subsubsec: cuboids}
Cuboids are a family of cubical polytopes described by Gr\"unbaum in \cite{G67}.  
For each dimension $d\geq1$, there are $d+1$ cuboids, denoted $Q_\ell^d$ for $0\leq \ell \leq d$.  
The first cuboid in dimension $d$, denoted $Q_0^d$, is the $d$-cube and the rest are defined recursively as follows:
To construct $Q_\ell^d$ for $\ell>0$, glue two copies of $Q_{\ell-1}^d$ at a common $Q^{d-1}_{\ell-1}$.  
The boundary of the resulting complex is the boundary complex of the cuboid $Q_\ell^d$.  
Equivalently, to construct the $\ell^{th}$ $d$-dimensional cuboid for $0<\ell \leq d$, start by taking the geometric realization $[-1,1]^d$ of $Q_0^d$.  
Then consider the geometric subdivision of $[-1,1]^d$ given by intersecting $[-1,1]^d$ with the $\ell$ hyperplanes $x_1 = 0,x_2 = 0,\ldots, x_\ell =0$.  
The boundary of the resulting cubical complex is the boundary complex of the cuboid $Q_\ell^d$.  
Using this second construction of the cuboid $Q_\ell^d$, we can deduce that all cuboids admit a stable shelling, yielding the following corollary to Theorem~\ref{thm: shellable cubical complexes}:
\begin{corollary}
\label{cor: cuboids}
The barycentric subdivision of the boundary complex of a cuboid has a real-rooted $h$-polynomial. 
\end{corollary}

\begin{proof}
Recall that the cuboid $Q_0^d$ is the $d$-dimensional cube, whose barycentric subdivision is well-known to have a real-rooted $h$-polynomial.  
So fix $0<\ell\leq d$. 
By Theorem~\ref{thm: shellable cubical complexes}, it suffices to show that the boundary complex of $Q_\ell^d$, denoted $\CC(\partial Q_\ell^d)$, has a stable shelling.  
By construction, the boundary complex of $Q_\ell^d$ is isomorphic to the subdivision of the boundary complex of $[0,2]^d$ induced by the hyperplanes $x_1 = 1,\ldots,x_\ell = 1$.  
For $i\in[d]$, let $F_i$ and $F_{d+i}$ denote, respectively, the facets of $[0,2]^d$ lying in the hyperplanes $x_i = 0$ and $x_i =2$.  
It follows that, to construct $\CC(\partial Q_\ell^d)$, each facet $F_i$ and $F_{d+i}$ is subdivided into a complex isomorphic to the pile of cubes $\mathcal{P}_{d-1}(c^{(i,\ell)})$, where 
\[
c^{(i,\ell)} = (1,1,\ldots,1) + \sum_{j\in[\ell]\setminus\{i\}}e_j,
\]
where $e_1,\ldots,e_{d-1}\in\R^{d-1}$ denote the standard basis vectors (see Example~\ref{ex: piles of cubes} for the definition of a pile of cubes). 
The facets of $Q_\ell^d$ are then the facets of the piles of cubes $\mathcal{P}_{d-1}(c^{(i,\ell)})$ for all $i\in[d]$.  
Here we have two copies of each pile of cubes, one for $F_i$ and one for $F_{d+i}$. 

Fix the stable shelling order $(F_1,\ldots,F_d,\ldots, F_{2d})$ of $\CC(\partial\square_d)\simeq\CC(\partial[0,2]^d)$ from Example~\ref{ex: d-cube}, and suppose that $F_i$ and $F_{d+i}$ have been subdivided into the pile of cubes $\mathcal{P}_{d-1}(c^{(i,\ell)})$ consisting of $M_i$ cubes.  
Also set $M_{d+i}:=M_i$.  
Just as in Example~\ref{ex: piles of cubes}, we index the cubes in the facet $F_i$ (for $i\in[2d]$) by their lexicographically smallest vertex.  
As proven in Example~\ref{ex: piles of cubes}, this is a stable shelling order for $\mathcal{P}_{d-1}(c^{(i,\ell)})$.  
Suppose this shelling order for the cubes in $F_i$ is $(C_{1_i},\ldots,C_{M_i})$.  
We claim that the linear ordering
\begin{equation}
\label{eqn: order}
(C_{1_1},\ldots,C_{M_1},C_{1_2},\ldots,C_{M_2},\ldots,C_{1_{2d}},\ldots,C_{M_{2d}})
\end{equation}
is a stable shelling of the boundary of $Q_\ell^d$. 
To see this, fix $C_{j_k}$ for $k\in[d]$ and suppose that $C_{j_k}$ is indexed by the integer point $(a_1,\ldots,a_d)\in F_k$, the facet of $[0,2]^d$.  
Note here that $F_k$ is assumed to be a facet of $[0,2]^d$ whose corresponding facet-defining hyperplane is $x_k = 0$.  
Since $C_{j_k}$ is a cube in the subdivided facet $F_k$, it follows that $a_k = 0$ and that $C_{j_k}$ has facet-defining hyperplanes
\[
x_1 = a_1,\ldots,x_{k-1} = a_{k-1},x_{k+1} = a_{k+1},\ldots, x_d = a_d,
\]
and 
$
x_i = a_i +b_i 
$
for $i\in[d]\setminus\{k\}$ and some $b_i\in\{1,2\}$.
In this context, $x_i = a_i$ and $x_i = a_i+b_i$ contain the opposing pair of facets $G_i,G_{d+i}$ of $C_{j_k}$ for $i\in[d]\setminus\{k\}$.
 
We then have two cases: either $b_i =1$ or $b_i =2$. 
In the former case, the facet $G_{d+i}$ defined by $x_i = a_i+b_i$ would only not be a codimension $1$ face of the relative complex associated to $C_{j_k}$ if the cube indexed by $(a_1,\ldots,a_k+1,\ldots,a_d)$ in $F_k$ had preceded $C_{j_k}$ in the order~\eqref{eqn: order}.
However, this cannot happen since $(a_1,\ldots,a_k+1,\ldots,a_d) >_{\lex} (a_1,\ldots,a_k,\ldots,a_d)$ in the lexicographic order.  
In the latter case, the facet-defining hyperplane $x_i = a_i+b_i$ is the hyperplane $x_i = 2$, and hence the facet $G_{d+i}$ defined by this hyperplane would not be in the associated relative complex if and only if a cube lying in a facet $F_{d+i}$ of $[0,2]^d$ for some $i\in[d]$ had preceded $C_{j_k}$ in the order~\eqref{eqn: order}.  
Since this is impossible, we conclude that the facet of $C_{j_k}$ defined by $x_i = a_i+b_i$ is in the relative complex associated to $C_{j_k}$.  
Since this argument holds for all $i\in[d]\setminus\{k\}$, it follows that the set of facets of the complex determined by 
\[
C_{j_k}\cap(C_{1_1}\cup\cdots\cup C_{M_1}\cup\cdots\cup C_{1_k}\cup\cdots\cup C_{{j-1}_{k}})
\]
does not contain an opposing pair.
Hence by Lemma~\ref{lem: opposing pairs}, the relative complex associated to $C_{j_k}$ by the order~\eqref{eqn: order} is stable.  

Now suppose that $k = d+k^\prime$ for some $k^\prime\in[d]$. 
Hence, $C_{j_k}$ is a facet of $\CC(\partial Q_\ell^d)$ lying in the facet $F_{d+k^\prime}$ of $[0,2]^d$.  
Assume once more that $C_{j_k}$ is indexed by the integer point $(a_1,\ldots,a_d)$ in the facet $F_k$ of $[0,2]^d$.  
We claim that none of the facets of $C_{j_k}$ defined by the hyperplanes 
\[
x_1 = a_1,\ldots,x_{k^\prime-1} = a_{k^\prime-1},x_{k^\prime+1} = a_{k^\prime+1},\ldots, x_d = a_d,
\]
are facets of the relative complex $\RR_{j_k}$ associated to $C_{j_k}$ by the order~\eqref{eqn: order}.
This follows via induction.  
Suppose first that $j_k = 1_k$. 
In this case, $C_{j_k}$ is the first cube in the facet $F_k$ that appears in the order~\eqref{eqn: order}.  
Hence, the integer point in $F_k$ indexing $C_{j_k}$ must be $e_{k^\prime}$, the standard basis vector in $\R^d$.  
Thus, $a_i = 0$ for all $i\in[d]\setminus\{k^\prime\}$ (and $a_{k^\prime} = 1$).  
Since all cubes in the facets $F_1,\ldots,F_d$ preceded $C_{j_k}$ in the order~\eqref{eqn: order}, then none of the facets defined by the hyperplanes $x_i =a_i (=0)$ for $i\in[d]\setminus\{k^\prime\}$ can be a facet of the relative complex $\RR_{j_k}$ associated to $C_{j_k}$.  
Hence, by Lemma~\ref{lem: opposing pairs}, the relative complex $\RR_{j_k}$ is stable.  

Similarly, the next cube in the ordering $C_{j_k+1}$ will be indexed by $(a_1,\ldots,a_{k^\ast}+1,\ldots,a_d)$, where $k^\ast$ is the right-most coordinate in $(a_1,\ldots,a_d)$ for which adding $1$ produces a new point in $[0,2]^d\cap\Z^d$ that indexes a cube.  
It then follows that $a_{k^\ast} = 0$ and $b_{k^\ast} = 1$, by construction of $Q_\ell^d$.  
Hence, $C_{j_k+1}$ has the set of facet-defining hyperplanes
\begin{equation}
\label{eqn: hyperplanes}
x_1 = a_1,\ldots,x_{k^\ast} = a_{k^\ast}+1,\ldots, x_{k^\prime-1} = a_{k^\prime-1},x_{k^\prime+1} = a_{k^\prime+1},\ldots, x_d = a_d,
\end{equation}
and their opposites $x_i = a_i +b_i $ for $i\in[d]\setminus\{k^\prime\}$ and some $b_i\in\{1,2\}$.
Since each of the hyperplanes listed in~\eqref{eqn: hyperplanes} is also a facet-defining hyperplane of $C_{j_k}$, it follows that none of them define a facet in the relative complex $\RR_{j_k+1}$ associated to $C_{j_k+1}$.  
Hence the set of facets of $\RR_{j_k+1}$ does not contain an opposing pair. 
Therefore, it is a stable complex by Lemma~\ref{lem: opposing pairs}.    
By iterating this argument, we see that the relative complex of each facet following $C_{j_k}$ in the order~\eqref{eqn: order} is stable.  
The fact that the order~\eqref{eqn: order} is a shelling order now follows from Lemma~\ref{lem: stability implies shelling}.  
Applying Theorem~\ref{thm: shellable cubical complexes} completes the proof.
\end{proof}

Corollary~\ref{cor: cuboids} gives a positive answer to Problem~\ref{prob: subdivisions of polytopes} in the case of cuboids, one of the three well-known constructions of cubical polytopes.  
In the next subsection, we deduce a positive answer to Problem~\ref{prob: subdivisions of polytopes} for the remaining two.  

\subsubsection{Barycentric subdivisions of capped cubical polytopes}
\label{subsubsec: capped cubical polytopes}
Capped cubical polytopes, or {\em stacked cubical polytopes}, are the cubical analogue to {\em stacked simplicial polytopes} \cite{K04}.  
A polytope $P$ is called \emph{capped} over a given cubical polytope $Q$ if there is a combinatorial cube $C$ such that $P = Q\cup C$ and $F := Q\cap C$ is a facet of $Q$. 
In this case, we then think of $P$ as produced by \emph{capping} $Q$ \emph{over} $F$, and we write
\[
P = \capped(Q,F).
\]
We say that a polytope is \emph{$\ell$-fold capped cubical} for some $\ell\in\Z_{\geq0}$ if it can be obtained from a combinatorial cube by $\ell$ capping operations.
In the following, let $\CC$ denote a $(d-1)$-dimensional cubical complex that is the boundary complex of an $\ell$-fold capped polytope.  
Let $\square_d$ denote the (abstract) $d$-cube. 
Our goal in this subsection is to show that $h(\sd(\CC);x)$ is real-rooted whenever $\CC$ is the boundary complex of an $\ell$-capped cubical polytope $P$ for some $\ell\in\Z_{\geq0}$.
To do so, we again use Theorem~\ref{thm: shellable cubical complexes} and the machinery of stable shellings developed in Subsection~\ref{subsec: stable shellings}.

\begin{corollary}
\label{cor: capped real-rootedness}
Let $\CC$ denote the boundary complex of a $d$-dimensional $\ell$-capped cubical polytope.
Then $h(\sd(\CC);x)$ is real-rooted.
\end{corollary}

\begin{proof}
By Theorem~\ref{thm: shellable cubical complexes}, it suffices to show that $\CC$ admits a stable shelling.  
To prove this, we proceed by induction on $\ell\geq 0$.
When $\ell = 0$, $\CC$ is the boundary complex of a $d$-dimensional cube.  
Hence, as we saw in Example~\ref{ex: d-cube}, $\CC$ admits a stable shelling.  

Suppose now that $\CC$ is the boundary complex of a $d$-dimensional $\ell$-capped cubical polytope $P = \capped(Q,F)$, for $\ell > 0$.  
Then $Q$ is a $d$-dimensional $(\ell-1)$-capped cubical polytope, and hence, by our inductive hypothesis, the boundary complex $\DD$ of $Q$ admits a stable shelling 
\begin{equation}
\label{eqn: capped inductive}
(F_1,\ldots,F_M).
\end{equation}
Since $P = \capped(Q,F)$, then it follows that $\CC$ is produced from $\DD$ by subdividing the facet $F$ of $\DD$ into the Schlegel diagram \cite[Definition 5.5]{Z12} of the $d$-dimensional cube $\square_d$ based at a facet $G$ of $\square_d$ that is identified with $F$ in $\DD$.  
Suppose that the facets of $\square_d$ are $G_1,\ldots,G_{2d}$, and suppose that $(G_1,\ldots,G_d,G_{d+1},\ldots,G_{2d})$ is the stable shelling order of $\CC(\partial\square_d)$ given in Example~\ref{ex: d-cube}, where we assume $G = G_1$.  
We claim that the linear ordering 
\begin{equation}
\label{eqn: capped stable}
(F_1,\ldots,F_{k-1},G_2,\ldots,G_d,G_{d+1},\ldots,G_{2d},F_{k+1},\ldots,F_M)
\end{equation}
is a stable shelling of $\CC$.  
Since the relative complex of each $F_i$ for $i\neq k$ in the linear ordering~\eqref{eqn: capped stable} is the same is its relative complex in the ordering~\eqref{eqn: capped inductive}, it suffices to show that the relative complex associated to each $G_i$ for $i\in[2d]$ is stable.  

To see this, consider first a facet $G_i$ for $1< i\leq d$.  
Since $G = G_1$ is identified with the facet $F_k = F$ of $\DD$, then 
\begin{equation}
\label{eqn: subcomplex}
G_i\cap(F_1\cup\ldots\cup F_{k-1}\cup G_2 \cup \cdots \cup G_{i-1}) = G_i\cap(G_1\cup G_2 \cup \cdots \cup G_{i-1})
\end{equation}
if $G_i\cap(F_1\cup\ldots\cup F_{k-1}) \neq \emptyset$, or 
\begin{equation}
\label{eqn: subcomplex nonint}
G_i\cap(F_1\cup\ldots\cup F_{k-1}\cup G_2 \cup \cdots \cup G_{i-1}) = G_i\cap(G_2 \cup \cdots \cup G_{i-1})
\end{equation}
otherwise. 
In the former case, as in Example~\ref{ex: d-cube}, the subcomplex~\eqref{eqn: subcomplex} consists of the facets of $G_i$ defined by the hyperplanes $x_1 = 0, \ldots, x_{i-1} = 0$.  
Therefore, $G_i\cap(F_1\cup\ldots\cup F_{k-1}\cup G_2 \cup \cdots \cup G_{i-1})$ determines a subcomplex of $\CC(\partial\square_{d-1})$ whose set of facets does not contain an opposing pair.  
In the latter case, the subcomplex~\eqref{eqn: subcomplex nonint} consists of the facets of $G_i$ defined by the hyperplanes $x_2 = 0, \ldots, x_{i-1} = 0$, and it again determines a subcomplex of $\CC(\partial\square_{d-1})$ whose set of facets does not contain an opposing pair. 
Hence, by Lemma~\ref{lem: opposing pairs}, the relative complex $\RR_i$ associated to $G_i$ by the ordering~\eqref{eqn: capped stable} is stable in both cases.  

Now consider the facet $G_{d+i}$ for some $i\in[d]$ and the subcomplex
\[
H := G_i\cap(F_1\cup\ldots\cup F_{k-1}\cup G_2 \cup \cdots \cup G_d \cup \cdots \cup G_{d+i-1}).
\]
of the boundary complex $\CC(\partial G_{d+i})$.  
Just as in the case of $G_i$ with $1< i\leq d$, this subcomplex is equal to $G_i\cap (G_1 \cup G_2 \cup \cdots \cup G_d \cup \cdots \cup G_{d+i-1})$ if $G_i\cap(F_1\cup\ldots\cup F_{k-1}) \neq \emptyset$.  
Otherwise, it is equal to $G_i\cap (G_2 \cup \cdots \cup G_d \cup \cdots \cup G_{d+i-1})$. 
In either case, the relative complex $\RR_{d+i}$ associated to $G_{d+i}$ by the ordering~\eqref{eqn: capped stable} is
\[
\RR_{d+i} = \CC(G_{d+i})\setminus \CC(H).
\]
In the former case, its codimension $1$ faces are determined by the hyperplanes $x_{i+1} = 1,\ldots, x_d = 1$.  
Therefore, the set of codimension $1$ faces of $\RR_{d+i}$ does not contain an opposing pair.  
In the latter case, $G_i\cap(F_1\cup\ldots\cup F_{k-1}) = \emptyset$, which can happen in one of two ways:  either $i = 1$ or $i\neq 1$.  
In the case that $i = 1$, it follows that the set of codimension $1$ faces of $\RR_{d+i}$ is determined by the hyperplanes $x_{i+1} = 1,\ldots, x_d = 1$, and thus does not contain an opposing pair.  
In the case that $i \neq 1$, the set of codimension one $1$  is determined by the hyperplanes $x_{i+1} = 1,\ldots, x_d = 1$ and the hyperplane $x_1 = 0$.  
However, since $i+1 >1$, this set still does not contain an opposing pair.  
Thus, by Lemma~\ref{lem: opposing pairs}, the relative complexes $\RR_i$ and $\RR_{d+i}$ are stable for all $i\in[d]$.  
It follows that the ordering~\eqref{eqn: capped stable} is a stable shelling order of $\CC$, which completes the proof.
%
\end{proof}

The third class of well-known cubical polytopes with an explicit construction are the neighborly cubical polytopes.  
These polytopes were introduced in \cite{BBC97}, but no explicit constructions was given. 
However, in \cite{JZ00}, Joswig and Zeigler proved that neighborly cubical polytopes exist by giving such explicit constructions, which they denoted by $C_n^d$.  
In \cite[Comment 1]{JZ00}, they note that the constructions $C_n^d$ are also capped cubical polytopes.  
Hence, as a corollary to Corollary~\ref{cor: capped real-rootedness}, we also obtain a positive answer to Problem~\ref{prob: subdivisions of polytopes} for the family of neighborly cubical polytopes. 
\begin{corollary}
\label{cor: neighborly}
The barycentric subdivision of the boundary complex of a neighborly cubical polytope has a real-rooted $h$-polynomial.
\end{corollary}
By combining Corollary~\ref{cor: cuboids}, Corollary~\ref{cor: capped real-rootedness}, and Corollary~\ref{cor: neighborly}, we obtain a positive answer to Problem~\ref{prob: subdivisions of polytopes} for the well-known families of cubical polytopes; namely, the cuboids, capped cubical polytopes, and the neighborly cubical polytopes. 
To deduce these results we used the theory of stable shellings, developed in Subsection~\ref{subsec: stable shellings}.  
As we will see in Subsection~\ref{subsec: simplicial polytopes}, we can also use stable shelling techniques to give an alternate proof of Brenti and Welker's original solution to Problem~\ref{prob: subdivisions of polytopes} in the case of simplicial polytopes. 
This suggests that the framework of stable shellings is perhaps appropriate for all polytopes, as it yields a proof for all known cases.  
In fact, by a theorem of Bruggesser and Mani \cite{BM72}, it is not unreasonable that these same techniques may be useful in addressing Problem~\ref{prob: subdivisions of polytopes} in its fullest generality.  
In Section~\ref{sec: stable line shellings}, we will explain some of the open questions pertaining to this approach in more detail.  
Before doing so, in Subsections~\ref{subsec: edgewise simplicial} and~\ref{subsec: edgewise cubical}, we will demonstrate how stable shellings can also be applied to subdivisions other than the barycentric subdivision.  
As one application, we will derive an answer to a second problem of Mohammadi and Welker \cite{MW13} for shellable simplicial complexes.  

\subsection{Barycentric subdivisions of simplicial polytopes}
\label{subsec: simplicial polytopes}
In this section, we give an alternative proof of the result of \cite{BW08} that motivated Problem~\ref{prob: subdivisions of polytopes}.  
Namely, we apply Theorem~\ref{thm: subdivisions of shellable complexes} to show that the $h$-polynomial of the boundary complex of a simplicial polytope has only real zeros.  
We will prove this using a similar narrative as in Subsections~\ref{subsubsec: cuboids} and~\ref{subsubsec: capped cubical polytopes}, in that we will use a shelling argument to decompose the complex into relative simplicial complexes, and show that the $h$-polynomials of each complex form an interlacing family.  
Recall that, by Proposition~\ref{prop: stable shellings of simplicial complexes}, any shelling of a simplicial complex is stable, and that stable shellings were defined in Subsection~\ref{subsec: stable shellings} to capture those shellings to which we can apply Theorem~\ref{thm: subdivisions of shellable complexes} for multiple different uniform subdivisions.  
In this subsection and the next, we will indeed see that we can apply Theorem~\ref{thm: subdivisions of shellable complexes} to any shelling of a simplicial complex with respect to the two most common uniform subdivisions: the barycentric subdivision and the edgewise subdivision.  

In this case, the relative complexes associated to facets of our shelling will be subdivided, half-open simplices.
In the following, let $\Delta_d$ denote the $d$-dimensional simplex, and let $\Delta_{d,\ell}$ denote the relative simplicial complex given by removing $\ell$ of the facets of $\Delta_d$ for $0\leq \ell \leq d+1$.  
Note that $\Delta_{d,0} = \Delta_d$.
%
%
%
%
%
%
%
We will require the following well-known result. 

\begin{lemma}
\label{lem: h part open simp}
Let $d\geq1$ and $0\leq\ell\leq d+1$.  
Then
$
h(\Delta_{d,\ell};x) = x^\ell.
$
\end{lemma}

\begin{proof}
By the Principle of Inclusion-Exclusion, we deduce that
\[
f(\Delta_{d,\ell};x) = \sum_{j=0}^\ell(-1)^j\binom{\ell}{j}f(\Delta_{d-j};x).
\]
Since $h(\Delta_d;x) = 1$ for all $d\geq 1$, it then follows that
\[
h(\Delta_{d,\ell};x) = \sum_{j=0}^\ell(-1)^j\binom{\ell}{j}(1-x)^j = x^\ell.
\]
\end{proof}

To make use of this observation, we need to generalize the main result of \cite{BW08} (i.e. ~\cite[Theorem 1]{BW08}) to relative complexes.  
We refer to the reader to \cite{BW08} for the definition of a Boolean cell complex.
\begin{lemma}
\label{lem: generalizing BW}
Let $\CC$ be a $(d-1)$-dimensional Boolean cell complex and $\DD$ a subcomplex of $\CC$.
If the relative complex $\CC\setminus\DD$ is also $(d-1)$-dimensional then
\[
h(\sd(\CC)\setminus\sd(\DD);x) = \sum_{\ell=0}^dh_\ell(\CC\setminus\DD)A_{d,\ell}^{(1)}.
\]
\end{lemma}

\begin{proof}
The definition of the $h$-polynomial of a relative complex $\CC\setminus\DD$ is known to be equivalent to the condition that
\[
h_r(\CC\setminus\DD) = \sum_{i=0}^r(-1)^{r-i}\binom{d-i}{r-i}f_{i-1}
\]
for all $r = 0,\ldots,d$.  
From this formula it follows that for all $r=0,\ldots,d$
\begin{equation*}
\begin{split}
h_r(\sd(\CC)\setminus\sd(\DD)) 
&= \sum_{i=0}^r(-1)^{r-i}\binom{d-i}{r-i}f_{i-1}(\sd(\CC)\setminus\sd(\DD)),	\\
&= \sum_{i=0}^r(-1)^{r-i}\binom{d-i}{r-i}(f_{i-1}(\sd(\CC))-f_{i-1}(\sd(\DD)),	\\
&= \sum_{i=0}^r(-1)^{r-i}\binom{d-i}{r-i}f_{i-1}(\sd(\CC))-\sum_{i=0}^r(-1)^{r-i}\binom{d-i}{r-i}f_{i-1}(\sd(\DD)).	\\
\end{split}
\end{equation*}
Since $\sd(\CC)$ and $\sd(\DD)$ are both Boolean cell complexes then we can apply \cite[Lemma 1]{BW08}.  
This yields
\begin{equation*}
\begin{split}
h_r(\sd(\CC)\setminus\sd(\DD)) 
&= \sum_{i=0}^r(-1)^{r-i}\binom{d-i}{r-i}\sum_{m=0}^d f_{m-1}(\CC)S(m,k)m!	\\
&\,\,\,\,\,\,\,\,-\sum_{i=0}^r(-1)^{r-i}\binom{d-i}{r-i}\sum_{m=0}^d f_{m-1}(\DD)S(m,k)m!.	\\
\end{split}
\end{equation*}
Notice here that the formula for $f_{m-1}(\sd(\DD))$ given in \cite[Lemma 1]{BW08} still holds with respect to degree $d$ even if $\DD$ has dimension less than $d$.  
This is immediate from the proof of \cite[Lemma 1]{BW08} since $f_{m-1}(\DD) = 0$ for all $m$ greater than the dimension of $\DD$. 
Hence, it follows that
\begin{equation*}
\begin{split}
h_r(\sd(\CC)\setminus\sd(\DD)) 
&= \sum_{i=0}^r(-1)^{r-i}\binom{d-i}{r-i}\sum_{m=0}^d S(m,k)m! (f_{m-1}(\CC) - f_{m-1}(\DD)),	\\
&= \sum_{i=0}^r(-1)^{r-i}\binom{d-i}{r-i}\sum_{m=0}^d S(m,k)m! f_{m-1}(\CC\setminus\DD).	\\
\end{split}
\end{equation*}
Since we have assumed that $\CC\setminus\DD$ is $(d-1)$-dimensional, it follows that
\[
f_{m-1}(\CC\setminus\DD) = \sum_{\ell=0}^m\binom{d-\ell}{d-m}h_\ell(\CC\setminus\DD),
\]
and so
\begin{equation*}
\begin{split}
h_r(\sd(\CC)\setminus\sd(\DD)) 
&= \sum_{i=0}^r(-1)^{r-i}\binom{d-i}{r-i}\sum_{m=0}^d S(m,k)m! \sum_{\ell=0}^m\binom{d-\ell}{d-m}h_\ell(\CC\setminus\DD),	\\
&= \sum_{\ell=0}^d\left(\sum_{m=0}^d \sum_{i=0}^r(-1)^{r-i}\binom{d-i}{r-i}\binom{d-\ell}{d-m}S(m,k)m!\right)h_\ell(\CC\setminus\DD).	\\
\end{split}
\end{equation*}
In the proof of \cite[Theorem 1]{BW08}, it is shown that the coefficient of $h_\ell(\CC\setminus\DD)$ in the above expression is equal to the $k^{th}$ coefficient of $A_{d,\ell}^{(1)}$.  
Hence, 
\[
h(\sd(\CC)\setminus\sd(\DD);x) = \sum_{\ell=0}^dh_\ell(\CC\setminus\DD)A_{d,\ell}^{(1)}.
\]
\end{proof}

From Lemma~\ref{lem: h part open simp} and Lemma~\ref{lem: generalizing BW}, we recover the following proposition:
 \begin{proposition}
 \label{prop: max}
Let $\Delta_{d-1}$ be a $(d-1)$-dimensional simplex, and let $0\leq \ell\leq d$ be the number of facets of $\Delta_{d-1}$ missing in $\Delta_{d-1,\ell}$.  
Then, 
  \[
  h(\sd(\Delta_{d-1,\ell}),x) = A_{d,\ell}^{(1)}.
  \]
 \end{proposition}

 Applying Lemma~\ref{lem: colored eulerian} to Proposition~\ref{prop: max}, we see that the $h$-polynomials of the barycentric subdivision of a simplex restricted to half-open simplices form an interlacing family.  Since a shelling of the boundary of a simplicial polytope will decompose this complex into such half-open simplices, the $h$-polynomial of the barycentric subdivision of the complex will be real-rooted.
We summarize this observation in the following theorem, which is originally due to Brenti and Welker \cite{BW08}.
\begin{theorem}
\label{thm: BW reproof}
Let $\CC$ be a $(d-1)$-dimensional shellable simplicial complex.  
Then $h(\sd(\CC);x)$ is real-rooted.  
In particular, the $h$-polynomial of the barycentric subdivision of the boundary complex of a $d$-dimensional simplicial polytope is real-rooted.  
\end{theorem}

\begin{proof}
The result is an immediate consequence of Theorem~\ref{thm: subdivisions of shellable complexes}, Lemma~\ref{lem: colored eulerian}, and Proposition~\ref{prop: max}.  
The special case of boundary complexes of simplicial polytopes follows from the fact that the boundary complex of any polytope admits a shelling \cite{BM72}.
\end{proof}

The proofs of Theorems~\ref{thm: shellable cubical complexes} and~\ref{thm: BW reproof} given here suggest that stable shellability of all boundary complexes of polytopes could be key to answering Problem~\ref{prob: subdivisions of polytopes} in its fullest generality.
In Section~\ref{sec: stable line shellings}, we offer some first results in this direction and pose some related open questions. 
However, we first examine some applications of Theorem~\ref{thm: subdivisions of shellable complexes} and stable shellings to subdivisions other than the barycentric subdivision.  

\subsection{Edgewise subdivisions of simplicial complexes}
\label{subsec: edgewise simplicial}
The edgewise subdivision of a simplicial complex is another well-studied subdivision that arises frequently in algebraic and topological contexts (see for instance \cite{BW09,BruR05,EG00,G89}).  
Within algebra, it is intimately tied to the Veronese construction, and it is considered to be the algebraic analogue of barycentric subdivision \cite[Acknowledgements]{BW09}.  
For $r\geq1$, the $r^{th}$ edgewise subdivision of a simplex is defined as follows:
Suppose that $\Delta:=\conv(e^{(1)},\ldots,e^{(d)})\subset\R^d$ is a $(d-1)$-dimensional simplex with $0$-dimensional faces $e^{(1)},\ldots,e^{(d)}$, the standard basis vectors in $\R^d$.  
For $x = (x_1,\ldots,x_d)\in\Z^d$, we let
\[
\supp(x):=\{i\in[d]:x_i\neq0\},
\]
and we define the linear transformation $\iota:\R^d\rightarrow\R^d$ by
\[
\iota: x\longmapsto (x_1,x_1+x_2,\ldots,x_1+\cdots+x_d).
\]
The {\em $r^{th}$ edgewise subdivision} of $\Delta$ is the simplicial complex $\Delta^{\langle r\rangle}$ whose set of $0$-dimensional faces are the lattice points in $r\Delta\cap\Z^d$ and for which $F\subset r\Delta\cap\Z^d$ is a face of $\Delta^{\langle r\rangle}$ if and only if 
\[
\bigcup_{x\in F}\{\supp(x)\}\in\Delta,
\]
and for all $x,y\in F$ either
$
\iota(x)-\iota(y)\in\{0,1\}^d
$
or
$
\iota(y)-\iota(x)\in\{0,1\}^d.
$
Given a simplicial complex $\CC$, the {\em $r^{th}$ edgewise subdivision} of $\CC$, denoted $\CC^{\langle r\rangle}$, is given by gluing together the $r^{th}$ edgewise subdivisions of each of its facets.
In \cite{MW13}, the authors proposed the following problem.  
\begin{problem}
\cite[Problem 27]{MW13}
\label{prob: edgewise subdivisions}
If $\CC$ is a $d$-dimensional simplicial complex with $h_k(\CC)\geq0$ for all $0\leq k\leq d+1$, is $h(\CC^{\langle r\rangle};x)$ real-rooted whenever $r>d$?
\end{problem}
Applying Theorem~\ref{thm: subdivisions of shellable complexes} and Proposition~\ref{prop: stable shellings of simplicial complexes}, we will give a positive answer to Problem~\ref{prob: edgewise subdivisions} for shellable simplicial complexes via geometric methods.  
We then also observe that a positive answer to Problem~\ref{prob: edgewise subdivisions} for both shellable and non-shellable complexes follows from some recent enumerative results of Jochemko \cite{J16}.  
To do this, we first note that for every $r\geq1$ and polynomial $p\in\R[x]$ there are uniquely determined polynomials $p^{(0)},p^{(1)},\ldots,p^{(r-1)}\in\R[x]$ satisfying
\[
p = p^{(0)}(x^r)+xp^{(1)}(x^r)+x^2p^{(2)}(x^r)+\cdots+x^{r-1}p^{(r-1)}(x).
\]
We define the linear operator
\[
\,^{\langle r,\ell\rangle}:\R[x]\longrightarrow\R[x]
\quad
\mbox{ where}
\quad
\,^{\langle r,\ell\rangle}:p\longrightarrow p^{(\ell)},
\]
and the polynomial
$
p_{(r,d)}:=(1+x+\cdots+x^{r-1})^d.
$
It is well-known that the sequence
\begin{equation}
\label{eqn: interlacing sequence}
\left(p_{(r,d)}^{\langle r,r-\ell\rangle}\right)_{\ell=1}^r = \left(p_{(r,d)}^{\langle r,r-1\rangle},p_{(r,d)}^{\langle r,r-2\rangle},\ldots,p_{(r,d)}^{\langle r,0\rangle}\right),
\end{equation}
is an interlacing sequence (see \cite[Remark 4.2]{S19} or \cite{J16}, for instance).
On the other hand, \cite[Equation 21]{A14} shows that for any $d$-dimensional simplicial complex $\CC$
\begin{equation}
\label{eqn: edgewise h}
h(\CC^{\langle r\rangle};x) = ((1+x+x^2+\cdots+x^{r-1})^{d+1}h(\Delta;x))^{\langle r,0\rangle}
\end{equation}
for all $r\geq 1$.
Let $\Delta_d$ denote the $d$-dimensional simplex, and let $\Delta_{d,\ell}$ denote the relative simplicial complex given by removing $\ell$ of the facets of $\Delta_d$ for $0\leq \ell \leq d+1$.  
\begin{lemma}
\label{lem: half-open edgewise}
Let $d\geq1$, $r> d$, and $0<\ell\leq d+1$.  
Then
\[
h(\Delta_{d,\ell}^{\langle r\rangle};x) = xp_{(r,d+1)}^{\langle r,r-\ell\rangle}.
\]
\end{lemma}

\begin{proof}
Notice first that the relative complex $\Delta_{d,\ell}^{\langle r\rangle}$ can be constructed in two equivalent ways: 
Either we first remove the $\ell$ facets of $\Delta_d$ and then apply the subdivision procedure outlined in the definition of the edgewise subdivision to the corresponding geometric realization of the half-open simplex $\Delta_{d,\ell}$, or we first compute $\Delta_d^{\langle r\rangle}$ and then remove the faces of $\Delta_d^{\langle r\rangle}$ lying in the $\ell$ facets of $\Delta_d$ scheduled for removal. 
For the purposes of this proof, we work with the latter construction.  
Our first goal, then, is to prove the following fact in analogy to equation~\eqref{eqn: edgewise h}:
\[
h(\Delta_{d,\ell}^{\langle r\rangle};x) = ((1+x+\cdots+x^{r-1})^{d+1}h(\Delta_{d,\ell};x))^{\langle r,0\rangle}.
\]

Given the chosen construction of $\Delta_{d,\ell}^{\langle r\rangle}$, we know that
\[
f(\Delta_{d,\ell}^{\langle r\rangle};x) = \sum_{j=0}^\ell(-1)^j\binom{\ell}{j}f(\Delta_{d-j}^{\langle r\rangle};x),
\] 
and so 
\begin{equation*}
\begin{split}
h(\Delta_{d,\ell}^{\langle r\rangle};x) 
&= (1-x)^{d+1}\sum_{j=0}^\ell(-1)^j\binom{\ell}{j}f\left(\Delta_{d-j}^{\langle r\rangle};\frac{x}{1-x}\right),		\\
&= \sum_{j=0}^\ell(-1)^j\binom{\ell}{j}(1-x)^jh(\Delta_{d-j}^{\langle r\rangle};x).						\\
\end{split}
\end{equation*}
Since $\Delta_{d-j}$ is a simplicial complex, it follows from equation~\eqref{eqn: edgewise h} that
\[
h(\Delta_{d-j}^{\langle r\rangle};x) = ((1+x+\cdots+x^{r-1})^{d+1-j}h(\Delta_{d-j};x))^{\langle r,0\rangle}.
\]
Since
\begin{equation*}
\begin{split}
(1-x)^j((1+\cdots+x^{r-1})^{d+1-j}&h(\Delta_{d-j};x))^{\langle r,0\rangle} \\
&= ((1-x^r)^j(1+\cdots+x^{r-1})^{d+1-j}h(\Delta_{d-j};x))^{\langle r,0\rangle},\\
\end{split}
\end{equation*}
it follows that
\begin{equation*}
\begin{split}
h(\Delta_{d,\ell}^{\langle r\rangle};x)
&= \sum_{j=0}^\ell(-1)^j\binom{\ell}{j}((1-x^r)^j(1+\cdots+x^{r-1})^{d+1-j}h(\Delta_{d-j};x))^{\langle r,0\rangle},	\\
&= \left((1+\cdots+x^{r-1})^{d+1}\left(\sum_{j=0}^\ell(-1)^j\binom{\ell}{j}(1-x)^jh(\Delta_{d-j};x)\right)\right)^{\langle r,0\rangle},	\\
&= \left((1+\cdots+x^{r-1})^{d+1}h(\Delta_{d,\ell};x)\right)^{\langle r,0\rangle},	\\
\end{split}
\end{equation*}
as desired.
We then note that $h(\Delta_{d,\ell};x) = x^\ell$, by Lemma~\ref{lem: h part open simp}.
Since $r>d$ and $0< \ell\leq d+1$, it follows that
\[
h(\Delta_{d,\ell}^{\langle r\rangle};x) = xp_{(r,d+1)}^{\langle r,r-\ell\rangle},
\]
which completes the proof.
\end{proof}

Lemma~\ref{lem: half-open edgewise} gives the necessary tools to positively answer Problem~\ref{prob: edgewise subdivisions} for shellable simplicial complexes.
\begin{theorem}
\label{thm: edgewise shellable}
Let $\CC$ be a $d$-dimensional shellable simplicial complex, and let $r>d$.
Then $h(\CC^{\langle r\rangle};x)$ is real-rooted.
\end{theorem}

\begin{proof}
Since $\CC$ is a shellable polytopal complex, then we can write $\CC^{\langle r\rangle}$ as a disjoint union of relative simplicial complexes $\RR_i$, one for each facet in a shelling order $(F_1,\ldots,F_s)$ of $\CC$, such that
\[
h(\CC^{\langle r\rangle};x) = \sum_{i=1}^sh(\RR_i;x),
\]
as in the hypothesis of Theorem~\ref{thm: subdivisions of shellable complexes}.
Since each $\RR_i$ is equal to $\Delta_{d,\ell}^{\langle r\rangle}$ for some $0\leq \ell\leq d+1$, then, by Lemma~\ref{lem: half-open edgewise} and equation~\eqref{eqn: edgewise h}, $h(\CC^{\langle r\rangle};x)$ is a convex combination of the polynomials in the sequence
\[
\left(p_{(r,d+1)}^{\langle r,0\rangle},xp_{(r,d+1)}^{\langle r,r-1\rangle},xp_{(r,d+1)}^{\langle r,r-2\rangle},\ldots,xp_{(r,d+1)}^{\langle r,1\rangle},xp_{(r,d+1)}^{\langle r,0\rangle}\right).
\]
Since the sequence~\eqref{eqn: interlacing sequence} is interlacing, it follows from Lemma~\ref{lem: basic facts} that this sequence is also interlacing.
By Theorem~\ref{thm: subdivisions of shellable complexes}, we conclude that $h(\CC^{\langle r\rangle};x)$ is real-rooted.
\end{proof}

As an immediate corollary to Theorem~\ref{thm: edgewise shellable} we get that, for $r\geq d$, the $r^{th}$ edgewise subdivision of the boundary complex of any $d$-dimensional simplicial polytope has a real-rooted $h$-polynomial.  
\begin{corollary}
\label{cor: edgewise polytopes}
Let $\CC$ be the boundary complex of a $d$-dimensional simplicial polytope.  
Then for $r\geq d$, the edgewise subdivision $\CC^{\langle r\rangle}$ of $\CC$ has a real-rooted $h$-polynomial. 
\end{corollary}
On the other hand, Theorem~\ref{thm: edgewise shellable} holds more generally.
This, in fact, follows directly from some recent results of Jochemko \cite{J16}. 
\begin{theorem}[Essentially due to \cite{J16}]
\label{thm: edgewise}
If $\CC$ is a $d$-dimensional simplicial complex with $h_k(\CC)\geq 0$ for all $0\leq k\leq d+1$ then $h(\CC^{\langle r\rangle};x)$ is real-rooted whenever $r>d$.
\end{theorem}

\begin{proof}
The proof is given by combining an observation of Athanasiadis in \cite{A14} with some recent results of Jochemko \cite{J16}. 
Combining \cite[Theorem 1.1]{J16} with \cite[Lemma 3.1]{J16} for $i = 0$, we see that the polynomial
$
((1+x+x^2+\cdots+x^{r-1})^{d+1}p)^{\langle r,0\rangle}
$
has only real zeros whenever $p$ is degree ${d+1}$ with only nonnegative coefficients and $r> d$.  
On the other hand, it follows from equation~\eqref{eqn: edgewise h} that
\[
h(\CC^{\langle r\rangle};x) = ((1+x+x^2+\cdots+x^{r-1})^{d+1}h(\CC;x))^{\langle r,0\rangle}
\]
for all $r\geq 1$.
The result follows.  
\end{proof}

Theorem~\ref{thm: edgewise} shows that the geometric approach used in Theorem~\ref{thm: edgewise shellable} was not necessary, as it was for the solution to Problem~\ref{prob: subdivisions of polytopes} for cuboids, capped cubical polytopes, and neighborly cubical polytopes given in Subsection~\ref{subsec: barycentric cubical}.
On the other hand, the geometric proof of Theorem~\ref{thm: edgewise shellable} highlights that the applications of Theorem~\ref{thm: subdivisions of shellable complexes} are not limited to barycentric subdivisions. 
In the next subsection, a similar result for the edgewise subdivision of a cube is derived.  
In this case, there is currently no other proof aside from the geometric methods developed in this paper.

\subsection{Edgewise subdivisions of cubical complexes}
\label{subsec: edgewise cubical}
In Subsection~\ref{subsec: edgewise simplicial} we defined the edgewise subdivision of a simplicial complex.  
We now extend this definition to cubical complexes. 
To do so, we perform the same operations on a unit cube that were performed on the standard simplex $\Delta$ in the construction of the $r^{th}$ edgewise subdivision of a simplex. 
To reiterate, let $\square_d$ denote the (abstract) $d$-dimensional cube, and consider its geometric realization $[0,1]^d$ and $r[0,1]^d = [0,r]^d$, the $r^{th}$ dilation of $[0,1]^d$.   
Recall the map $\iota$ defined in Subsection~\ref{subsec: edgewise simplicial} that sends $(x_1, \dots, x_d) \in C_d \cap \Z^d$ to $(x_1, x_1 + x_2, \dots, x_1 + \dots + x_d)$.  
We define the {\em $r^{th}$ edgewise subdivision} $\square_{d}^{\langle r\rangle}$ of the $d$-dimensional cube in terms of a subdivision of its geometric realization $[0,1]^d$ as follows:  
Let $A \subset [0,1]^d \cap \Z^d$. 
Then $\conv(A)$ is a face of the subdivision if and only if $\iota(v-v')$ or $-\iota(v-v')$ is in $\{0,1\}^d$ for all $v,v' \in A$.  
We first note that this a unimodular triangulation of $[0,1]^d$, as it splits the dilated cube $[0,r]^d$ into unit cubes which are each triangulated according to (a rotated version of) the standard unimodular triangulation of $[0,1]^d$; that is, the triangulation induced by the hyperplanes $x_i = x_j$ for all $1\leq i<j\leq d$.
Given a cubical complex $\CC$, its {\em $r^{th}$ edgewise subdivision}, denoted $\CC^{\langle r \rangle}$ is given by gluing together the $r^{th}$ edgewise subdivisions of each of its facets. 

Stable shellings were defined so as to capture those shellings to which Theorem~\ref{thm: subdivisions of shellable complexes} can be applied for multiple different subdivisions.  
In Subsections~\ref{subsec: simplicial polytopes} and~\ref{subsec: edgewise simplicial}, we saw this to be the case for simplicial complexes.  
To further substantiate this claim, we now show that the analogous result to Theorem~\ref{thm: shellable cubical complexes} holds for the edgewise subdivision of a cubical complex.  
To do so, we will first use the fact that the $r^{th}$ edgewise subdivision of a cube has a geometric realization that is a unimodular triangulation of the $r^{th}$ dilation of $[0,1]^d$ so as to give a formula for the $h$-polynomials of the stable relative complexes $\RR_i$ associated to the edgewise subdivision of a cubical complex.  
This formula will be in terms of the colored Eulerian polynomials $A_{d,\ell}^{(r)}$, which were introduced at the beginning of Section~\ref{sec: apps}.

We first give a combinatorial interpretation of the coefficients of $A_{d,\ell}^{(r)}$ in terms of a descent statistic for the wreath product $\Z_r\wr\mathfrak{S}_d$. 
Denote the elements of the wreath product $\Z_r\wr\mathfrak{S}_d$ as pairs $(\pi, \epsilon)$, where $\pi = \pi_1\cdots\pi_d \in \mathfrak{S}_d$ and $\epsilon = (c_1,\ldots,c_d)\in\{0,\ldots,r-1\}^d$.  
We will typically denote the pair $(\pi,\epsilon)$ as $\pi_1^{c_1}\cdots\pi_d^{c_d}$.  
Like the classical symmetric group, $\Z_r\wr\mathfrak{S}_d$ admits combinatorial statistics such as descents and excedances, and there exist several different well-studied versions of each.  
A brief survey of these different definitions, as well as their uses, can be found in \cite{BB13}.  
For our purposes, we use the following definition:
\begin{definition}
\cite[Definition 2.5]{BB13}
\label{def: colored descents}
Order the elements in the set $\{i^{c_i} : i\in[d], c_i \in \{0,1,\dots,r-1\}\}$ such that $i^{c_i}<j^{c_j}$ if $c_i>c_j$ or $c_i=c_j$ and $i<j$.  
For $(\pi,\epsilon) \in \Z_r\wr S_d$, the {\em descent set} of $(\pi,\epsilon)$ is
\[
\Des(\pi,\epsilon):= \{j \in \{0,1,\dots,d-1\}: \pi_j^{c_j} >\pi_{j+1}^{c_{j+1}} \},
\]
where we use the convention that $\pi_0 = 0$ and $c_0 = 0$. 
The {\em descent statistic} is $\des(\pi,\epsilon) := |Des(\pi,\epsilon)|$.  
\end{definition}

As in the case of the signed permutations used in Subsection~\ref{subsec: barycentric cubical}, we can extend this Definition~\ref{def: colored descents} to $\ell$-descents:
Let $(\pi, \epsilon) \in \Z_r\wr\mathfrak{S}_d$ and let $0\leq\ell\leq d$.  
Then the {\em $\ell$-descent set} of $(\pi, \epsilon)$ is 
\[
\Des_\ell(\pi, \epsilon) := 
\begin{cases} 
\Des(\pi,\epsilon)\cup\{0\} &\mbox{if } \pi_1 \in [\ell]\\
\Des(\pi,\epsilon) & \mbox{otherwise.} 
\end{cases}
 \]
The {\em $\ell$-descent statistic} is then $\des_\ell (\pi,\epsilon) := |\Des_\ell(\pi,\epsilon)|$.
Using this new statistic, we can now give a combinatorial interpretation of $A_{d,\ell}^{(r)}$. 
\begin{proposition}
\label{prop: eulerian comb int}
For $d,r\geq1$ and $0\leq \ell\leq d$,
\[
A_{d,\ell}^{(r)} = \sum_{(\pi,\epsilon) \in \Z_r\wr\mathfrak{S}_d} x^{\des_\ell(\pi, \epsilon)}.
\]
Moreover, $A_{d,\ell}^{(r)} = h^\ast([0,r]^d_\ell;x)$.  
\end{proposition}

\begin{proof}
We prove this by using the following unimodular triangulation of the cube $[0,r]^d$: 
Subdivide $[0,r]^d$ into the pile of (unit) cubes $\mathcal{P}_d(r,\ldots,r)$, and triangulate each cube in the resulting cubical complex according to the standard triangulation of $[0,1]^d$.
That is, the triangulation of the cube $C_\mathbf{z}$, for $\mathbf{z} = (z_1,\ldots,z_d)\in B^\circ(r,r,\ldots,r)$ is given by triangulating $[0,1]^d$ via the hyperplanes $x_i = x_j$ for $1\leq i<j\leq d$ and then translating this triangulated version of $[0,1]^d$ as $[0,1]^d+\mathbf{z} = C_\mathbf{z}$.  
Each simplex in this triangulation of $C_\mathbf{z}$ is then of the form 
\[
\Delta_{(\pi, \epsilon)}^d := \{x\in\R^d : 0 \leq x_{\pi_1} -z_{\pi_1}\leq \dots\leq x_{\pi_d} - z_{\pi_d} \leq 1\},
\] 
for $(\pi,\epsilon) \in \Z_r\wr S_d$ given by $\pi_1^{z_{\pi(1)}} \dots \pi_d^{z_{\pi(d)}}$.  
This correspondence between elements of $\Z_r\wr S_d$ and facets of this triangulation is similar in spirit to the triangulation defined in \cite[Section 3]{S94}.  
However, we use slightly different conventions. 

We now define the following half-open simplices, with $\pi$, $\mathbf{z}$, and $\epsilon$ defined as above:
\[
\Delta_{(\pi,\epsilon)}^{d,\ell} := \left\{\mathbf{x} \in \R^d : \begin{array}{l}
    0 \leq x_{\pi_1} -z_{\pi_1}\leq \dots\leq x_{\pi_d} - z_{\pi_d} \leq 1,  \\
   x_{\pi_i} - z_{\pi_i} < x_{\pi_{i+1}} - z_{\pi_{i+1}} \mbox{ for } i \in \Des(\pi,\epsilon) \mbox{ , and} \\
   0 < x_{\pi_1} - z_{\pi_1} \mbox{ if } 0 \in \Des_\ell(\pi, \epsilon)
  \end{array}\right\}.
\]
  
  Note that $\Delta_{(\pi,\epsilon)}^{d,\ell}$ is a unimodular half-open simplex with the number of missing facets equal to $\des_\ell(\pi, \epsilon)$.  
  Hence, $h^\ast(\Delta_{(\pi,\epsilon)}^{d,\ell}; x) = x^{\des_\ell(\pi,\epsilon)}$.  
  (A proof of this fact is an exercise analogous to the proof of Lemma~\ref{lem: h part open simp}.)
  We now show that for fixed $d,r\geq 1$ and $0\leq \ell\leq d$, the disjoint union
  \[
  \bigsqcup_{(\pi,\epsilon) \in \Z_r \wr \mathfrak{S}_d}\Delta_{(\pi, \epsilon)}^{d,\ell}
  \]
  is $[0,r]_\ell^d$.  
  To do so, we first decompose $[0,r]_\ell^d$ into a disjoint union of half-open unit cubes. 
  As before, let $\mathbf{z} \in B^\circ(r,\ldots,r)$, and set
  
  \[C_\mathbf{z}^\ell= C_\mathbf{z}\setminus\{x_i = z_i : z_i \neq 0 \mbox{ or } z_i = 0 \mbox{ and } i \in [\ell]\}.\]
  
  It follows, as in Example~\ref{ex: piles of cubes}, that $[0,r]^d_\ell$ is the disjoint union of the half-open cubes $C_\mathbf{z} ^\ell$ for $\mathbf{z}\in B^\circ(r,\ldots,r)$.  
  We now show that for a fixed $\mathbf{z}$, the half-open cube $C_\mathbf{z}^\ell$ is the disjoint union of $\Delta_{(\pi,\epsilon)}^{d,\ell}$ where $\pi$ ranges over all elements of $\mathfrak{S}_d$ and $(\pi,\epsilon)= \pi_1^{z_{\pi_1}}\cdots\pi_d^{z_{\pi_d}}$. 
  First, we observe that the closed cube $C_\mathbf{z}$ is the disjoint union of half-open simplices of the form 
  \[
  \left\{\mathbf{x} \in \R^d : 
  \begin{array}{l}
    0 \leq x_{\pi_1} -z_{\pi_1}\leq \dots\leq x_{\pi_d} - z_{\pi_d} \leq 1,  \\
   x_{\pi_i} - z_{\pi_i} < x_{\pi_{i+1}} - z_{\pi_{i+1}} \mbox{ for } i \in \Des(\pi,\epsilon) \cap  \{1,\dots,d-1\}
   \end{array} 
   \right\},
   \]
   since this is simply a translated version of the standard half-open decomposition of the unit cube described, for example, in \cite{BS18}.  
   From this, we can construct a half-open decomposition of $[0,r]^d_\ell$ by removing all of the points in the simplices described above that satisfy $  x_{\pi_1} = z_{\pi_1} \mbox{ for } 0 \in \Des_\ell(\pi, \epsilon)$.  
   Removing these points gives us the half-open simplex $\Delta_{(\pi,\epsilon)}^{d,\ell}$. 
   Hence, $C_\mathbf{z}^\ell$ is a disjoint union of $\Delta_{(\pi,\epsilon)}^{d,\ell}$ where $(\pi,\epsilon) = \pi_1^{z_{\pi_1}}\cdots\pi_d^{z_{\pi_d}}$. 
   Since $[0,r]^d_\ell$ is the disjoint union of all $C_\mathbf{z}^\ell$ for $z\in B^\circ(r,\ldots,r)$, it follows that $[0,r]^d_\ell$ is the disjoint union of $\Delta_{(\pi,\epsilon)}^{d,\ell}$ over all $(\pi,\epsilon)\in\Z_r\wr\mathfrak{S}_d$.
   
   We now compute the $h^*$-polynomial of $[0,r]_\ell^d$ in two different ways.  
   First, since $[0,r]^d_\ell$ is the disjoint union of $\Delta_{(\pi,\epsilon)}^{d,\ell}$ over all $(\pi,\epsilon)\in\Z_r\wr\mathfrak{S}_d$, it follows that
   \[
   \Ehr([0,r]_\ell^d;x) = \sum_{(\pi,\epsilon) \in \Z_r \wr S_d} h^*(\Delta_{(\pi, \epsilon)}^{d,\ell};x) = \sum_{(\pi,\epsilon) \in \Z_r \wr S_d} x^{\des_{\ell}(\pi,\epsilon)},
   \]
   where the last equality follows from the fact that the $h^*$-polynomial of a unimodular simplex with $m$ facets removed is $x^m$.  
   On the other hand, since $\ell \leq d$, we know that $[0,r]^d_\ell$ is the product of $\ell$ copies of the half-open $1$-dimensional cube $[0,r)$ and $d-\ell$ copies of the $1$-dimensional cube $[0,r]$.  
   Since Ehrhart polynomials are multiplicative, 
   
   it follows that 
\[
i([0,r]^d_\ell;t) = i([0,r);t)^\ell i([0,r];t)^{d-\ell}) = (rt)^\ell(rt+1)^{d-\ell}. 
\]
Thus,
\[
\Ehr_{[0,r]^d_\ell}(x) = \sum_{t\geq 0}  (rt)^\ell(rt+1)^{d-\ell}x^t  = \frac{h^*([0,r]^d_\ell;x)}{(1-x)^{d+1}}.\]
From the definition of $r$-colored $\ell$-Eulerian polynomials given in Equation~\eqref{eqn: colored}, we see that $h^*([0,r]^d_\ell;x) =A_{d,\ell}^{(r)}.$  
Thus, $A_{d,\ell}^{(r)} = \sum_{(\pi,\epsilon) \in \Z_r \wr S_d} x^{\des(\pi,\epsilon)}$, as desired. 
\end{proof}

\begin{corollary}
\label{cor: degree eulerian}
 For all $d\geq 1$, if $r = 1$ and $0\leq \ell < d$, then the degree of $A_{d,\ell}^{(r)}$ is $d-1$.  Otherwise, the degree of $A_{d,\ell}^{(r)}$ is $d$.  
\end{corollary}

\begin{proof}
By Proposition~\ref{prop: eulerian comb int}, the degree of  $A_{d,\ell}^{(r)}$ is the maximum number of $\ell$-descents of $(\pi, \epsilon) \in \Z_r \wr S_d$.  Since $\Des_\ell(\pi,\epsilon)$ is a subset of $\{0, \dots, d-1\}$, the degree is at most $d$.  
 When $r = 1$, $0\in\Des_\ell(\pi,\epsilon)$ if and only if $\ell >0$ and $\pi_1\leq \ell$.  If $0\in\Des_\ell(\pi,\epsilon)$ and $r = 1$, then $i\in\Des_\ell(\pi,\epsilon)$ for all $i\in[d-1]$ if and only if $\pi_i > \pi_{i+1}$ for all $i\in[d-1]$.  However, this is only possible if $\pi = d(d-1)\cdots 1$.  Hence, $A_{d,\ell}^{(1)}$ has degree $d$ if and only if $\ell = d$.  Otherwise, $A_{d,\ell}^{(1)}$ has degree $d-1$, since $(\pi,\varepsilon) = \pi_1^0\pi_2^0\cdots\pi_d^0$ has exactly $d-1$ descents.  
If $r>1$ then $A_{d,\ell}^{(r)}$ has degree $d$ since the element $(\pi,\varepsilon) = d^1(d-1)^1\cdots1^1$ has $d$ descents.
\end{proof}


Using these facts, we can apply the results on stable shellings developed in Section~\ref{sec: shellable complexes} to prove the following:
\begin{theorem}
\label{thm: cubical edgewise real}
Let $\CC$ be a cubical complex with a stable shelling. 
Then $h(\mathcal{C}^{\langle r \rangle};x)$ is real rooted for $r\geq2$.
\end{theorem} 

\begin{proof}
Using the  shelling argument detailed in the proof of Theorem \ref{thm: shellable cubical complexes}, we see that we can decompose $\mathcal{C}^{\langle r \rangle}$ into relative complexes $\mathcal R_j$ in which each $\mathcal R_j$ has geometric realization the unimodular triangulation $T_{d,r}$ of $[0,r]^d_\ell$, induced by the $r^{th}$ edgewise subdivision as described above, for some $0\leq\ell\leq 2d$.   
We first consider the cases in which $0\leq \ell \leq d$.  
In this case, $[0,r]^d_\ell$ is either a closed, convex polytope and hence has Euler characteristic $1$, or it is missing a subset of facets that forms a contractible subcomplex, and hence it has Euler characteristc $0$.  
So by Lemmas \ref{lem: whole polytope} and \ref{lem: not whole polytope},
\[
h(\mathcal{R}_j;x) = h^*([0,r]^d_l;x)=A_{d,\ell}^{(r)}.
\]
where the last equality follows from Proposition~\ref{prop: eulerian comb int}.  

We now must consider the case in which more than $d$ facets are removed.  From Lemma \ref{lem: half-open reciprocity}, we know that for $0\leq\ell\leq d$:
\[
h^*([0,r]^d_{2d -\ell};x ) = \mathcal{I}_{d+1}h^*([0,r]^d_\ell;x)  = x\mathcal{I}_d A_{d,\ell}^{(r)},
\]
where the last equality follows from Corollary \ref{cor: degree eulerian} stating that the degree of $A_{d,\ell}^{(r)}$ is $d$.  
Thus, we see that $h(\mathcal{R}_j;x)$ is either $A_{d,\ell}^{(r)}$ or $x\mathcal{I}_d A_{d,\ell}^{(r)}$ for some $0\leq\ell\leq d$. 
Thus it suffices to show that for a fixed $r\geq2$ and $d\geq1$ the concatenated sequence 
\[
\left((A_{d,\ell}^{(r)})_{\ell = 0}^d, (x\mathcal{I}_d A_{d,\ell}^{(r)})_{\ell = d}^0\right)
\] 
is interlacing.  
The justification is identical to the one found in the proof of Theorem \ref{thm: shellable cubical complexes}.
\end{proof}

In the proofs of Corollary~\ref{cor: cuboids}, Corollary~\ref{cor: capped real-rootedness}, and Corollary~\ref{cor: neighborly}, it was argued that the boundary complexes of cuboids, capped cubical polytopes, and neighborly cubical polytopes all admit stable shellings.  
Hence, as a corollary to Theorem~\ref{thm: cubical edgewise real}, we recover that the $r^{th}$ edgewise subdivision of the boundary complex of all well-known examples of cubical polytopes have real-rooted $h$-polynomials. 
\begin{corollary}
\label{cor: cubical examples edgewise}
If $\CC$ is the boundary complex of a cuboid, a capped cubical polytope, or a neighborly cubical polytope, then $h(\CC^{\langle r\rangle};x)$ is real-rooted for all $r\geq 2$.  
\end{corollary}
The results of Corollary~\ref{cor: cubical examples edgewise}, Corollary~\ref{cor: cuboids}, Corollary~\ref{cor: capped real-rootedness}, and Corollary~\ref{cor: neighborly} collectively show how the same stable shelling allows one to deduce the real-rootedness of $h$-polynomials for a variety of different uniform subdivisions.

\section{Stable Line Shellings}
\label{sec: stable line shellings}
One of our applications of stable shellings in Section~\ref{sec: apps} was that any cubical complex admitting such a shelling has a barycentric subdivision with a real-rooted $h$-polynomial (Theorem~\ref{thm: shellable cubical complexes}).
This gave a positive answer to Problem~\ref{prob: subdivisions of polytopes} for all well-known constructions of cubical polytopes. 
A classic result of Bruggesser and Mani \cite{BM72} states that, in fact, the boundary complex of any polytope admits a shelling. 
Hence, the result of Theorem~\ref{thm: shellable cubical complexes} suggests the following general question:
\begin{question}
\label{quest: stable shellings of polytopes}
Does every polytope (cubical or otherwise) admit a stable shelling?
\end{question}

The result of Bruggesser and Mani \cite{BM72} is, in fact, stronger than stated since they further demonstrate that the boundary complex of any polytope admits a special type of shelling known as a line shelling. 
Suppose that $\CC$ is the boundary complex of a $d$-dimensional polytope $P$ with $s$ facets and $m$ vertices.  
The \emph{realization space} of $\CC$, denoted $R_\CC$, is the space of all geometric realizations $\Sigma\subset\R^d$ of $\CC$ as the boundary complex of a convex polytope in $\R^d$.  
Hence, each realization $\Sigma\in R_\CC$ is the boundary complex of a convex polytope in $\R^d$, which we will denote by $Q_\Sigma$.  
Note that $R_\CC$ can be thought of as a semialgebraic set living in $\R^{d\times m}$ where each realization $\Sigma$ corresponds to a $d\times m$ matrix whose columns are the realizations of the vertices of $\CC$. 
Since $\CC$ is $(d-1)$-dimensional, then for all $\Sigma\in R_\CC$, the geometric realization $\sigma_i\in\Sigma$ of a given facet $F_i$ of $\CC$ spans an affine hyperplane $H_i\subset\R^d$.  
Let $\A_\Sigma := \{H_i : i\in[s]\}$ denote the corresponding hyperplane arrangement.  
Fix $\Sigma\in R_\CC$, and let $\ell\subset\R^d$ be a line that intersects each hyperplane $H_i\in\A_\Sigma$ at a point $q_i:=\ell\cap H_i$.  
Assume that $q_1,\ldots,q_s$ are all distinct and that $\ell$ intersects the interior of the convex polytope $Q_\Sigma$.  
Without loss of generality, the point $q_1$ then lies in the interior of some facet $\sigma_1$ of $\Sigma$.  
Consider $\A_\Sigma$ and $\ell$ in the one-point compactification of $\R^d$, denoted $\R^d\cup\{\infty\}$.  
By fixing an orientation of $\ell$ and following this orientation outwards from the initial point $q_1$, we obtain a linear ordering of the points of intersection of $\ell$ with $\A_\Sigma$ and the point $\infty$:
\[
(q_1,\ldots,q_t,\infty,q_{t+1},\ldots,q_s).
\]
If the corresponding linear order $(F_1,\ldots,F_t,F_{t+1},\ldots,F_s)$ of the facets of $\CC$ is a shelling, we call it a \emph{line shelling} of $\CC$ (induced by $\A_\Sigma$ and $\ell$).
A well-known fact about line shellings is the following (see, for instance, \cite{Z12} or \cite{BM72}): 
If $i\leq t$, then the set of facets of $F_i$ that are not included in the relative complex $\RR_i$ associated to $F_i$ by $(F_1,\ldots,F_s)$ is the collection of facets $F_i$ realized by facets of $\sigma_i$ that are {\em visible} from the point $q_i$; that is, the set of all facets containing a point $q$ that is visible from $q_i$ in $H_i$ (as defined in Subsection~\ref{subsec: ehrhart theory}).  
On the other hand, if $i>t$, then the set of facets of $F_i$ that are not included in $\RR_i$ are those facets realized by facets of $\sigma_i$ that are not visible from $q_i$ in $H_i$.  
These are callled the {\em covisible facets} of $F_i$.  


Line shellings are a well-studied tool in geometric and algebraic combinatorics (see for instance \cite{Z12}).  
The result of Bruggesser and Mani \cite{BM72} states that the boundary complex of any convex polytope admits a line shelling.
Given this result, a positive answer to Question~\ref{quest: stable shellings of polytopes} that gives a stable line shelling for the boundary complex of every polytope would yield an answer to Problem~\ref{prob: subdivisions of polytopes} in the case of all cubical polytopes (via Theorem~\ref{thm: shellable cubical complexes}).  
Hence, it would be of general interest to better understand the geometry of stable line shellings.  
Namely, given the boundary complex $\CC$ of a $d$-dimensional polytope $P$, we let $\mathcal{L}_\CC$ denote the collection of all pairs $(\Sigma,\ell)$ for which $\Sigma\in R_\CC$ and $\ell\in\R^d$ induce a line shelling of $\CC$.  
We further let $\mathcal{L}_\CC^s\subseteq \mathcal{L}_\CC$ denote the space of all pairs that induce stable line shellings of $\CC$.  
\begin{problem}
\label{prob: stable line shellings}
Let $\CC$ be the boundary complex of a $d$-dimensional polytope $P$.  
Describe the space of stable line shellings $\mathcal{L}_\CC^s$.  
\end{problem}
The geometry of $\mathcal{L}_\CC^s$ is tied to the geometry of realization spaces and hyperplane arrangements, both of which have a long history in algebraic and geometric combinatorics.  
A description of the space $\mathcal{L}_\CC^s$ would likely bring the study of such ideas closer to the theory of interlacing polynomials, and in doing so, could lead to a full answer to Problem~\ref{prob: subdivisions of polytopes} in the case of cubical polytopes.  
To do so, we need to see that the space $\mathcal{L}_\CC^s$ is nonempty whenever $\CC$ is the boundary complex of a cubical polytope.  
This is true in the simplest case; i.e., when $\CC$ is the boundary of the $d$-cube.  
\begin{example}[The boundary of the $d$-cube]
\label{ex: d-cube line}
Consider the geometric realization of the boundary of the $d$-dimensional cube as the boundary of $[0,1]^d\subset\R^d$.  
Let $\ell\subset\R^d$ be a general line passing through the interior of $[0,1]^d$.  
It follows that $\ell$ intersects each facet-defining hyperplane of $[0,1]^d$ at a distinct point, and ordering these points with respect to an orientation of $\ell$ yields a line shelling $(F_1,\ldots,F_{2d})$ of the facets of $[0,1]^d$. 
Let $H_i$ denote the facet-defining hyperplane of $[0,1]^{d}$ containing the facet $F_i$ for all $i\in[2d]$, and consider the point $q_i$ lying in the hyperplane $H_i$ defining the facet $F_i$ for some $i\in[2d]$.  
Without loss of generality, this hyperplane is of the form $x_j = 0$ for some $j\in[d]$.  
By identifying the affine subspace of $\R^d$ defined by the hyperplane $x_j = 0$ with $\R^{d-1}$, we see that the induced arrangement $\A_i := \{H_i\cap H_k : k\in[2d]\setminus\{i\}\}$ lying in the hyperplane $H_i$ is, up to reindexing, the hyperplane arrangement $\A_{[0,1]^{d-1}}$.  
Since $q_i$ does not lie in any facet-defining hyperplane for a facet $F_k$ of $[0,1]^d$ for $k\neq i$, it follows that $q_i$ lies in an open region of $\R^{d-1}\setminus\{x\in\R^{d-1}: x\in H \mbox{ for some $H\in\A_{[0,1]^{d-1}}$}\}$.
Since the arrangement $\A_{[0,1]^{d-1}}$ consists of the hyperplanes $x_j = 0$ and $x_j = 1$ for all $j\in[d-1]$, it follows that this region is either the interior of $[0,1]^{d-1}$ or it is all $x\in\R^{d-1}$ satisfying $x_k < 0$ and $x_{k^\prime} > 1$ for the facet-defining hyperplanes $x_k = 0$ and $x_{k^\prime} =1$ of the tangent cone $T_F([0,1]^{d-1})$ of some face $F$ of $[0,1]^{d-1}$.  
In the former case the set of visible facets of $F_i$ from $q_i$ is empty and the set of covisible facets is the complete set of facets of $F_i$. 
In the latter case, the set of visible facets of $F_i$ from $q_i$ is the set of all facets containing the face $F$, and the set of covisible facets is their complement.  
In either case, the relative complex $\RR_i$ is stable, and hence, so is the line shelling $(F_1,\ldots, F_s)$. 
\end{example}

When attempting to generalize Example~\ref{ex: d-cube line}, we see that a stable line shelling of the boundary complex $\CC$ of a $d$-dimensional polytope $P$ should be a line shelling induced by a hyperplane arrangement $\A_\Sigma$, for $\Sigma\in R_\CC$, and a line $\ell\in\R^d$ for which the points $q_1,\ldots,q_s$ are all `sufficiently close' to the complex $\Sigma$ in the following sense:   
Note that the hyperplane arrangement $\A_\Sigma$ naturally subdivides $\R^d$ into a collection of disjoint connected components.  
A \emph{region} of the arrangement $\A_\Sigma$ is a connected component of the complement of the hyperplanes in $\A_\Sigma$:
\[
\R^d\setminus\bigcup_{H\in\A_\Sigma}H.
\]
We let $\mathfrak{R}(H_\Sigma)$ denote the collection of all regions of the $\HH_\Sigma$.  
Similarly, given any subset $Y\subset[s]$, we can consider the hyperplane arrangement $\A_\Sigma^Y$ living in the real-Euclidean space $\bigcap_{i\in Y}H_i \subset\R^d$ given by
\[
\A_\Sigma^Y :=
\{
H_j\cap H_i : j\in[s]\setminus Y
\}.
\]
It follows that $\A_\Sigma$ subdivides $\R^d$ into the disjoint, connected components
\[
\comp(\A_\Sigma) 
:= 
\bigcup_{Y\subset[s]}\mathfrak{R}(\A_\Sigma^Y).
\]
Since $\CC$ is the boundary complex of a convex polytope, then for each facet $F_i$ of $\CC$, there is a subset of hyperplanes $Y_i\subset \A_\Sigma$ such that for all $H\in Y_i$, the intersection $H\cap H_i$ is a facet-defining hyperplane in the affine subspace $H_i$ of the realization $\sigma_i$ of the facet $F_i$ of $\CC$.  
Let $\A_i$ denote the hyperplane arrangement $\{H\cap H_i : H\in Y_i\}\subset H$.  
Since $q_i\in H_i$, it follows that $q_i$ is contained in a unique region $C^{(i)}\in\mathfrak{R}(A_i)$.
The following proposition notes that a line shelling of $\CC$ in which the points $q_i$ are sufficiently close to $\Sigma$ with respect to the (combinatorial) geometry of the hyperplane arrangement $\A_\Sigma$ will be stable.  
\begin{proposition}
\label{prop: stable line shellings}
Let $\CC$ be the boundary complex of a $d$-dimensional polytope $P$ and let $(F_1,\ldots,F_s)$ be a line shelling of $\CC$ induced by $\A_\Sigma$ and a line $\ell\in\R^d$.  
Then $(F_1,\ldots, F_s)$ is stable if, for all $i\in[s]$, the closure of the region $C^{(i)}\in\mathfrak{R}(\A_i)$ containing the point $q_i$ also contains a point of the geometric realization $\sigma_i$ of $F_i$. 
\end{proposition}

\begin{proof}
We work directly with the shelling $(\sigma_1,\ldots,\sigma_s)$ of the geometric realization $\Sigma$ of $\CC$.  
Suppose that $(\sigma_1,\ldots,\sigma_s)$ is a line shelling induced by $\ell$ such that, for all $i\in[s]$, the closure of the region $C^{(i)}\in\mathfrak{R}(\A_i)$ containing the point $q_i$ also contains a point of $\sigma_i$. 
Notice first that, since $\Sigma$ is the boundary complex of a convex polytope $Q_\Sigma$, the facet-defining hyperplanes of $\sigma_i$ are given by a  subset $\{H_1,\ldots, H_M\}$ of the hyperplanes in $\A_\Sigma\setminus \{H_i\}$ in the sense that
\[
\A_i = \{G_1 := H_i\cap H_1,\ldots, G_M := H_i\cap H_M\}.
\]
Moreover, since $H_i \simeq \R^{d-1}$, each hyperplane $G\subset \A_i$ consists of the set of solutions $x\in\R^{d-1}$ to a linear equation $\langle a_{G},x\rangle = b_{G}$ for some $a_{G}\in\R^{d-1}$ and $b_{G}\in\R$.  

Since $(\sigma_1,\ldots,\sigma_s)$ is a line shelling it is induced by the ordering 
\[
(q_1,\ldots,q_t,\infty,q_{t+1},\ldots,q_s)
\]
 of points in $\R^d\cup\{\infty\}$.
Let $\vis(q_i)$ denote the set of facet-defining hyperplanes of $\sigma_i$ that define facets visible from $q_i$, and let $\covis(q_i)$ denote the set of facet-defining hyperplanes of all covisible facets from $q_i$. 

As noted in Subsection~\ref{subsec: ehrhart theory}, a point $q$ lying in a facet $\sigma$ of $\sigma_i$ defined by the hyerplane $G\in\A_i$ is visible from $q_i$ if and only if $q_i\notin T_{\sigma}(\sigma_i)$, the tangent cone of $\sigma_i$ at $\sigma$.  
Hence, $q$ is visible from $q_i$ if and only if $\langle a_G,q_i\rangle > b_G$.  
It follows that the point $q_i$ lies in the region $C^{(i)}$ of $\A_i$ consisting of all points $x\in H_i$ satisfying $\langle a_G,x\rangle > b_G$ for all $G\in \vis(q_i)$ and $\langle a_G,x\rangle < b_G$ for all $G\in\covis(q_i)$.   
Therefore, the closure of this region is all $x\in H_i$ satisfying $\langle a_G,x\rangle \geq b_G$ for all $G\in \vis(q_i)$ and $\langle a_G,x\rangle \leq b_G$ for all $G\in\covis(q_i)$.
From the inequality description of $C^{(i)}$ we know that the closure of $C^{(i)}$ contains the face $\sigma$ of $\sigma_i$ defined by $\langle a_G,x\rangle \leq b_G$ for all $G\in \vis(q_i)$. 
We claim that for all $1< i < s$, the face $\sigma$ is nonempty.

To see this, recall that we are assuming that the closure of $C^{(i)}$ contains a point $q$ of $\sigma_i$.  
Since $q$ is in the closure of $C^{(i)}$ then $\langle a_G,q\rangle \geq b_G$ for all $G\in\vis(q_i)$.  
On the other hand, since $q$ is a point in the convex polytope $\sigma_i$, which is defined as the set of solutions $x\in\R^d$ to the system of inequalities $\langle a_G, x\rangle \leq b_G$ for $G\in \A_i$, then $\langle a_G, q\rangle \leq b_G$ for all $G\in\vis(q_i)$.  
Hence, $\langle a_G, q\rangle = b_G$ for all $G\in\vis(q_i)$, and so the face $\sigma$ is nonempty. 

Now, if $1<i\leq t$, since the facets of $\sigma_i$ not included in the relative complex $\RR_i$ is precisely the set of visible facets from $q_i$, it follows that 
\[
\RR_i = \CC(\sigma_i)\setminus\CC(L([\sigma,\sigma_i]^\ast)).
\]
On the other hand, if $t< i< s$, then the set of facets of $\sigma_i$ not included in the relative complex $\RR_i$ is precisely the set of covisible facets $\covis(q_i)$, and so it follows that 
\[
\RR_i = \CC(\sigma_i)\setminus\CC(A(L(\sigma_i)^\ast)\setminus A([\sigma,\sigma_i]^\ast)).
\]
Hence, $\RR_i$ is stable for all $1< i <s$.  
Finally, when $i = 1$ or $i = s$, the point $q_i$, by definition of a line shelling, lies in the interior of $\sigma_i$.  
Hence, the set of visible facets from $q_i$ is empty.  
Therefore, the relative complex $\RR_1$ is simply the entire complex $\CC(\sigma_1)$, and the relative complex $\RR_s$ is $\CC(\sigma_s)\setminus\CC(\partial\sigma_s)$.  
Each of these arises as a reciprocal domain for the face $\emptyset$ of $\sigma_1$ and $\sigma_s$, respectively.  
Therefore, for all $i\in[s]$, the relative complex $\RR_i$ is stable, and we conclude that $(F_1,\ldots,F_s)$ is a stable line shelling of $\CC$. 
\end{proof}

\begin{figure}
\label{fig: stable line shellings}
\centering

\subfigure[]{
\begin{tikzpicture}[scale=0.22]

	\draw[fill=blue!25] (-4,0) -- (4,0) -- (2,2) -- (-1,2) -- cycle;

	
	\node [circle, draw, fill=black!100, inner sep=1pt, minimum width=1pt] (1) at (0,2) {};
	\node [circle, draw, fill=black!100, inner sep=1pt, minimum width=1pt] (2) at (3,1) {};
	\node [circle, draw, fill=black!100, inner sep=1pt, minimum width=1pt] (3) at (-2/3,20/9) {};
	\node [circle, draw, fill=black!100, inner sep=1pt, minimum width=1pt] (4) at (6,0) {};

 	 \draw   	 (-6,0) -- (8,0) ;
 	 \draw   	 (-7,-2) -- (5,6) ;
 	 \draw   	 (-4,2) -- (7,2) ;
	 \draw   	 (6,-2) -- (-2,6) ;
 	 \draw[<-]   	 (-6,4) -- (9,-1) ;
%

	\node (L1) at (-6.5,5) {\scriptsize$\ell_1$} ;
	\node (q1) at (3.5,1.5) {\tiny$q_1$} ;
	\node (q2) at (.2,1.4) {\tiny$q_2$} ;
	\node (q3) at (-2/3-1/2,27/9) {\tiny$q_3$} ;
	\node (q4) at (6.5,-0.75) {\tiny$q_4$} ;

\end{tikzpicture}
}
\subfigure[]{	
\begin{tikzpicture}[scale=0.22]

	\draw[fill=blue!25] (-4,0) -- (4,0) -- (2,2) -- (-1,2) -- cycle;

	
	\node [circle, draw, fill=black!100, inner sep=1pt, minimum width=1pt] (1) at (0,0) {};
	\node [circle, draw, fill=black!100, inner sep=1pt, minimum width=1pt] (1) at (1,2) {};
	\node [circle, draw, fill=black!100, inner sep=1pt, minimum width=1pt] (3) at (4/3,8/3) {};
	\node [circle, draw, fill=black!100, inner sep=1pt, minimum width=1pt] (4) at (2,4) {};

 	 \draw   	 (-6,0) -- (8,0) ;
 	 \draw   	 (-7,-2) -- (5,6) ;
 	 \draw   	 (-4,2) -- (7,2) ;
	 \draw   	 (6,-2) -- (-2,6) ;
 	 \draw[->]   (-2,-4) -- (3,6) ;
%

	\node (L2) at (3.5,6.5) {\scriptsize$\ell_2$} ;
	\node (q1) at (0.5,-0.75) {\tiny$q_1$} ;
	\node (q2) at (1.5,1.5) {\tiny$q_2$} ;
	\node (q3) at (4/3+1,8/3) {\tiny$q_3$} ;
	\node (q4) at (3,3.9) {\tiny$q_4$} ;

\end{tikzpicture}
}
\subfigure[]{
\begin{tikzpicture}[scale=0.22]

	\draw[fill=blue!25] (-4,0) -- (4,0) -- (2,2) -- (-1,2) -- cycle;

	
	\node [circle, draw, fill=black!100, inner sep=1pt, minimum width=1pt] (1) at (0,0) {};
	\node [circle, draw, fill=black!100, inner sep=1pt, minimum width=1pt] (1) at (3,1) {};
	\node [circle, draw, fill=black!100, inner sep=1pt, minimum width=1pt] (3) at (6,2) {};
	\node [circle, draw, fill=black!100, inner sep=1pt, minimum width=1pt] (4) at (-8,-8/3) {};

 	 \draw   	 (-6,0) -- (8,0) ;
 	 \draw   	 (-10,-4) -- (5,6) ;
 	 \draw   	 (-4,2) -- (7,2) ;
	 \draw   	 (6,-2) -- (-2,6) ;
 	 \draw[->]   	 (-11,-8/3-1) -- (8,8/3) ;
%

	\node (L2) at (7.9,8/3+1) {\scriptsize$\ell_3$} ;
	\node (q1) at (0.5,-0.75) {\tiny$q_1$} ;
	\node (q2) at (4,.75) {\tiny$q_2$} ;
	\node (q3) at (6,2.75) {\tiny$q_3$} ;
	\node (q4) at (-7.25,-8/3-.5) {\tiny$q_4$} ;

\end{tikzpicture}
}

\vspace{-0.2cm}
\caption{Some examples of (stable) line shellings.}
\end{figure}
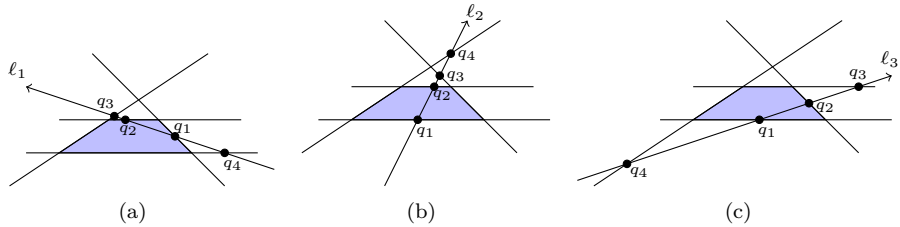
\begin{example}
\label{ex: stable line shellings}
In Figure~\ref{fig: stable line shellings}, we see three different lines $\ell_1,\ell_2,$ and $\ell_3$ in $\R^2$ that induce line shellings of the quadrilateral defined by the given arrangement of facet-defining hyperplanes. 
The blue quadrilateral is the polytope $Q_\Sigma$ with respect to the embedding $\Sigma$ of the boundary complex of the quadrilateral.  
The four unlabeled lines in each figure constitute the hyperplane arrangement $\A_\Sigma$.  
Each of these line shellings is stable because it satisfies the conditions of Proposition~\ref{prop: stable line shellings}.  
This can also be seen by Proposition~\ref{prop: stable shellings of simplicial complexes} and the fact that all $2$-dimensional polytopes are simplicial.

These small examples help us to see some of the intricacies of stable line shellings.  
The line shellings given by Figure~\ref{fig: stable line shellings}~(a) and~(c) are seen to be stable since each point $q_i$ for $i\in\{1,2,3,4\}$ lies in the closure of a region of the hyperplane arrangement $\A_\Sigma$ containing a face of $\sigma_i$ (the facet of $Q_\Sigma$ on the line in $\A_\Sigma$ containing the point $q_i$).  
Hence, each point $q_i$ must also lie in a region of $\A_i$ whose closure contains a point of $\sigma_i$, and so stability follows by Proposition~\ref{prop: stable line shellings}.  
On the other hand, this same reasoning does not apply when deducing that the line shelling given by Figure~\ref{fig: stable line shellings}~(b) is stable.  
Here, the point $q_4$ lies in the closure of a region of $\A_\Sigma$ that does not contain a point of $\sigma_4$.  
This is because the point $q_4$ is separated from $\sigma_4$ by the hyperplane that defines the facet $\sigma_3$.  
However, this line shelling is still stable by Proposition~\ref{prop: stable line shellings}, since we do not include the hyperplane $H_3\cap H_4$ in the arrangement $\A_4$.  

Figure~\ref{fig: stable line shellings}~(a) and~(c) show that there are line shellings that can be deduced to be stable by using the fact that each point $q_i$ is contained in the closure of a region of $\A_\Sigma$ that also contains a point of $F_i$.  
We can call such line shellings \emph{strongly stable}, since the shelling can be deduced to be stable via the global geometry of the arrangement $\A_\Sigma$, without restricting to the induced arrangements $\A_i$.  
In fact, reflecting on Example~\ref{ex: d-cube line}, we see that all line shellings of the boundary complex of the $d$-dimensional cube are strongly stable. 
\end{example}

In general, the authors hope that the boundary complex $\CC$ of every cubical polytope admits a stable line shelling induced by some realization $\Sigma\in R_\CC$ and line $\ell\in \R^d$ that is sufficiently close to $\Sigma$, in the sense of Proposition~\ref{prop: stable line shellings}.
\begin{conjecture}
\label{conj: stable line shelling}
The boundary complex of any cubical polytope admits a stable line shelling.
\end{conjecture}
It would also be of interest to know whether or not the boundary complex of every cubical or simplicial (or otherwise) polytope admits a strongly stable line shelling.

\bigskip

\noindent
{\bf Acknowledgements}. 
Max Hlavacek was supported by The Chancellor's Fellowship of the University of California, Berkeley.  
Liam Solus was supported by an NSF Mathematical Sciences Postdoctoral Research Fellowship (DMS - 1606407), a Starting Grant (No. 2019-05195) from Vetenskapsr\aa{}det, and the Wallenberg AI, Autonomous Systems and Software Program (WASP) funded by the Knut and Alice Wallenberg Foundation. 
The authors would like to thank Matthias Beck, Katharina Jochemko, and Joseph Doolittle for helpful discussions.  
This work began while the authors were attending the Workshop on \emph{Hyperbolic Polynomials and Hyperbolic Programming} at the Simons Institute during the semester program on \emph{Geometry of Polynomials}, it continued while authors were at the {\em 2019 Workshop on Commutative Algebra and Lattice Polytopes} at RIMS in Kyoto, Japan, and when both authors were visiting Freie Universit\"at Berlin. 
It was completed while the second author was participating in the Semester Program on \emph{Enumerative and Algebraic Combinatorics} at Institut Mittag-Leffler in Stockholm, Sweden, and while the first author was participating in the {\em Combinatorial Coworkspace} at Kleinwalsertal, Austria.  
The authors would like to thank the organizers of all of these events.

%
%
\end{document}